\documentclass[11pt]{article}

\usepackage{latexsym}
\usepackage{amssymb}
\usepackage{amsmath}
\usepackage{amsfonts}
\usepackage{amsthm}
\usepackage{stmaryrd}
\usepackage{hyperref}
\usepackage[all]{xy}

\numberwithin{equation}{subsection}
\newtheorem{theorem}[equation]{Theorem}
\newtheorem*{thm}{Theorem}

\newtheorem{conjecture}[equation]{Conjecture}
\newtheorem{corollary}[equation]{Corollary}
\newtheorem{lemma}[equation]{Lemma}
\newtheorem{proposition}[equation]{Proposition}

\theoremstyle{definition}
\newtheorem{caution}[equation]{Caution}

\newtheorem{construction}[equation]{Construction}
\newtheorem{convention}[equation]{Convention}

\newtheorem{definition}[equation]{Definition}

\newtheorem{example}[equation]{Example}

\newtheorem{notation}[equation]{Notation}
\newtheorem{question}[equation]{Question}
\newtheorem{remark}[equation]{Remark}

\newtheorem{hypo}[equation]{Hypothesis}

\def\bbf{\mathbf{f}}

\def\bbj{\mathbf{j}}
\def\bbk{\mathbf{k}}

\def\bbv{\mathbf{v}}

\def\bbP{\mathbf{P}}

\def\bbS{\mathbf{S}}
\def\bbT{\mathbf{T}}

\def\bbX{\mathbf{X}}

\def\AAA{\mathbb{A}}

\def\FF{\mathbb{F}}

\def\NN{\mathbb{N}}

\def\QQ{\mathbb{Q}}
\def\RR{\mathbb{R}}

\def\TT{\mathbb{T}}

\def\ZZ{\mathbb{Z}}

\def\calE{\mathcal{E}}
\def\calF{\mathcal{F}}

\def\calI{\mathcal{I}}

\def\calO{\mathcal{O}}
\def\calP{\mathcal{P}}

\def\calR{\mathcal{R}}

\def\gothb{\mathfrak{b}}

\def\gothd{\mathfrak{d}}

\def\gothm{\mathfrak{m}}

\def\gotho{\mathfrak{o}}

\def\gothA{\mathfrak{A}}
\def\gothB{\mathfrak{B}}

\def\gothD{\mathfrak{D}}

\def\gothX{\mathfrak{X}}

\renewcommand{\mod}{\mathrm{\,mod\,}}

\newcommand{\serie}[2]{{#11},\dots,{#1{#2}}}
\newcommand{\seriezero}[2]{{#10}, \dots, {#1{#2}}}
\newcommand{\vectzero}[2]{\left(\begin{array}{c}
{#10} \\ \vdots \\ {#1{#2}}
\end{array}\right)}

\newcommand{\Def}{\stackrel{\mathrm{def}}=}

\newcommand{\rar}{\rightarrow}

\newcommand{\lrar}{\longrightarrow}

\newcommand{\Rar}{\Rightarrow}

\newcommand{\inj}{\hookrightarrow}
\newcommand{\surj}{\twoheadrightarrow}
\newcommand{\isom}{\stackrel \sim \rar}
\newcommand{\map}[2]{\stackrel{#2}{#1}}
\newcommand{\bs}{\backslash}

\newcommand{\GL}{\mathrm{GL}}

\newcommand{\Sp}{\mathrm{Sp}}

\newcommand{\Mat}{\mathrm{Mat}}

\newcommand{\rank}{\mathrm{rank}\,}
\newcommand{\id}{\mathrm{id}}
\newcommand{\Ker}{\mathrm{Ker}\,}

\newcommand{\Fil}{\mathrm{Fil}}

\renewcommand{\log}{\mathrm{log}}

\newcommand{\alg}{\mathrm{alg}}
\newcommand{\sep}{\mathrm{sep}}

\newcommand{\pf}{\mathrm{pf}}

\newcommand{\Frac}{\mathrm{Frac}}
\newcommand{\Gal}{\mathrm{Gal}}

\newcommand{\unr}{\mathrm{unr}}
\newcommand{\Swan}{\mathrm{Swan}}
\newcommand{\Art}{\mathrm{Art}}

\newcommand{\Spec}{\mathrm{Spec}\,}

\newcommand{\pr}{\mathrm{pr}}

\newcommand{\geom}{\mathrm{geom}}

\newcommand{\OK}{{\calO_K}}

\newcommand{\Qp}{\QQ_p}
\newcommand{\Zp}{\ZZ_p}
\newcommand{\Fp}{\FF_p}

\newcommand{\Fq}{{\FF_q}}

\renewcommand{\sp}{\mathrm{sp}}

\newcommand{\Spf}{\mathrm{Spf}}
\newcommand{\Max}{\mathrm{Max}}

\newcommand{\dif}{\mathrm{dif}}

\newcommand{\CRP}{\mathtt{CRP}}
\newcommand{\PerfAlg}{\mathtt{PerfAlg}}

\newcommand{\inte}{\mathrm{int}}

\newcommand{\lead}{\mathrm{lead}}
\newcommand{\Lead}{\mathrm{Lead}}
\newcommand{\quot}{\mathrm{quot}}

\begin{document}

\title{On ramification filtrations and $p$-adic differential modules, I: equal characteristic case}

\author{Liang Xiao \\ Department of Mathematics, Room 2-090 \\ Massachusetts Institute of Technology \\ 77 Massachusetts Avenue \\ Cambridge, MA 02139 \\
\texttt{lxiao@math.mit.edu}}

\date{June 24, 2010}

\maketitle

\begin{abstract}
Let $k$ be a complete discretely valued field of equal characteristic $p > 0$ with possibly imperfect residue field, and let $G_k$ be its Galois group. We prove that the conductors computed by the arithmetic ramification filtrations on $G_k$ defined in \cite{AS-cond1} coincide with the differential Artin conductors and Swan conductors of Galois representations of $G_k$ defined in \cite{KSK-Swan1}. As a consequence, we obtain a Hasse-Arf theorem for arithmetic ramification filtrations in this case.  As applications, we obtain a Hasse-Arf theorem for finite flat group schemes; we also give a comparison theorem between the differential Artin conductors and Borger's conductors \cite{Borger-conductor}.
\end{abstract}

\tableofcontents

\setcounter{section}{-1}
\section{Introduction}

Let $k$ be a complete discretely valued field and let $G_k$ be the Galois group of a fixed separable closure $k^\sep$ over $k$.  When the residue field $\kappa_k$ of $k$ is perfect, classical ramification theory gives Artin conductors and Swan conductors, which measure the ramification of representations of $G_k$ of finite local monodromy (i.e., the image of the inertia group being finite).  A fundamental result, the Hasse-Arf theorem, states that Artin and Swan conductors are nonnegative integers.  However, when the residue field $\kappa_k$ is not perfect, the classical ramification theory is no longer applicable.  For one thing, the transition functions $\phi$ and $\psi$ in \cite[\S~IV.3]{BOOK-local-fields} fail the basic properties; for another, the extension of the rings of integers may not be generated by a single element (compare \cite[\S~III.6~Proposition~12]{BOOK-local-fields}).

In \cite{Kato-cond}, Kato defined Swan conductors for one-dimensional representations when the residue field is not perfect.  Later in \cite{AS-cond1, AS-cond2}, Abbes and Saito defined an arithmetic (non-logarithmic) filtration and a logarithmic variant on $G_k$ by counting geometric connected components of certain rigid spaces $as_{l/k}^a$ and $as_{l/k, \log}^a$ over $k$, which we refer to as Abbes-Saito spaces.  The filtrations give the arithmetic Artin conductors and Swan conductors naturally.

Abbes and Saito in \cite{AS-micro-local} showed that their definition of Swan conductors coincides with Kato's when $k$ is of equal characteristic $p>0$.  Moreover, they proved that the subquotients of both filtrations are abelian groups \cite{AS-cond2}. (See also \cite{Saito-wild-ram}, where Saito proved that the subquotients of the logarithmic filtration on wild inertia are elementary abelian $p$-groups.)  However, they were not able to establish certain integrality result that is analogous to the classical Hasse-Arf theorem.

Through another completely different path, when $k$ is of equal characteristic $p>0$ and has perfect residue field, Christol, Matsuda, Mebkhout, and Tsuzuki \cite{Mats-Swan-cond} gave a completely new interpretation of the classical Swan conductors using the theory of $p$-adic differential modules.  Given a $p$-adic Galois representation of finite local monodromy, they associated a $p$-adic differential module over the Robba ring and proved that the Swan conductor of the representation can be retrieved from the irregularity of the differential module, or equivalently, the spectral norms of the differential operator.

Partly inspired by Matsuda \cite{Mats-Conj-AS-fil}, Kedlaya generalized this framework to the case when the residue field $\kappa_k$ is not perfect.  In \cite{KSK-Swan1}, he adopted the same construction and counted in the effects of other differential operators corresponding to elements in a $p$-basis of $\kappa_k$.  He defined the differential Swan conductor to be, vaguely speaking, the maximum of the numbers computed by each of the differential operators, under certain normalization; he was aware of a definition for the differential Artin conductors using a slightly different normalization.  Most importantly, he was able to prove a Hasse-Arf theorem for the differential Swan conductors \cite[Theorem~3.5.8]{KSK-Swan1};  his proof can be easily adapted for differential Artin conductors (Theorem~\ref{T:properties-Ked-cond}(1)(4)). 

In \cite{KSK-Swan1}, Kedlaya asked, as Matsuda suggested, whether the differential conductors are the same as the arithmetic ones, in which case the Hasse-Arf theorem for the arithmetic filtrations in the equal characteristic case would follow from that for the differential conductors.  Chiarellotto and Pulita \cite{ChPu-cond} gave an affirmative answer to this question when the representations are one-dimensional, using the setting of Kato's conductors \cite{Kato-cond}.

There is a third story of defining conductors.  In \cite{Borger-conductor}, Borger introduced the notation of generic perfection of a complete discretely valued field and defined the Artin conductors to be the ones obtained by base changing to the generic residual perfection of $k$, which is a complete discretely valued field with perfect residue field satisfying certain universal properties.  The Hasse-Arf theorem of these conductors will follow immediately from that of the classical ones.  Kedlaya in \cite[P.297]{KSK-Swan1} asked if this also coincides with the two definitions above.

This paper assures these questions for all representations of finite local monodromy.  Our precise result is the following.

\begin{thm}
Let $k$ be a complete discretely valued field of equal characteristic $p>0$ and let $G_k$ be its absolute Galois group.
\begin{enumerate}
\item[1] (Hasse-Arf Theorem) Let $\rho: G_k \rar GL(V_\rho)$ be a $p$-adic representation of finite local monodromy.  Then the arithmetic Artin conductor $\Art_{ar}(\rho)$, the differential Artin conductor $\Art_\dif(\rho)$, and the Borger's conductor $\Art_B(\rho)$ are the same.  Similarly, the arithmetic Swan conductor $\Swan_{ar}(\rho)$ is the same as the differential Swan conductor $\Swan_\dif(\rho)$.  As a consequence, they are all nonnegative integers.
\item[2] The subquotients $\Fil^a G_k / \Fil^{a+} G_k$ of the arithmetic ramification filtrations are trivial if $a \notin \QQ$ and are elementary $p$-abelian groups if $a \in \QQ_{> 1}$; the subquotients $\Fil_\log^a G_k / \Fil_\log^{a+} G_k$ of the arithmetic logarithmic ramification filtrations are trivial if $a \notin \QQ$ and are elementary $p$-abelian groups if $a \in \QQ_{> 0}$.
\end{enumerate}
\end{thm}

This theorem consists of Theorems~\ref{T:main-theorem} and \ref{T:ar=diff=Borger} and Corollary~\ref{C:HA-thm-AS-fil}.
\vspace{10pt}

We now explain the main idea of the proof, which shows that the arithmetic conductors and the differential conductors coincide in a natural way.  (We will use the comparison of Artin conductors as an example; that of Swan conductors is proved similarly.)  

Let $k$ be a complete discretely valued field of equal characteristic $p$, with residue field $\kappa_k$.  Let $l$ be a finite Galois extension of $k$ with residue field $\kappa_l$. An elementary reduces the comparison to proving that the arithmetic highest ramification break of $l/k$ is the same as the differential one. There are three main ingredients.

(a) A useful way of visualizing spectral norms is to consider the convergence loci or radii at a generic point; see for example \cite[Section~5]{KSK-overview}.  However, the convergence loci cannot be defined on the rigid annulus because one cannot separate $m+1$ differential operators on a one-dimensional space.  Matsuda \cite{Mats-Conj-AS-fil} had a pioneering attempt to obtain an $(m+1)$-dimensional space on which we may discuss convergence loci.  Our approach, which is independently developed and looks very similar to Matsuda's work, uses a thickening technique. (Unfortunately, we do not know how to relate the two methods.)  If the field $k$ can be realized as the field of rational functions on a smooth variety over certain perfect field, the thickening space is just a subspace of the generic fiber of the tube corresponding to the diagonal embedding in a formal lifting (see Section~\ref{S:geometric-thickening}).  An observation is that this thickening space, after a certain base change, ``looks the same" as the Abbes-Saito space $as_{l/k}^a$, whose geometric connected components give the ramification information.  However, we have the following technical issue.

(b) The thickening space is a rigid space over $K$, the fraction field of a Cohen ring of $\kappa_k$, which in particular is a field of characteristic zero.  In contrast, the Abbes-Saito space $as_{l/k}^a$ is a rigid space over $k$, which is of characteristic $p$.  In order to ``identify" two spaces, we need a lifting technique (see Section 1) to lift the Abbes-Saito space to characteristic zero and compare the geometric connected components before and after the lifting process.  A similar idea is also alluded to as a conjecture in Matsuda's paper \cite{Mats-Conj-AS-fil}.  (Again, we do not know whether our result answers Matsuda's conjecture.)

(c) The lifted Abbes-Saito space is isomorphic to the thickening space after a certain base change (Theorem~\ref{T:AS=TS}), but not in the na\"ive way.  Very vaguely speaking, if the extension $l/k$ is generated by a series of equations, then the Abbes-Saito space consists of the points which are close to the solutions to those equations; in contrast, the (base change of the) thickening space consists of points which are solutions to some equations whose coefficients are close to the original equations.  These two types of points coincide when $l/k$ is totally and wildly ramified.

Combining these three ingredients, we can prove the comparison between the arithmetic conductors and the differential ones.  The following diagram may be helpful to illustrate the process.
\[
\xymatrix{
Y = A_L^1[\eta_0^{1/e}, 1) \ar[d] & TS^a \times_{\tilde \pi, Z} Y \ar[d] \ar@{~}[r]^-{(c)} \ar[l] & AS_{l/k}^a & as_{l/k}^a \ar@{~>}[l]_{(b)}\\
Z = A_K^1[\eta_0, 1) & TS^a = ``\cup_{\eta \in [\eta_0, 1)} A_K^1[\eta, 1) \times A_K^{m+1}[0, \eta^a)" \ar[l]^-{(a)}_-{\tilde \pi}
}
\]
Here, $K$ and $L$ are fraction fields of Cohen rings of $\kappa_k$ and $\kappa_l$, respectively; $A_K^1[\eta_0, 1)$ denotes the half-open annulus over $K$ (centered at the origin) with inner radius $\eta_0$ and outer radius $1$, for some $\eta_0 \in (0,1)$; $A_K^{m+1}[0, \eta^a)$ is the open polydisc (again centered at the origin) of dimension $m+1$ and radius $\eta^a$ for some $a \in \QQ_{>1}$; (for the quotation marks, see Caution~\ref{Cau:not-admissible}); $TS^a$ denotes the space obtained by the thickening process (a); $as_{l/k}^a$ is the rigid analytic space over $k$ defined by Abbes and Saito with respect to a set of distinguished generators, and $AS_{l/k}^a$ is the lifting space given by lifting process (b).  The argument in (c) links the two spaces as shown in the graph.

Part (a) is carried out throughout Section 3; see Theorem~\ref{T:ram-break-equiv-disconn}.  Part (b) is developed in Section 1; see Corollary~\ref{C:connected-components} and Example~\ref{Example:lifting-space}.
Part (c) occupies Section 4; see Theorem~\ref{T:AS=TS}.  We finally wrap up the proof in Theorem~\ref{T:main-theorem}.

We also obtain a comparison theorem between Borger's Artin conductors and the differential Artin conductors, or equivalently the arithmetic Artin conductors.  The key is to show that the differential Artin conductors are invariant under the operation of ``adding generic $p^\infty$-th roots" (see Definition~\ref{D:gen-pinfty-root}).  This fact follows easily from the study of differential operators.

\subsection*{Plan of the Paper}

In Section 1, we make a construction, which lifts a rigid space over $k$ to a rigid space over an annulus over $K$.  In particular, we prove that the connected components of the original rigid space are in one-to-one correspondence with the connected components of the lifting space, when the annulus is ``thin" enough.  This part is written in a relatively independent and self-contained manner, since we feel that it has its own interest.

In Section 2, we discuss how to associate a differential module $\calE_\rho$ on the Robba ring over $K$ to a representation $\rho$ of $G_k$ of finite local monodromy.  Then we review the definition of differential Swan conductors following \cite{KSK-Swan1}.  At the same time, we introduce differential Artin conductors and discuss their properties.

In Section 3, we introduce a thickening construction.  In Subsection~\ref{S:geometric-thickening}, as an intuitive example, we first construct the thickening space when $k$ can be realized geometrically.  In Subsection~\ref{S:thicken-annulus}, we define the thickening spaces for general $k$ and discuss spectral properties of the differential module obtained by pulling back $\calE_\rho$ to the thickening spaces.  In Subsections~\ref{S:construction-calE} and \ref{S:calF}, we link the (highest) differential breaks and spectral norms with the connected components of certain base change of the thickening spaces.

In Section 4, we first quickly review the definition of arithmetic ramification filtrations, following \cite{AS-cond1}.  Then, in Subsection~\ref{S:standard-AS}, we define the standard Abbes-Saito spaces $as_{l/k}^a$ and their lifts $AS_{l/k}^a$.  Next, we prove in Subsection~\ref{S4:comparison-rigid-spaces} that the lifted Abbes-Saito spaces and (the base change of) the thickening spaces are isomorphic (Theorem~\ref{T:AS=TS}).  From this, in Subsection~\ref{S4:Comparison-thm}, we deduce our main Theorem~\ref{T:main-theorem}: the differential conductors coincide with the arithmetic conductors.

In Section 5, we give two applications. One is to deduce a Hasse-Arf theorem for finite flat group schemes in Subsection~\ref{S:ffgs}; the other one is the comparison of the arithmetic and differential Artin conductors with the Borger's Artin conductors \cite{Borger-conductor}, which occupies the last three subsections.

\subsection*{Acknowledgments}
Many thanks are due to my advisor, Kiran Kedlaya, for introducing me to the problem, for generating ideas, for constant encouraging and supporting, and for spending many hours reviewing early drafts.  In particular, I would like to thank him for helping me fix a gap in an early version of this paper by providing the argument in Section 1.

The author owes many thanks to the referee of this paper; this paper is largely improved in presentation upon many invaluable suggestions by him/her.

Thanks to Ahmed Abbes and Takeshi Saito for organizing the great conference on vanishing cycles; it provided the author a great opportunity to learn about this field.
Thanks to Christopher Davis for helping me review early drafts, correct the grammar, and smooth the argument.  Thanks to Shin Hattori for helping me clarify Subsection~\ref{S:ffgs}.  Thanks to Shun Ohkubo for pointing out two errors in an early version of this paper and their link to the work of Sweedler.  Thanks to Brian Conrad for continuous encourage and support.   Thanks also to Michael Artin,  Johan de Jong, Ivan Fesenko, Xuhua He, Ruochuan Liu, Christian Kappen, Andrea Pulita, Yichao Tian and Wei Zhang for interesting discussions and helping review early drafts.

Financial support was provided by the MIT Department of Mathematics.  Also, when writing the paper, The author had Research Assistantship funded by Kedlaya's NSF CAREER grant DMS-0545904.

\section{Lifting rigid spaces}

In this section, we introduce a construction, which lifts a rigid space over a field of characteristic $p>0$ to a rigid space over an annulus over a field of characteristic zero.  This section is written in a relative independent and self-contained manner as it has its own interest.  The notation will not be carried over to other sections unless explicitly quoted.

Most of the content in this section should be credited to Kedlaya.  The author would like to thank him for allowing to include the proofs.

\begin{remark}
We make a pre-remark that for most of the places in the paper, we implicitly use rigid analytic spaces in the sense of Berkovich spaces \cite{Berkovich-book} by allowing discs or annuli with irrational radii.  This is mostly for notational convenience.  We will only encounter at two places (See Remarks~\ref{R:must-rigid-space} and \ref{R:Berkovich-vs-classical}) where we have to shift back to the classical rigid analytic setting to talk about connected components by assuming some rationality on the radii of discs or annuli.
\end{remark}

\subsection{A Gr\"obner basis argument}

In this subsection, we introduce a division algorithm using Gr\"obner basis, which enables us to find a representative in the quotient ring achieving the quotient norm.

\begin{notation}
Let $K$ be a complete discretely valued field of mixed characteristic $(0,p)$, with ring of integers $\calO_K$ and residue field $\kappa$.  Fix a uniformizer $\pi_K$ and normalize the valuation $v_K(\cdot)$ on $K$ so that $v_K(\pi_K) = 1$.  We also normalize the norm on $K$ so that $|p| = p^{-1}$.
\end{notation}

\begin{notation}
For a nonarchimedean ring $R$, we use $R \langle u_1, \dots, u_n \rangle$ to denote the Tate algebra, consisting of formal power series $\sum_{i_1, \dots, i_n \in \ZZ_{\geq 0}} f_{i_1, \dots, i_n}u_1^{i_1}\cdots u_n^{i_n}$ with $f_{i_1, \dots, i_n} \in R$ and $|f_{i_1, \dots, i_n}| \rightarrow 0$ as $i_1 + \cdots + i_n \rightarrow +\infty$.  For $\eta_1, \dots, \eta_n \in (0,1]$, the ring admits a $(\eta_1, \dots, \eta_n)$-\emph{Gauss norm} given by
\[
\Big|\sum_{i_1, \dots, i_n \in \ZZ_{\geq 0}} f_{i_1, \dots, i_n} u_1^{i_1} \cdots u_n^{i_n} \Big|_{\eta_1, \dots, \eta_n} = \max_{i_1, \dots, i_n} \big\{|f_{i_1, \dots, i_n}| \eta_1^{i_1} \cdots \eta_n^{i_n} \big\}.
\]

\end{notation}

\begin{notation}
Fix a positive integer $n$, and put
\begin{align*}
R^{\inte} &= \calO_K \langle u_1,\dots,u_n \rangle  ((S)), \\
R &= R^{\inte} \otimes_{\calO_K} K, \\
R_\kappa = R^{\inte} \otimes_{\calO_K} \kappa &\cong
\kappa [u_1,\dots,u_n] ((S)) = \kappa((S)) \langle u_1, \dots, u_n \rangle.
\end{align*}
For $\eta \in (0,1]$, let $|\cdot|_\eta$ (for short) denote the $(1,\dots,1,\eta)$-Gauss
norm on $R$. 
\end{notation}

\begin{notation}
The \emph{lexicographic} order on $\ZZ^n$ is that for $(i_1, \dots, i_n)$ and $(i'_1, \dots, i'_n) \in \ZZ^n$, we have $(i_1, \dots, i_n) \succ (i'_1, \dots, i'_n)$ if there exists some $j \in \{1, \dots, n\}$ such that $i_1 = i'_1, \dots, i_{j-1} = i'_{j-1}$ and $i_j > i'_j$.
\end{notation}

\begin{definition}  \label{D:Grobner-basis}
We equip $R_\kappa$ with the lexicographic term ordering induced by the correspondence $u_1^{i_1} \cdots u_n^{i_n} S^j 
\mapsto (-j, i_1,\dots,i_n)$, i.e., we write $\bar \alpha u_1^{i_1} \cdots u_n^{i_n} S^j \succeq \bar \beta u_1^{i'_1} \cdots u_n^{i'_n} S^{j'}$ if $(-j, i_1,\dots,i_n) \succeq (-j', i'_1,\dots,i'_n)$ under the lexicographic order, where $\bar \alpha, \bar \beta \in \kappa^\times$.

Using this ordering, we define the \emph{leading term} $\lead(\bar f)$ of a nonzero element $\bar f \in R_\kappa$ to be its largest term under the ordering.  In particular, for $\bar f, \bar g \in R_\kappa \bs \{0\}$, $\lead (\bar f \bar g) = \lead (\bar f) \lead (\bar g)$.

For an ideal $I_\kappa$ of $R_\kappa$, a \emph{Gr\"obner basis} of $I_\kappa$ is a finite subset
$\{\bar r_1,\dots, \bar r_m\} \subset I_\kappa$ such that $\lead(\bar r_i)$ do not have exponents on $S$ and the ideal consisting of the leading terms of all elements of $I_\kappa$ is generated by $\lead(\bar r_1),\dots, \lead(\bar r_m)$.   Such a basis exists because $R_\kappa$ is noetherian.  By \cite[Lemma~15.5]{Esbd-comm-alg}, $\serie{\bar r_}m$ also generate $I_\kappa$. 
\end{definition}

\begin{proposition} \label{P:grobner-basis-Rk}
For any $\bar f \in R_\kappa$, there exists $\bar g_1, \dots, \bar g_m, \bar f' \in R_\kappa$ such that 
\begin{equation} \label{E:grobner-algorithm}
\bar f = \bar g_1 \bar r_1 + \cdots +  \bar g_m \bar r_m + \bar f',
\end{equation}
where any term of $\bar f'$ is not divisible by any $\lead(\bar r_h)$, and $\lead(\bar f) \succeq \lead (\bar g_h \bar r_h)$ for all $h$.
\end{proposition}
\begin{proof}
Let $j$ be the exponent of $S$ in $\lead (\bar f)$ and let $S^j\bar f_{(j)}$ be the sum of terms in $\bar f$ for which the exponents of $S$ are $j$.  Applying \cite[Proposition-Definition~15.6]{Esbd-comm-alg} to $\bar f_{(j)}$, we can write
\[
\bar f_{(j)} \equiv \bar g_{1, (j)} \bar r_1 + \dots + \bar g_{m, (j)} \bar r_m + \bar f'_{(j)} \pmod {S \cdot \kappa[u_1, \dots, u_m]\llbracket S \rrbracket},
\]
where $\bar g_{h, (j)} \in \kappa[u_1, \dots, u_m]$ and $\lead (\bar g_{h, (j)} \bar r_h) \preceq \lead(\bar f_{(j)})$ for $h = 1, \dots, m$ and any term in $\bar f'_{(j)} \in \kappa[u_1, \dots, u_m]$ is not divisible by any $\lead(\bar r_h)$.

If we repeat the above argument for $\bar f_{(j)} - S^j \big( \bar g_{1, (j)} \bar r_1 + \dots + \bar g_{m, (j)} \bar r_m + \bar f'_{(j)} \big) \in S^{j+1}\cdot \kappa[u_1, \dots, u_m]\llbracket S \rrbracket$ in place of $\bar f$, we will obtain $\bar f'_{(j')}$ and $\bar g_{h, (j')}$ for $h = 1, \dots, m$ and for some $j' \geq j+1$.  We can then iterate this process.

For $h = 1, \dots, m$, put $\bar g_h = S^j\bar g_{h, (j)} + S^{j+1} \bar g_{h, (j+1)} + \cdots$ and $\bar f' = S^j \bar f'_{(j)} + S^{j+1} \bar f'_{(j+1)} + \cdots$; the power series converge to elements in $R_\kappa$ we are looking for.
\end{proof}

\begin{definition}  \label{D:grobner-basis}

For $f \in R$, write
\begin{equation} \label{E:expression-f}
f = \sum_{i_1,\dots,i_n,j} f_{i_1,\dots,i_n,j} u_1^{i_1}
\cdots u_n^{i_n} S^j.
\end{equation}
Of the monomials for which $|f_{i_1,\dots,i_n,j}| = |f|_1$, there must
be one which is lexicographically largest; we call the corresponding term 
$f_{i_1,\dots,i_n,j} u_1^{i_1} \cdots u_n^{i_n} S^j$ the 
\emph{$1$-leading term} of $f$, denoted by $\Lead(f)$.
\end{definition}

\begin{hypo}
Let $I^{\inte}$ be an ideal of $R^{\inte}$ such that $R^{\inte}/I^{\inte}$ is \emph{flat} over $\gotho_K$.
\end{hypo}

\begin{notation} \label{N:j_I-and-epsilon}
Denote $I = I^{\inte} \otimes_{\calO_K} K$ and $I_\kappa = I^{\inte} \otimes_{\calO_K} \kappa$; the latter is an ideal in $R_\kappa$ by the flatness hypothesis above.  Choose $r_1,\dots,r_m \in I^{\inte}$
which project to elements of a Gr\"obner basis $\bar r_1, \dots, \bar r_m$ of $I_\kappa$.

For $f \in R$, let $\bbj_f$ denote the minimal exponents of $S$ in the expression \eqref{E:expression-f} of $f$.  Denote $\bbj_I = \min \{ \bbj_{r_h}; h = 1, \dots, m \}$; it is a nonpositive integer.
\end{notation}

\begin{notation} \label{N:eta_0}
In this subsection, fix $\eta_0 \in (|\pi_K|^{-1/\bbj_I}, 1)$.  In particular, $|\pi_K| \eta_0^{\bbj_I} <1$.
\end{notation}

\begin{notation}
Let $\calR_{\eta_0}$ be the Fr\'echet completion of $R$ for $|\cdot|_\eta$ for $\eta \in [\eta_0,1)$.  Let $R^{\inte}_{\eta_0}$ denote $\{ f \in \calR_{\eta_0} | |f|_1 \leq 1 \}$ and put
$R_{\eta_0} = R^{\inte}_{\eta_0} \otimes_{\calO_K} K$ and $I_{\eta_0} = I \otimes_R R_{\eta_0}$.
\end{notation}

\begin{notation}
For an element $f \in \calR_{\eta_0}$ written as in \eqref{E:expression-f} and $l \in \ZZ$, let $\pi_K^lf_{(l)}$ be the sum of all terms $f_{i_1,\dots,i_n,j} u_1^{i_1} \cdots u_n^{i_n} S^j$ for which $v_K(f_{i_1,\dots,i_n,j}) = l$.  Thus, $f_{(l)} \in R^\inte_{\eta_0}$;  we use $\bar f_{(l)}$ denote its reduction in $R_\kappa$.
\end{notation}

\begin{lemma} \label{L:redundancy-of-r_h}
For $h = 1, \dots, m$ and $\eta \in [\eta_0, 1]$, 
\[
|r_h|_\eta = 1, \qquad |r_{h, (l)}|_\eta \leq \eta^{\bbj_I}, \textrm{ for } l \in \ZZ_{\geq 0}.
\]
\end{lemma}
\begin{proof}
The former inequality follows from the choice of $\eta_0$ in Notation \ref{N:eta_0}.  The latter follows from the definition of $\bbj_I$ in Notation~\ref{N:j_I-and-epsilon}.
\end{proof}

\begin{construction}
For $f \in R_{\eta_0}$ with $|f|_1 = |\pi_K|^{l_0}$, the \emph{division algorithm} is the following procedure.
Put $f_{l_0} = f$. Given $f_l$ for $l \geq l_0$, we apply Proposition~\ref{P:grobner-basis-Rk} to write
\[
\bar f_{l, (l)} = \bar g_{l, 1} \bar r_1 + \cdots +\bar g_{l, m} \bar r_m + \bar f'_{l, (l)},
\]
where $\bar g_{l,h} \in R_\kappa$ and $\lead(\bar g_{l,h} \bar r_h) \preceq \lead (\bar f_{l, (l)})$ for $h = 1, \dots, m$ and any term of $\bar f'_{l, (l)} \in R_\kappa$ is not divisible by any $\lead(\bar r_h)$.  For each $h$, pick lifts $g_{l, h}$ of $\bar g_{l,h}$ in $R^\inte$ so that $g_{l,h} = g_{l,h,(0)}$, namely, we only lift nonzero terms.  
Put
\[
f_{l+1} = f_l - \pi_K^l\big( g_{l,1} r_1 + \cdots + g_{l,m} r_m \big).
\]
\end{construction}

\begin{remark}
Division algorithm depends on many choices but we will prove in Proposition~\ref{P:uniqueness-division-alg} that the outcome $\lim_{l \rar +\infty} f_l$ is uniquely determined by $f$.
\end{remark}

\begin{lemma} \label{L:norm-goes-down}
At each step of the division algorithm, for $\eta \in [\eta_0, 1]$ and $h = 1, \dots, m$,
\begin{equation}
|g_{l,h}|_\eta \leq |f_{l, (l)}|_\eta, \qquad |f_{l+1, (l')} - f_{l, (l')}|_\eta \left\{
\begin{array}{ll}
\leq \eta^{\bbj_I} |f_{l, (l)}|_\eta & l' > l \\
\leq |f_{l, (l)}|_\eta & l' = l \\
= 0 & l' < l
\end{array}\right..
\end{equation}
\end{lemma}
\begin{proof}
The former inequality holds because $\lead(\bar g_{l,h} \bar r_h) \preceq \lead (\bar f_{l, (l)})$.  The latter relation follows from the former one, using Lemma~\ref{L:redundancy-of-r_h}.
\end{proof}

\begin{corollary} \label{C:converge}
For $h=1,\dots,m$, the series $g_h = \pi_K^{l_0}g_{l_0,h} + \pi_K^{l_0+1}g_{l_0+1,h} + \cdots$ converges under $|\cdot|_\eta$
for $\eta \in [\eta_0,1)$. Consequently, 
$g_h \in R_{\eta_0}$ for $h=1,\dots,m$.
\end{corollary}
\begin{proof}
By Lemma~\ref{L:norm-goes-down}, 
\begin{align*}
|\pi_K^l g_{l,h}|_\eta &\leq |\pi_K^l f_{l, (l)}|_\eta 
\leq |\pi_K|^l \max\big\{\eta^{\bbj_I}|f_{l-1, (l-1)}|_\eta, |f_{l-1, (l)}|_\eta \big\}\\
&\leq |\pi_K|^l \max\big\{ \eta^{2\bbj_I}|f_{l-2, (l-2)}|_\eta, \eta^{\bbj_I}|f_{l-2, (l-1)}|_\eta, \eta^{\bbj_I}|f_{l-2, (l-2)}|_\eta, |f_{l-2, (l)}|_\eta\big\} \leq \cdots \\
&\leq |\pi_K|^l \max_{l' < l} \big\{\eta^{(l-l')\bbj_I}|f_{(l')}|_\eta \big\}
\leq \max_{l'<l} \big\{\big( |\pi_K| \eta_0^{\bbj_I} \big)^{l-l'} \big|\pi_K^{l'}f_{(l')} \big|_\eta \big\}; 
\end{align*}
this goes to zero as $l \rar +\infty$.
\end{proof}

\begin{proposition} \label{P:uniqueness-division-alg}
Keep the notation as above.  The quantity $f - g_1 r_1 - \cdots - g_m r_m$ is the unique element
of $f + I_{\eta_0}$ for which none of its term is divisible by any $\Lead(r_h)$.
\end{proposition}
\begin{proof}
It follows from the definition of $g_1, \dots, g_m$ that none of the term of $f - g_1 r_1 - \cdots - g_m r_m$ is divisible by any $\Lead(r_h)$.

Assume that $f \in R_{\eta_0}$ does not contain any term divisible by any of $\Lead(r_h)$, then we need to show that for any nonzero $g \in I_{\eta_0}$, there is a term in $f+g$ divisible by some of $\Lead(r_h)$.  Assume the contrary.  Let $n = \log_{|\pi_K|} |g|_1$.  Then $\bar g_{(n)} \in I_\kappa$ does not contain any term which divides any of $\lead(\bar r_h)$.  This forces $\bar g_{(n)} = 0$  because the leading term of any nonzero element in $I_\kappa$ is divisible by some $\lead(\bar r_h)$.  Contradiction.  The lemma follows.
\end{proof}

\begin{lemma}
For $\eta \in [\eta_0,1]$,
$|f-g_1 r_1 - \cdots - g_m r_m|_\eta$ equals the minimum $\eta$-norm
of any element of $f + I_{\eta_0}$. Moreover, this continues to hold if
we pass from $R_{\eta_0}$ to its completion $R_{\eta_0}^{\wedge, \eta}$ under $|\cdot|_\eta$.
\end{lemma}
\begin{proof}
For $\eta \in [\eta_0, 1]$, by Lemma~\ref{L:norm-goes-down}, $|f_{l+1}|_\eta \leq |f_l|_\eta$ and hence $|f - g_1 r_1 - \cdots -g_m r_m|_\eta \leq |f|_\eta$.  By Proposition~\ref{P:uniqueness-division-alg}, starting with any element in $f + I_{\eta_0}$, the division algorithm will eventually lead to a unique element $f - g_1r_1 - \cdots - g_m r_m$; hence the first statement follows.

The second statement follows from the fact that any element in $f+I_{\eta_0} R_{\eta_0}^{\wedge, \eta}$ is a limit of elements in $f+ I_{\eta_0}$.
\end{proof}

\begin{proposition} \label{P:norm-small<=>inte}
Let $f$ be a rigid analytic function on the space
\[
\bbX_{\eta_0} = \big\{
(u_1,\dots,u_n,S) \in \AAA^{n+1}_K \big| \eta_0 \leq |S| < 1; |u_1|, \dots, |u_n|
\leq 1; r_1, \dots, r_m = 0 \big\}.
\]
Then the following are equivalent.
\begin{enumerate}
\item[(a)]
$f$ is induced by an element of $R^{\inte}_{\eta_0}$.
\item[(b)]
There exists a function $r: [\eta_0, 1) \rightarrow \RR$ with $\lim_{\eta \rightarrow 1^-} r(\eta) \leq 1$, such that for each $\eta \in [\eta_0, 1)$, $f$ lifts to an element of
the $|\cdot|_\eta$-completion of $R_{\eta_0}$ having $\eta$-norm less than 
or equal to $r(\eta)$.
\end{enumerate}
\end{proposition}
\begin{proof}
It is clear that (a) implies (b), so assume (b).
We can write $f$ as a Fr\'echet limit of the projections of
some sequence of elements $f_1,f_2, \dots$ of $R$,
under the quotient norms associated to the 
$|\cdot|_\eta$ for $\eta \in [\eta_0,1)$. 
Use the division algorithm to write $f_l = g_{l,1} r_1 + \cdots + g_{l,m} r_m
+ h_l$ with $g_{l,1},\dots,g_{l,m},h_l \in R_{\eta_0}$.   Moreover, as $f_l - f_{l+1}$ tends to zero under the Fr\'echet topology, so is $h_l - h_{l+1}$ since it can be obtained from the division algorithm of $f_l - f_{l+1}$ and Lemma~\ref{L:norm-goes-down} ensures that $|f_l - f_{l+1}|_\eta \geq |h_l - h_{l+1}|_\eta$.  Hence, the $h_l$ form a F\'echet convergent sequence; denote the limit by $h$, which is a lift of $f$.  Note that for a fixed $\eta$,
$|h_l|_\eta$ equals the $\eta$-quotient norm of $f_l$, which in turn equals the $\eta$-quotient norm of $f$ when $l$ is large enough.  Thus, $|h|_\eta \leq r(\eta)$ for all $\eta \in [\eta_0, 1)$. Hence it lies in $R_{\eta_0}^{\inte}$.
\end{proof}

\begin{notation}
Define 
\begin{eqnarray*}
A^{\inte} = R^{\inte}/I^{\inte} && A = R/I \\
A_{\eta_0} = R_{\eta_0} / I_{\eta_0} &&
A_\kappa = A^{\inte} \otimes_{\gotho_K} \kappa \cong R_\kappa / I_\kappa;
\end{eqnarray*}
we may view $A_\kappa$ as an affinoid algebra over $\kappa((S))$, whose corresponding rigid analytic space is denoted by $X$.
\end{notation}

\subsection{Quotient norms versus spectral norms}
\label{S:quotient-norm}
In this subsection, we compare spectral norms with the quotient norms discussed in previous section.  As an application, we deduce that the connected components of $\bbX_{\eta_0}$ when $\eta_0 \rar 1^-$ as a rigid space over $K$ are the same as the connected components of $X$ as a rigid space over $\kappa((S))$.

Keep the notation as above and assume the following.

\begin{hypo} \label{H:S-reduced}
In this subsection, we assume that $A_\kappa$ is reduced.
\end{hypo}

\begin{notation}
Let $|\cdot|_{\kappa, \quot}$ denote the quotient norm on $A_\kappa$ induced by the Gauss norm on $R_\kappa$.  Let $|\cdot|_{\kappa, \sp} = \lim_{n \rightarrow +\infty} |\cdot^n|_{\kappa, \quot}^{1/n}$ be the spectral norm; it is a norm because $A_\kappa$ is reduced.  By \cite[Theorem~6.2.4/1]{BGR}, there exists $c > 0$ such that $|\cdot|_{\kappa, \sp} \leq |\cdot|_{\kappa, \quot} \leq |S|_\kappa^{-c} |\cdot|_{\kappa, \sp}$, where $|S|_\kappa$ is the norm of $S$ in $\kappa((S))$.
\end{notation}

\begin{notation} \label{N:new-eta_0}
In this subsection, fix $\eta_0 \in \big(|\pi_K|^{1/(-\bbj_I + pc)}, 1 \big)$; in particular, $|\pi_K| \eta_0^{\bbj_I} < \eta_0^{pc}$ and $\eta_0 > p^{-1/pc}$.
\end{notation}

\begin{notation}
For $\eta \in [\eta_0, 1]$, let $|\cdot|_{\eta, \quot}$ denote the quotient norm on $A_{\eta_0}$ or $A$ induced by the $\eta$-Gauss norm on $R_{\eta_0}$ or $R$.  Similarly, we have the $\eta$-spectral (semi)norm $|\cdot|_{\eta, \sp} = \lim_{n\rightarrow +\infty} |\cdot^n|_{\eta, \quot}^{1/n}$; we will see in Lemma~\ref{L:compare-norm} that it is a norm.
\end{notation}

\begin{proposition} \label{P:A^circ=A^inte}
The quotient norm $|\cdot|_{1, \quot}$ on $A$ is the same as the spectral (semi)norm $|\cdot|_{1, \sp}$.  As a consequence, the map 
$A^{\inte} \to A_\kappa$ induces an isomorphism 
$A^\circ/A^{\circ\circ} \cong A_\kappa$, where $A^\circ = \{f \in A | |f|_{1, \sp} \leq 1 \}$ and $A^{\circ\circ} = \{f \in A | |f|_{1, \sp} < 1 \}$.
\end{proposition}
\begin{proof}
Since $A^{\inte}/\gothm_K A^{\inte} = A_\kappa$ is reduced, by \cite[6.2.1/4(iii)]{BGR}, the quotient norm on $A$ is equal to the spectral seminorm, $A^\circ = A^{\inte}$,
and $A^{\circ\circ} = \gothm_K A^{\inte}$. This proves the claim.
\end{proof}

\begin{lemma} \label{L:compare-norm}
For $\eta \in [\eta_0, 1)$, we have $|\cdot|_{\eta, \sp} \leq |\cdot|_{\eta, \quot} \leq \eta^{-pc/(p-1)} |\cdot|_{\eta, \sp}$ on $A_{\eta_0}$.  The same is true when extending both norm to the completion of $A_{\eta_0}$ with respect to $|\cdot|_{\eta, \quot}$ (which is the same as the completion with respect to the spectral norm).  In particular, this shows that $|\cdot|_{\eta, \sp}$ is a norm on $A_{\eta_0}$.
\end{lemma}
\begin{proof}
It suffices to show that for any $f \in A_{\eta_0}$, $|f^p|_{\eta, \quot} \geq \eta^{pc} |f|_{\eta, \quot}^p$; then it would follow that $|f^{p^n}|_{\eta, \quot} \geq \eta^{(p^n-1)pc/(p-1)} |f|_{\eta, \quot}^{p^n}$ for all $n \in \NN$ by iteration, and hence the statement follows by taking limit.

Pick a representative $\tilde f$ of $f$  in $R_{\eta_0}$ containing no terms divisible by any $\Lead(r_h)$ (hence by Proposition~\ref{P:uniqueness-division-alg}, $|\tilde f|_\eta = |f|_{\eta, \quot}$).  Fix $\eta \in [\eta_0, 1)$, we will show that
\begin{equation} \label{E:sp-norm-vs-quot-norm}
|\tilde f^p|_{\eta, \quot} = \Big|\sum_l (\pi_K^l\tilde f_{(l)})^p \Big|_{\eta, \quot} \geq  \eta^{pc} |\tilde f|_{\eta}^p = \eta^{pc} |f|_{\eta, \quot}^p.
\end{equation}
First, we remark that, given the middle inequality, the former equality follows; this is because $\tilde f^p - \sum_l (\pi_K^l \tilde f_{(l)})^p$ consists of products of $\pi_K^l\tilde f_{(l)}$ with an extra multiple $p$ from the multinomial coefficients and then $|\tilde f^p - \sum_l (\pi_K^l \tilde f_{(l)})^p|_{\eta, \quot} \leq |\tilde f^p - \sum_l (\pi_K^l \tilde f_{(l)})^p|_{\eta} \leq p^{-1} |\tilde f|_{\eta}^p < \eta^{pc} |\tilde f|_{\eta}^p$, for $\eta \in [\eta_0, 1)$.  So it suffices to prove the middle inequality in \eqref{E:sp-norm-vs-quot-norm}.  For any $l$, we have
\[
\big| (\bar{\tilde f}_{(l)})^p \big|_{\kappa, \quot} \geq \big| (\bar{\tilde f}_{(l)})^p \big|_{\kappa,\sp} = \big| \bar{\tilde f}_{(l)}\big|_{\kappa, \sp}^p \geq |S|_\kappa^{pc}\cdot \big| \bar{\tilde f}_{(l)}\big|_{\kappa, \quot}^p.
\]
Let $(\tilde f_{(l)})^p = g_{l, 1} r_1 + \cdots + g_{l, m} r_m + h_l$ be the result of the first step of applying the division algorithm to $(\tilde f_{(l)})^p$.  Then $ \log_\eta |h_{l, (0)}|_\eta = \log_{|S|_\kappa} \big| (\bar{\tilde f}_{(l)})^p \big|_{\kappa, \quot}$ and hence $|h_{l, (0)}|_\eta \geq \eta^{pc} |\tilde f_{(l)}|_\eta^p$.  Moreover, by Lemma~\ref{L:norm-goes-down}, $|h_l - h_{l, (0)}|_\eta \leq \eta^{\bbj_I}|\pi_K| |\tilde f_{(l)}|_\eta^p < \eta^{pc} |\pi_K|^{-pl}|\tilde f|_\eta^p$; this implies that $|h_{l, (0)}|_{\eta, \quot} = |h_{l, (0)}|_\eta$.

Now, we can write
\begin{equation} \label{E:sp-norm-vs-quot-norm2}
\sum_l (\pi_K^l\tilde f_{(l)})^p = \sum_l \pi_K^{pl}h_{l, (0)} + \sum_l \pi_K^{pl}(h_l - h_{l, (0)})
\end{equation}
in the quotient ring.
The former term on the right hand side of \eqref{E:sp-norm-vs-quot-norm2} has (quotient) norm at least $\eta^{pc} |\tilde f|_{\eta}^p$ because none of them is divisible by any $\Lead(r_h)$.  In contrast, the latter term on the right hand side of \eqref{E:sp-norm-vs-quot-norm2} has norm strictly less than $\eta^{pc} |\tilde f|_{\eta}^p$.  Thus, the inequality in \eqref{E:sp-norm-vs-quot-norm} holds.
\end{proof}

\begin{remark}
It is attractive to think that $|\cdot|_{\eta, \sp} \leq |\cdot|_{\eta, \quot} \leq \eta^{-c} |\cdot|_{\eta, \sp}$ when $\eta \rar 1^-$.  However, the best we know is that for any $c' > c$, we have an $\epsilon$ depending on $c'$, for which $|\cdot|_{\eta, \sp} \leq |\cdot|_{\eta, \quot} \leq \eta^{-c'} |\cdot|_{\eta, \sp}$ for all $\eta \in [\epsilon, 1)$.
\end{remark}

\begin{corollary} \label{C:sp-norm-small<=>inte}
For a rigid analytic function $f$ on $\bbX_{\eta_0}$, the following are equivalent.
\begin{enumerate}
\item[(a)]
$f$ is an element in $A^{\inte}_{\eta_0}$.
\item[(b)]
There exists a function $r: [\eta_0, 1) \rightarrow \RR$ with $\lim_{\eta \rightarrow 1} r(\eta) \leq 1$, such that for each $\eta \in [\eta_0, 1)$, $|f|_{\eta, \sp} \leq r(\eta)$.
\end{enumerate}
\end{corollary}
\begin{proof}
It follows from combining Lemma~\ref{L:compare-norm} with Proposition~\ref{P:norm-small<=>inte}.
\end{proof}

\begin{theorem}
There are one-to-one correspondences among the following four sets.
\begin{enumerate}
\item [(a)] the idempotent elements of $A_\kappa$;
\item [(b)] the idempotent elements of $A_{\eta_0}^{\inte}$;
\item [(c)] the idempotent elements of $A_{\eta_0}$;
\item [(d)] the idempotent elements on $\bbX_{\eta_0}$.
\end{enumerate}
\end{theorem}
\begin{proof}
By Corollary~\ref{C:sp-norm-small<=>inte}, the sets (b), (c), and (d) are the same because idempotent elements have spectral norms 1.  It suffices to match up (a) and (b).  We have a map from the set of idempotent elements of $A_{\eta_0}^\inte$ to the set of idempotent elements of $A_\kappa$ by reducing modulo $\pi_K$.  We first show the injectivity.  Let $f, g \in R_{\eta_0}^\inte$ be idempotents whose reductions modulo $\pi_K$ are the same, i.e., $\bar f = \bar g \in A_\kappa$.  This implies that $\bar f^{p-1} + \bar f^{p-2} \bar g + \dots + \bar g^{p-1} = 0$ in $A_\kappa$.  Since $f - g = f^p - g^p = (f-g) (f^{p-1} + f^{p-2} g + \cdots + g^{p-1})$, we have
\begin{align*}
|f-g|_{1, \quot} &= \big| (f-g)(f^{p-1} + f^{p-2} g + \cdots + g^{p-1}) \big|_{1, \quot} \\
&\leq |f-g|_{1, \quot} \cdot |f^{p-1} + f^{p-2} g + \cdots + g^{p-1}|_{1, \quot} \leq |f-g|_{1, \quot} \cdot |\pi_K|.
\end{align*}
This forces $|f-g|_{1, \quot} = 0$ and hence $f = g$.

To see the surjectivity, we start with an idempotent $\bar f \in A_\kappa$, viewed as an element in $R_\kappa$ with none of its terms divisible by any of $\Lead(\bar r_h)$; pick a lift $\tilde f_0 \in R^\inte$ of $\bar f$ which only contains the terms that $\bar f$ has and let $f_0 \in A^\inte$ denote its image in $A^\inte$.  If we set $\tilde h_0$ be the result of applying the division algorithm to $\tilde f_0^2 - \tilde f_0$ and $h_0 = f_0^2 - f_0$, then $|h_0|_{1, \quot} = |\tilde h_0|_{1, \quot} \leq |\pi_K|$ and $|h_0|_{\eta, \quot} = |\tilde h_0|_{\eta, \quot} \leq p^{-1}\eta^{-2c} < 1$ for all $\eta \in [\eta_0, 1)$, where the latter inequality holds because all terms in $\tilde f_0$ come from the terms in $\bar f$ which have norms $\leq |\bar f|_{\kappa, \quot} \leq |S|_\kappa^{-c} |\bar f|_{\kappa, \sp} = |S|_\kappa^{-c}$.  We apply a Hensel lemma type iteration to $f_0$ as follows.  For $\alpha \geq 0$, we set $f_{\alpha+1} = f_\alpha + h_\alpha - 2h_\alpha f_\alpha$ and
\[
h_{\alpha+1} : = f_{\alpha+1}^2 - f_{\alpha+1} = (f_\alpha + h_\alpha - 2h_\alpha f_\alpha)^2 - (f_\alpha+h_\alpha-2h_\alpha f_\alpha) = 4h_\alpha^2(h_\alpha - 1).
\]
Hence, $|h_{\alpha+1}|_{\eta, \quot} \leq |h_\alpha|_{\eta, \quot}^2$ for all $\eta \in [\eta_0, 1]$.  Thus $|h_\alpha|_{\eta, \quot} \rar 0$ as $\alpha \rar +\infty$; hence $f_\alpha$ converges to an element $f \in A_{\eta_0}^\inte$ which is idempotent.  It is clear from the construction that the reduction of $f$ modulo $\pi_K$ is the same as $\bar f$.  This proves the surjectivity.
\end{proof}

\begin{corollary}  \label{C:connected-components}
When $\eta_0 \in p^\QQ$, there is a one-to-one correspondence between the connected components of $X$ and the connected components of $\bbX_{\eta_0}$.
\end{corollary}

\begin{remark}  \label{R:must-rigid-space}
This is the first place where we need the rationality of $\log_p \eta_0$ to ensure that we are in the classical rigid analytic space setting to talk about connected components \cite[9.1.4/8]{BGR}.
\end{remark}

\subsection{Lifting construction}
In order to apply the results from previous two subsections later in the paper, we, reversing the picture, start with a rigid analytic space $X$ and try to construct $\bbX_{\eta_0}$ out from it.

Let $\kappa$ and $K$ be as before.

\begin{definition} \label{D:lifting-space}
Let $X$ be a \emph{reduced} affinoid rigid space over $\kappa((S))$ with ring of analytic functions $A_\kappa = R_\kappa / I_\kappa$ where $R_\kappa = \kappa ((S)) \langle u_1, \dots, u_n \rangle$ and $I_\kappa$ is some ideal.  The \emph{lifting construction} refers to the following.

(1) Find an ideal $I^\inte$ in $R^\inte = K \langle u_1, \dots, u_n \rangle ((S))$ so that $R^\inte / I^\inte$ is \emph{flat} over $\OK$ and $I^\inte \otimes_{\OK} \kappa = I_\kappa$.

(2) Choose a Gr\"obner basis of $I_\kappa$ and lift its elements to $r_1, \dots, r_m \in I^\inte$ as in Notation~\ref{N:j_I-and-epsilon} and define $\eta_0$ as in Notation~\ref{N:new-eta_0}.

(3) We call the rigid analytic space
\[
\bbX_{\eta_0} = \big\{
(u_1,\dots,u_n,S) \in \AAA^{n+1}_K \big|\; \eta_0 \leq |S| < 1; |u_1|, \dots, |u_n|
\leq 1; r_1, \dots, r_m = 0 \big\}
\]
the \emph{lifting space} of $X$; it depends only on the choice of $I^\inte$ and $\eta_0$.
\end{definition}

\begin{remark}
We do not know if such a lifting space always exists in general.  The only obstruction is to find an ideal $I^\inte$ lifting $I_\kappa$ such that $R^\inte / I^\inte$ is flat over $\calO_K$.
\end{remark}

\begin{question}
It would be interesting to know if the above lifting construction can be globalized for arbitrary rigid spaces over $\kappa((S))$.  In particular, given a morphism between two rigid spaces over $\kappa((S))$, can we lift the morphism (non-canonically) to a morphism between (some strict neighborhood of) their lifting spaces?  Can we ``glue" the lifting spaces up to homotopy?  This situation is very similar to Berthelot's construction of rigid cohomology \cite{Berthelot-rig-coh-part-I}.
\end{question}

For an affinoid subdomain of a polydisc, we explicate this lifting process.

\begin{example}  \label{Example:lifting-space}
Let $p_1, \dots, p_m \in \kappa \llbracket S\rrbracket [u_1, \dots, u_n]$ be polynomials and $\serie{a_}m \in \NN$.  We consider the following affinoid subdomain of the unit polydisc
\[
X = \big\{(u_1, \dots, u_n) \in \AAA_{\kappa((S))}^n \big|\; |u_1|, \dots, |u_n| \leq 1; |p_1| \leq |S|^{a_1}, \dots, |p_m| \leq |S|^{a_m}\big\}.
\]
The ring of analytic functions on $X$ is
\[
\kappa((S)) \langle u_1, \dots, u_n, v_1, \dots, v_m \rangle \big/ (v_1 S^{a_1} - p_1, \dots, v_m S^{a_m} - p_m).
\]

For each $i$, let $P_i$ be a lift of $p_i$ in $\calO_K \llbracket S \rrbracket [u_1, \dots, u_n]$ (here we allow $P_i$ to have new terms other than the terms of $p_i$).  We claim that the ring
\begin{equation} \label{E:ring-lifting-space}
\OK \langle u_1, \dots, u_n, v_1, \dots, v_m \rangle ((S)) \big/ (v_1 S^{a_1} - P_1, \dots, v_m S^{a_m} - P_m)
\end{equation}
is flat over $\calO_K$.  This is because the ring
\[
\OK ((S)) [u_1, \dots, u_n, v_1, \dots, v_m]  \big/ (v_1 S^{a_1} - P_1, \dots, v_m S^{a_m} - P_m) \simeq \OK ((S)) [u_1, \dots, u_n]
\]
is flat and hence torsion free over $\calO_K$, and so is its completion \eqref{E:ring-lifting-space} with respect to the topology induced by $(p, S)^r \calO_K \llbracket S \rrbracket [u_1, \dots, u_n, v_1, \dots, v_m]$ for $r \in \NN$.

Therefore, by Definition~\ref{D:lifting-space},
\[
\bbX_{\eta_0} = \big\{ (u_1, \dots, u_n, S) \in \AAA_K^{n+1} \big| \; \eta_0 \leq |S| <1; |u_1|, \dots, |u_n| \leq 1; |P_1| \leq |S|^{a_1}, \dots, |P_m| \leq |S|^{a_m}\big\}
\]
is a lifting space for $X$, for some $\eta_0 \in (0,1)$.
\end{example}

\section{Differential conductors}

In this section, we recall the definition of differential Swan conductors following \cite{KSK-Swan1}.  Along the way, we define the differential Artin conductors using a slightly different normalization.

As a reminder, we do not use any notation from the previous section.

\subsection{Setup}

\begin{convention}
Let $J$ be an index set.  We write $e_J$ for a tuple $(e_j)_{j \in J}$.  For an element $x$, we use $x^{e_J}$ to denote $(x^{e_j})_{j \in J}$.  For another tuple $b_J$, we denote $b_J^{e_J} = \prod_{j \in J} b_j^{e_j}$ if only finitely many $e_j \neq 0$. We also use $\sum_{e_J = 0}^{n}$ to mean the sum over $e_j \in \{0, 1, \dots, n\}$ for each $j \in J$, only allowing finitely many of them to be non-zero.
\end{convention}

\begin{definition}
For a finite field extension $l/k$ of characteristic $p>0$, a \textit{$p$-basis} of $l$ over $k$ is a set $(c_j)_{j \in J} \subset l$ such that $c_J^{e_J}$, where $e_j \in \{0, 1, \dots, p-1\}$ for all $j \in J$ and $e_j = 0$ for all but finitely many $j$, form a basis of the vector space $l$ over $kl^p$. By a \textit{$p$-basis} of $l$ we mean a $p$-basis of $l$ over $l^p$; it is an empty set if and only if $l$ is perfect. (For more details, see \cite[P.565]{Esbd-comm-alg} or EGA\cite[Ch.0~\S21]{EGA-IV-1}.)
\end{definition}

\begin{remark}\label{R:p-basis-expression}
For a $p$-basis $c_J \subset l$, $dc_J$ form a basis for the differentials $\Omega^1_l$ as an $l$-vector space.
\end{remark}

\begin{convention}
Throughout this paper, all differentials are $p$-adically continuous.  In other words, for a continuous homomorphism $A \rar B$ of $p$-adic rings, $\Omega_{B/A}^1$ is the relative $p$-adically continuous differentials.  Sometimes, we may use the corresponding geometric objects (e.g. rigid space $\Max(B)$) instead of $A$ or $B$ in the notation. When $A = \ZZ_p$, we may suppress it from the notation (by simply writing $\Omega^1_B$).

For a homomorphism $A \rar B$ between rings, a \textit{$\nabla$-module} or a \emph{differential module} over $B$ \emph{relative to} $A$ is a finite projective $B$-module $M$ equipped with an integrable connection $\nabla: M \rar M \otimes \Omega^1_{B/A}$.  Sometimes, we may use the corresponding geometric objects instead of $A$ or $B$ in the notation.  When $A = \ZZ_p$, we may omit the reference to the base ring.
\end{convention}

\begin{notation}
Let $k$ be a complete discretely valued field of equal characteristic $p>0$.  Denote its ring of integers, maximal ideal, and residue field by $\calO_k$, $\gothm_k$, and $\kappa_k$, respectively.  Fix a uniformizer $s$ and a non-canonical isomorphism
\begin{equation} \label{E:k((s))=k}
\kappa_k ((s)) \simeq k.
\end{equation}
Let $v_k(\cdot)$ denote the valuation, normalized so that $v_k(s) = 1$.  Let $(\bar b_j)_{j \in J}$ be a $p$-basis of $\kappa_k$, where $J$ is an index set.  Let $b_j$ be the image of $\bar b_j$ in $k$ under the isomorphism \eqref{E:k((s))=k}.  Hence, $(db_j)_{j \in J}$ and $ds$ form a basis of $\Omega^1_{\calO_k/\FF_p}$.  We set $\kappa_0= \cap_{n>0} k^{p^n} \cong \cap_{n>0} \kappa_k^{p^n}$; it is a perfect field.
\end{notation}

\begin{notation}
Let $\calO_K$ denote the Cohen ring of $\kappa_k$ with respect to $(\bar b_j)_{j \in J}$ and let $(B_j)_{j\in J} \subset \OK$ be the canonical lifts of the $p$-basis.  (For more about Cohen rings, see \cite[Section~3.1]{KSK-Swan1} or \cite{Whitney-Cohen-rings}.)  Denote $K = \Frac \calO_K$.  We use $\calO_{K_0}$ to denote the ring of Witt vectors $W(\kappa_0)$ of $\kappa_0$, as a subring of $\calO_K$.  Denote $K_0 = \Frac \calO_{K_0}$.
\end{notation}

We insert here a proposition discussing the functoriality of Cohen rings.  For more detailed study of functoriality of Cohen rings, one may consult \cite{Whitney-Cohen-rings}.

\begin{proposition} \label{P:cohen-functorial}
Keep the notation as above and let $R$ be a complete noetherian local ring with the maximal ideal $\gothm$ containing $p$. Assume that we have a homomorphism $\bar \psi: \kappa_k \inj R/\gothm$. Then, for any $B'_J \subseteq R$ lifting $\bar \psi(\bar b_J)$, there exists a unique continuous homomorphism $\psi: \OK \rar R$ lifting $\bar \psi$ and sending $B_j$ to $B'_j$ for all $j \in J$.
\end{proposition}

\begin{proof}
For any $n \in \NN$, a \emph{level $n$ expression} of an element $g \in \OK$ is a (non-canonical) way of writing $g$ as
\begin{equation}\label{E:cohen-functorial}
g = \sum_{i, i' \geq 0} \sum_{e_J = 0}^{p^n-1} p^i A^{p^n}_{i, i',e_J} B_J^{e_J}
\end{equation}
for some $A_{i, i',e_J} \in \calO_K$ and for a fixed $i$, $A_{i, i', e_J} = 0$ when $i' \gg 0$ for all $e_J$. Then we set
\[
\psi_n(g) = \sum_{i, i' \geq 0} \sum_{e_J = 0}^{p^n-1} p^i \tilde A^{p^n}_{i,i', e_J} B'^{e_J}_J
\]
where $\tilde A_{i,i', e_J}$ is some lift of $\bar \psi (a_{i,i', e_J})$ in $R$ with $a_{i, i',e_J}$ being the reduction of $A_{i,i', e_J}$ in $\kappa_k$.  Different choices of lifts $\tilde A_{i,i', e_J}$ may change the definition of $\psi_n(g)$ by an element in $\gothm^n$; a different level $n$ expression as in \eqref{E:cohen-functorial} may also vary $\psi_n(g)$ by some element in $\gothm^n$.  For a level $n$ expression of $g$ as in \eqref{E:cohen-functorial} with $n \geq 1$, we can rewrite it as
\[
g = \sum_{i, i' \geq 0} \sum_{e'_J = 0}^{p-1} \sum_{e_J = 0}^{p^{n-1}-1} p^i \big(A^p_{i, i',e_J+p^{n-1}e'_J} B_J^{e'_J})^{p^{n-1}} B_J^{e_J},
\]
which is a level $n-1$ expression for $g$.  From this, we conclude that $\psi_n(g) \equiv \psi_{n-1}(g) \mod \gothm^{n-1}$. Taking $n \rar \infty$, we get our map $\psi(g) = \lim_{n \rar \infty} \psi_n(g)$. It is not hard to check that $\psi$ is actually a homomorphism; this is because for $g, h \in \OK$, the formal sum and product of level $n$ expressions of $g$ and $h$ are level $n$ expressions of $g+h$ and $gh$, respectively.

To prove the uniqueness, take another continuous homomorphism $\psi': \OK \rar R$ satisfying all the conditions.  Then for a level $n$ expression of $g$ as in \eqref{E:cohen-functorial}, 
\[
\psi'\Big( \sum_{i, i' \geq 0} \sum_{e_J = 0}^{p^n-1} p^i A^{p^n}_{i, i',e_J} B_J^{e_J} \Big) = \sum_{i, i' \geq 0} \sum_{e_J = 0}^{p^n-1} p^i \psi'(A_{i, i',e_J})^{p^n} B'^{e_J}_J
\]
is exactly one possible definition for $\psi_n$. As we proved above, $\psi'(g) \equiv \psi_n(g) \equiv \psi(g) \mod \gothm^n$.  Let $n \rar \infty$ and we have $\psi = \psi'$.
\end{proof}

\begin{corollary} \label{C:cohen-deformation}
Keep the notation as above and assume that $J = \{1, \dots, m\}$ is a finite set.  There exists a unique continuous homomorphism $\psi: \OK \rar \OK \llbracket \serie{\delta_}m \rrbracket$ such that for all $j \in J$, $\psi(B_j) = B_j + \delta_j$ and for any $g \in \OK$, $\psi(g) - g$ lies in the ideal generated by $\serie{\delta_}m$.  Moreover, the $\psi$ is an $\calO_{K_0}$-homomorphism.
\end{corollary}
\begin{proof}
The first statement follows from previous proposition.  By the functoriality of Witt vectors, $\psi$ has to be identity when restricted to $\calO_{K_0}$ because $\kappa_0$ is perfect.  Hence, $\psi$ is an $\calO_{K_0}$-homomorphism.
\end{proof}

\begin{corollary} \label{C:k=k((s))-change}
Assume that $\kappa_k$ has a finite $p$-basis $b_J$.
Fix $j \in J$ and let $b'_j \in \calO_k$ be an element such that $b'_j \equiv b_j \pmod {\gothm_k}$.  Then there exists an automorphism $g^*: k \rar k$ such that $g^*(s) = s$, $g^*(b_j) = b'_j$, and $g^*(b_{J \bs j}) = b_{J \bs j}$.
\end{corollary}
\begin{proof}
Applying Proposition~\ref{P:cohen-functorial} to $R = \kappa_k \llbracket s \rrbracket$ and $\gothm = (s)$ gives us a homomorphism $g^*: \OK / (p) = \kappa_k \rar k \llbracket s \rrbracket$ such that $g^*(b_j) = b'_j$, and $g^*(b_{J \bs j}) = b_{J \bs j}$.  One can extend this to an automorphism $g^*: k \rar k$ by setting $g^*(s) = s$.
\end{proof}

\subsection{Construction of differential modules} \label{S2-construction-diff-eqn}

In this subsection we review Tsuzuki's construction \cite{Tsuzuki-finite-monodromy-on-var} of differential modules over the Robba ring associated to $p$-adic Galois representations. For a systematic treatment, one may consult, for example, \cite[Section~3]{KSK-Swan1}.

\begin{notation}
Keep the notation as in the previous subsection.  Fix a separable closure $k^\sep$ of $k$ and let $G_k = \Gal(k^\sep/k)$ be the absolute Galois group of $k$.

For a (not necessarily algebraic) separable extension $l/k$ of complete discretely valued fields, the \emph{na\"ive ramification degree} $e$ is the index of the valuation group of $k$ in that of $l$; note that this might not be the same as the usual ramification degree because the inseparable part of the residue field extension $\kappa_l / \kappa_k$ is not counted in.  We say $l/k$ is \emph{tamely ramified} if $p \nmid e$ and the residue field extension is algebraic and separable.  Moreover, if $e=1$, we say that $l/k$ is \emph{unramified}.
\end{notation}

\begin{notation} \label{N:representation}
By a \emph{representation} of $G_k$, we mean a continuous homomorphism $\rho: G_k \rar \GL(V_\rho)$, where $V_\rho$ is a vector space over a (topological) field $F$ of characteristic zero.  We say that $\rho$ is a \emph{$p$-adic representation} if $F$ is a finite extension of $\Qp$.

Let $F$ be a finite extension of $\Qp$.  Let $\calO$ and $\Fq$ denote its ring of integers and residue field, respectively, where $q$ is a power of $p$. Write $\ZZ_q$ for the Witt vectors $W(\Fq)$ and $\QQ_q$ for its fraction field.  By an \emph{$\calO$-representation} of $G_k$, we mean a continuous homomorphism $\rho: G_k \rar \GL(V_\rho)$ with $V_\rho$ a finite free $\calO$-module.

We always assume that $\Fq \subseteq \kappa_0$ (see the Remark~\ref{R:k_0-alg-closed}).  Denote $K' = KF$.  Since $F/\QQ_q$ is totally ramified, we have the ring of integers $\calO_{K'} \cong \calO_K \otimes_{\ZZ_q} \calO$.  Let $v_{K'}$ denote the valuation on $K'$ normalized so that $v_{K'}(p) = 1$.
\end{notation}

\begin{notation}\label{N:cohen-ring-k}
Let $C_k$ be the Cohen ring of $k$ with respect to the $p$-basis $\{(b_j)_{j \in J}, s\}$.  By functoriality of Cohen ring (Proposition~\ref{P:cohen-functorial}), $C_k$ has a natural structure of $\OK$-algebra via the isomorphism \eqref{E:k((s))=k}.  In particular, the (canonical) lifts of $(b_j)_{j \in J}$ in $C_k$ are $(B_j)_{j \in J}$.  We denote the canonical lift of $s$ in $C_k$ by $S$.

Put $\Gamma = C_k \otimes_{\ZZ_q} \calO$; it is a complete discrete valuation ring since $\calO$ is totally ramified over $\ZZ_q$.  It carries a Frobenius structure $\phi$ lifting the $q$th-power Frobenius on $k$ which acts trivially on $\calO$.
\end{notation}

\begin{definition}
If $R$ is equipped with an endomorphism $\sigma: R \rar R$, a \textit{$(\sigma, \nabla)$-module} over $R$ is a $\nabla$-module over $R$ (relative to $\ZZ_p$) equipped with an isomorphism $\sigma^* M \rar M$ of $\nabla$-modules.
\end{definition}

\begin{definition}
For every $\calO$-representation $\rho : G_k \rar \GL(V_\rho)$, define its associated $(\phi, \nabla)$-module over $\Gamma$ by
$$
D(\rho) = \big( V_\rho \otimes_\calO \widehat{\Gamma^\unr} \big)^{G_k},
$$
where $\widehat{\Gamma^\unr}$ is the $p$-adic completion of the maximal unramified extension of $\Gamma$.  All $\nabla$-modules we encounter in this section are relative to $\Zp$, so we omit the reference to the base ring $\Zp$ in the notation.
\end{definition}

\begin{proposition}
For any Frobenius lift $\phi$ on $\Gamma$, the functor $D$ from $\calO$-representations of $G_k$ to $(\phi, \nabla)$-modules over $\Gamma$ is an equivalence of categories.
\end{proposition}

\begin{proof}
For convenience of the reader, we briefly describe the functor here; for more details, one may consult \cite[Propositions 3.2.7 and 3.2.8]{KSK-Swan1}.  It is well-known that $D$ establishes an equivalence between the category of representations and the category of $\phi$-modules over $\Gamma$ (finite free $\Gamma$-modules with semi-linear $\phi$-actions), with $V(M) = \big(M \otimes_\Gamma \widehat{\Gamma^\unr} \big)^{\phi = 1}$ as the inverse. The nontrivial part is that every $\phi$-module over $\Gamma$ admits a unique structure of $(\phi, \nabla)$-module; this involves a standard approximation argument.
\end{proof}

\begin{definition}
Let $I_k = \Gal(k^\sep / k^\unr)$ be the inertia subgroup of $G_k$, where $k^\unr$ is the maximal unramified extension of $k$ in $k^\sep$. We say that an ($\calO$-)representation $\rho$ has \textit{finite local monodromy} if the image $\rho(I_k)$ is finite.
\end{definition}

For an $\calO$-representation $\rho$ of finite monodromy, one can refine the $(\phi, \nabla)$-module associated to $\rho$ as follows.

\begin{construction} \label{Cstr:repn->bounded-module}
Since $C_k$ has an $\OK$-algebra structure, any element $x \in \Gamma$ can be uniquely written in the form of $\sum_{i \in \ZZ} x_i S^i$ for $x_i \in \OK \otimes _{\ZZ_q} \calO = \calO_{K'}$ such that the indices $i$ for which $v_{K'}(x_i) \leq n$ are bounded below.

For $r > 0$, put $\Gamma^r = \{ x \in \Gamma | \displaystyle\lim_{n \rar -\infty} v_{K'}(x_n) + rn = \infty\}$ and $\Gamma^\dag = \cup_{r > 0} \Gamma^r$; the latter is commonly known as the \emph{integral Robba ring} over $K'$. It is not hard to show that the Frobenius $\phi$ preserves $\Gamma^\dag$ and that $\Omega^1_{\Gamma^\dag / \calO} = \bigoplus_{j \in J} \Gamma^\dag dB_j \oplus \Gamma^\dag dS$.  Also, $\Gamma^\dagger$ is Henselian discrete valuation ring as cited in Lemma~\ref{L:bounded-Robba-henselian}.

Since $\calO_{K'} \inj \Gamma^\dag$, we can identify $\calO_{K'}^\unr \inj (\Gamma^\dag)^\unr$, where the superscript unr means taking the maximal unramified extensions of discrete valuation rings. Put $\widetilde \Gamma^\dag = \widehat{\calO_K^\unr} \otimes_{\calO_K^\unr} (\Gamma^\dag)^\unr \subset \widehat{\Gamma^\unr}$, where we take the $p$-adic completion.  For a $p$-adic representation $\rho$ with finite local monodromy, define 
\begin{equation}\label{E:def-of-D-dag-rho}
D^\dag(\rho) = D(\rho) \cap \big( V_\rho \otimes_\calO \widetilde \Gamma^\dag \big) = \big( V_\rho \otimes_\calO \widetilde \Gamma^\dag \big) ^ {G_k}.
\end{equation}
\end{construction}

\begin{lemma}
\emph{\cite[Proposition~3.20]{KSK-overview}} \label{L:bounded-Robba-henselian}
The integral Robba ring $\Gamma^\dag$ is a henselian discrete valuation ring.
\end{lemma}

\begin{theorem} \label{T:equiv-diff-mod-repn}
\emph{\cite[Theorem~3.3.6]{KSK-Swan1}}
Let $\phi$ be a lift of Frobenius to $\Gamma$ acting on $\Gamma^\dag$. Then $D^\dag$ induces an equivalence between the category of $\calO$-representations with finite local monodromy and the category of $(\phi, \nabla)$-modules over $\Gamma^\dag$.
\end{theorem}

\begin{notation}\label{N:affinoids}
For $I \subset [0, +\infty)$ an interval, let $A_K^1(I)$ denote
the annulus (centered at the origin) with radii in $I$.
(We do not impose any rationality condition on the endpoints of $I$, so this
space should be viewed as an analytic space in the sense of 
Berkovich \cite{Berkovich-book}.)
If $I$ is written explicitly in terms of its
endpoints (e.g., $[\alpha, \beta]$),
we suppress the parentheses around $I$ (e.g., 
$A_K^1[\alpha, \beta]$).

For $0 \leq \alpha \leq \beta < \infty$, let 
$K \langle \alpha/t, t/\beta \rangle$ denote the ring of analytic functions on $A_K^1[\alpha, \beta]$. (If $\alpha = 0$, we write $K \langle t/\beta \rangle$ instead.)  For $\eta \in [\alpha, \beta] \bs \{0\}$, $K \langle \alpha/t, t/\beta \rangle$ admits an $\eta$-Gauss norm: for $f = \sum_{i \in \ZZ} a_i x^i \in K \langle \alpha / t, t/\beta \rangle$,
\[
|f|_\eta = \max_{i \in \ZZ} \{ |a_i| \eta^i\}.
\]
\end{notation}

\begin{notation}\label{N:Robba-ring-radius-eta}
For $\eta_0 \in (0,1)$, we use $Z_k^{\geq \eta_0}$ for short to denote $A_K^1[\eta_0, 1)$.  Denote the ring of analytic functions on it by $\calR_K^{\eta_0}$.  We define the \emph{Robba ring over $K$} to be $\calR_K = \cup_{\eta \in [\eta_0, 1)} \calR_K^\eta$.  Also denote $\calR_{K'}^{\eta_0} = \calR_K^{\eta_0} \otimes_{\QQ_q} F$ and $\calR_{K'} = \calR_{K} \otimes_{\QQ_q} F$.  We will be only interested in the behavior when $\eta_0$ is close to $1$.  
\end{notation}

\begin{remark}
We use $k$ in the subscript of $Z_k^{\geq \eta_0}$ because the space is functorially in $k$ but not in $K$, as we made a non-canonical choice in \eqref{E:k((s))=k}.
\end{remark}

Now, we restrict the $(\phi, \nabla)$-module $D^\dag(\rho)$ to the Robba ring over $K$ as follows.

\begin{construction} \label{Cstr:bounded->R_K}
Consider the natural injection $\Gamma^\dag \inj \calR_{K'}$.  Note that the Frobenius $\phi$ extends by continuity to $\calR_{K'}$. Thus, from an $\calO$-representation $\rho$ with finite local monodromy, we obtain a differential module $\calE_\rho = D^\dag(\rho) \otimes _{\Gamma^\dag} \calR_{K'}$ over $\calR_{K'}$.

Moreover, if we start with a $p$-adic representation $\rho: G_k \rar \GL(V_\rho)$ of finite local monodromy, we can choose an $\calO$-lattice $V_\rho^\inte$ of $V_\rho$ stable under the action of $G_k$.  Then we associate a differential module $\calE_\rho$ to the $\calO$-representation given by $V_\rho^\inte$.  It is clear that $\calE_\rho$ does not depend on the choice of the lattice $V_\rho^\inte$.  We call $\calE_\rho$ the \emph{differential module associated to} $\rho$.
\end{construction}

\begin{proposition} \emph{\cite[Proposition~3.5.1]{KSK-Swan1}}
The $(\phi, \nabla)$-module $\calE_\rho$ over $\calR_{K'}$ is independent of the choice of the $p$-basis (up to a canonical isomorphism).
\end{proposition}

\begin{proposition}
The differential module $\calE_\rho$ descends to a differential module over $\calR_{K'}^{\eta_0}$ for some $\eta_0 \in (0,1)$.
\end{proposition}

\begin{proof}
Indeed, defining a differential module  needs only finite data.  So, we can realize it on a certain annulus.  See for instance \cite[Remark~3.4.1]{KSK-Swan1}.
\end{proof}

\begin{remark}
We will often make $\eta_0$ closer to $1^-$ along the way of proving main theorems.  We will see later that all we care about is the asymptotic behavior of $\calE_\rho$ as $\eta_0 \rar 1^-$.
\end{remark}

\begin{remark}
The current construction of associating differential module to a representation (Constructions~\ref{Cstr:repn->bounded-module} and \ref{Cstr:bounded->R_K}) is \emph{not} functorial with respect to the base field $F$ of the representation.
If $F'$ is a finite extension of $F$, for a $p$-adic representation $\rho$ over $F$ of finite local monodromy, one can naturally obtain $\rho \otimes_F F'$ as a $p$-adic representation over $F'$.  Assume that $\kappa_k$ contains the residue field $\FF_{q'}$ of $F'$.  Then the differential modules associated to $\rho$ and $\rho \otimes_F F'$ are the same if $F'/F$ is unramified and $\calE_\rho \otimes_F F' = \calE_{\rho \otimes_F F'}$ if $F'/F$ is totally ramified.

There are two reasons of keeping this non-functoriality flaw.  For one, the differential conductors we define later will be the same if we change $\rho$ to $\rho \otimes_F F'$.  For the other, if we define $\calE_\rho$ using the tensor over $\Zp$ instead of $\ZZ_q$ in Notation~\ref{N:cohen-ring-k}, in which case we do have the functoriality, we will get the direct sum of $[\Fq: \Fp]$ copies of $\calE_\rho$ as differential modules.  When proving the integrality of Swan conductors, we have to come back to study $\calE_\rho$ because $K \otimes_{\Zp} \calO \simeq K'^{\oplus[\Fq:\Fp]}$ is not a field if $q>p$.
\end{remark}

\subsection{Differential conductors}\label{S2-def-diff-conductors}

Given a $p$-adic representation $\rho$ of finite local monodromy, Kedlaya \cite[Section~3.5]{KSK-Swan1} showed that one can define a differential Swan conductor for $\rho$, using the $p$-adic differential module associated to $\rho$.  In this subsection, we review this definition and give an analogous definition for the differential Artin conductor.

\begin{remark}
Starting from this subsection, the Frobenius $\phi$ plays almost no role in our theory; most of the arguments work for solvable differential modules (see \cite[Definition~2.5.1]{KSK-Swan1}), and since all the decompositions for differential modules we encounter are canonical, they automatically respect the Frobenius structure.  The only place we need Frobenius is to link back with representations; see Proposition~\ref{P:phi-Gamma-decomposition}.
\end{remark}

\begin{hypo}\label{H:J-finite-set}
In this subsection, we make an auxiliary hypothesis that $k$ admits a \textit{finite} $p$-basis $\{\serie {b_}m, s\}$.
\end{hypo}

\begin{notation}\label{N:J+}
Let $J = \{\serie{}m\}$ for notational convenience. We save the letters $j$ and $m$ for indexing $p$-basis, except in Subsections~\ref{S:Review-AS} (See Notation~\ref{N:free-jJm}).  We also use $J^+$ to denote $J \cup \{0\}$, where $0$ refers to the uniformizer $s$ in the $p$-basis.
\end{notation}

\begin{definition}\label{D:scale-multiset}
Let $E$ be a differential field of order 1 and characteristic zero, i.e., a field of characteristic zero equipped with a derivation $\partial$.  Assume that $E$ is complete for a non-archimedean norm $|\cdot|$.  Let $M$ be a finite differential module over $E$, i.e. a finite dimensional $E$-vector space equipped with an action of $\partial$ satisfying the Leibniz rule. The \textit{spectral norm of $\partial$ on $M$} is defined to be
$$
|\partial|_{M, \sp} = \lim_{n \rar \infty} |\partial^n|_M^{1/n}
$$
for any norm $|\cdot|_M$ on $M$; it does not depend on the choice of $|\cdot|_M$.
One can prove that $|\partial|_{M, \sp} \geq |\partial|_{E, \sp}$ (\cite[Lemma~5.2.4]{KSK-notes}).
\end{definition}

\begin{remark} \label{R:sp-norm-inv-BC}
For a complete extension $E'$ of $E$, to which the derivation $\partial$ extends, $M \otimes_E E'$ can be viewed as a differential module over $E'$ with spectral norm $|\partial|_{M \otimes_E E', \sp} = \max\{|\partial|_{M, \sp}, |\partial|_{E', \sp}\}$.
\end{remark}

\begin{notation}
Let $\partial_0 = \partial / \partial S, \partial_1 = \partial / \partial B_1, \dots, \partial_m = \partial / \partial B_m$ denote a dual basis of $\Omega_{\calO_K\llbracket S \rrbracket / \calO_{K_0}}^1$ with respect to $dS, dB_1, \dots, dB_m$; they also give a dual basis of $\Omega_{\calR_{K'}^{\eta_0} / K_0}^1$ for all $\eta_0 \in (0,1)$. For a $(\phi, \nabla)$-module $\calE$ over $\calR_{K'}^{\eta_0}$, these differential operators act on $\calE$, commuting with each other and commuting with the Frobenius action.
\end{notation}

\begin{notation}
For $\eta \in [\alpha, \beta] \subset (0, +\infty)$, denote the completion of $\Frac (K' \langle \alpha/t, t/\beta \rangle)$ with respect to the $\eta$-Gauss norm by $F'_\eta$; this does not depend on the choice of $\alpha$ and $\beta$.
\end{notation}

\begin{example}
For $\eta \in \RR_{>0}$, the operator norms of $\partial_{J^+}$ and spectral norms on $F'_\eta$ are as follows.
\[
|\partial_j|_{F'_\eta} = \left\{
\begin{array}{ll}
\eta^{-1} & j = 0 \\
1 & j \in J
\end{array}
\right.; \quad\quad
|\partial_j|_{F'_\eta, \sp} = \left\{
\begin{array}{ll}
p^{-1/(p-1)}\eta^{-1} & j = 0 \\
p^{-1/(p-1)} & j \in J
\end{array}
\right..
\]
\end{example}

\begin{definition}\label{D:scale-multiset-on-pDE}
Let $\calE$ be a $\nabla$-module over $\calR_{K'}^{\eta_0}$.  For $\eta \in [\eta_0, 1)$, denote $\calE_\eta = \calE \otimes_{\calR_{K'}^{\eta_0}} F'_\eta$, which inherits differential operators $\partial_{J^+}$.  Define the \textit{$($non-logarithmic$)$ generic radius (of convergence)} $\bbT (\calE, \eta)$ of $\calE_\eta$ to be
\begin{equation} \label{E:scale}
\min \left\{ \frac{p^{-1/(p-1)}}{|\partial_j|_{\calE_\eta, \sp}} ; j \in J^+ \right\}.
\end{equation}
If $\calE_{\eta,i}$, $i = \serie{}n$, are the Jordan-H\"older factors of $\calE_\eta$ as a $\nabla$-module over $F'_\eta$, we define the \textit{$($non-logarithmic$)$ radius multiset} $\bbS(\calE, \eta)$ to be the set consisting of the generic radius of $\calE_{\eta,i}$ with multiplicity $\dim_{F'_\eta} \calE_{\eta,i}$ for each $i$.

We define the \textit{logarithmic generic radius (of convergence)} $\bbT_\log (\calE, \eta)$ to be
\begin{equation} \label{E:log-scale}
\min \left\{ \frac{p^{-1/(p-1)}\eta^{-1}} {|\partial_0|_{\calE_\eta, \sp}}, \;\frac{p^{-1/(p-1)}} {|\partial_j|_{\calE_\eta, \sp}} ; j \in J \right\}.
\end{equation}

Similarly, we define the \textit{logarithmic radius multiset} $\bbS_\log (\calE, \eta)$ of $\calE$.
\end{definition}

\begin{remark}\label{R:b-nonlog-geq-1}
It is worthwhile to mention that $\bbT(\calE, \eta) \leq \eta$, or more generally, every element in $\bbS(\calE, \eta)$ is smaller than $\eta$.
\end{remark}

\begin{remark}
The logarithmic generic radius and the logarithmic radius multiset are the same as the notions of the generic radius of convergence and the radius multiset in \cite{KSK-Swan1}.
\end{remark}

\begin{definition}
For $j \in J^+$, $\partial_j$ is called \emph{dominant} (resp. \emph{log-dominant}) for $\calE_\eta$ if the minimum of $\bbT(\calE, \eta)$ in \eqref{E:scale} (resp. $\bbT_\log (\calE, \eta)$ in \eqref{E:log-scale}) is achieved by the term involving the spectral norm of $\partial_j$.
\end{definition}

\begin{lemma} \label{L:eventually-dominant}
For a $(\phi, \nabla)$-module $\calE$ over $\calR_{K'}^{\eta_0}$ and $j \in J^+$, there exists $\eta'_0 \in (0, 1)$ such that one of the following two statements is true:
\begin{itemize}
\item for all $\eta \in [\eta'_0, 1)$, $\partial_j$ is (log-)dominant for $\calE_\eta$;
\item for all $\eta \in [\eta'_0, 1)$, $\partial_j$ is not (log-)dominant for $\calE_\eta$.
\end{itemize}
\end{lemma}
\begin{proof}
The logarithmic part is proved in \cite[Lemma~2.7.5]{KSK-Swan1}; the non-logarithmic part can be proved verbatim.
\end{proof}

\begin{definition}
Keep the notation as in previous lemma.
For $j \in J^+$, $\partial_j$ is called \emph{eventually dominant} (resp. \emph{eventually log-dominant}) for $\calE$ if it is dominant (resp. log-dominant) for $\calE_\eta$ for $\eta \rar 1^-$.
\end{definition}

\begin{lemma} \label{L:rotation}
Keep the notation as in Lemma~\ref{L:eventually-dominant}.
Assume that $\partial_0$ is \emph{not} eventually dominant and $\partial_j$ is.  Consider the rotation $g^*: B_j \mapsto B_j + S, B_{J \backslash j} \mapsto B_{J \backslash j}$, and $S \mapsto S$ given by Proposition~\ref{P:cohen-functorial}.  Then $\partial_0 = \partial /\partial S$ is eventually dominant in $g^*\calE$.
\end{lemma}
\begin{proof}
The follows from the fact that the action of $\partial_0$ on $g^*\calE$ is the pull-back of the action of $\partial_0 + \partial_j$ on $\calE$.  For more details, one may consult the proof of \cite[Lemma~2.7.9]{KSK-Swan1}.
\end{proof}

\begin{remark} \label{R:rotation=changing-k=k((s))}
The rotation $g$ in the lemma corresponds to changing the isomorphism \eqref{E:k((s))=k} so that $\bar b_j$ is sent to $b_j + s$ instead; such an isomorphism can be obtained by Corollary~\ref{C:k=k((s))-change}.  In particular, if $\calE_\rho$ comes from a $p$-adic representation $\rho$ of finite local monodromy by Constructions~\ref{Cstr:repn->bounded-module} and \ref{Cstr:bounded->R_K}, $g^* \calE_\rho$ is the differential module associated to the same $\rho$ using the aforementioned alternative isomorphism in place of \eqref{E:k((s))=k}.
\end{remark}

\begin{proposition}\label{P:piecewise-linear}
The functions $f(r) = \log\, \bbT(\calE, e^{-r})$ and $f_\log(r) = \log\, \bbT_\log(\calE, e^{-r})$ on $(0, -\log\, \eta_0]$ are piecewise linear concave functions with slopes in $\frac 1 {\rank \calE !} \ZZ$. They are linear in a neighborhood of $0$.
\end{proposition}
\begin{proof}
The logarithmic version is proved in \cite[Section~2.5]{KSK-Swan1}.  The non-logarithmic version can be proved verbatim as the only difference is the factor $\eta^{-1}$ on the spectral norm of $\partial_0$, which gives an extra linear term $r$.
\end{proof}

\begin{definition}
\label{D:differential-break}
As a consequence of the previous proposition, there exists $b_\dif(\calE) \in \QQ_{\geq 0}$ and $\eta_0 \in (0,1)$ such that $\bbT(\calE, \eta) = \eta^{b_{\dif}(\calE)}$ for all $\eta \in [\eta_0,1)$. This $b_\dif(\calE)$ is called the \emph{$($non-logarithmic$)$ differential ramification break} of $\calE$.  We say that $\calE$ has \emph{uniform slope $b$} if the radius multiset $\bbS(\calE, \eta)$ consists only of $\eta^b$ when $\eta \rar 1$.  The \emph{logarithmic differential ramification break} $b_{\dif, \log}(\calE)$ and the \emph{uniformness of log-slope $b$} are defined similarly except having the subscript $\log$ everywhere.
\end{definition}

The ramification breaks give rise to the break decomposition.

\begin{theorem}\label{T:decomposition}
Let $\calE$ be a $(\phi, \nabla)$-module over $\calR_{K'}^{\eta_0}$, for some $\eta_0 \in (0, 1)$.  Then after making $\eta_0$ sufficiently close to $1^-$, there exists a \emph{unique} decomposition of $(\phi, \nabla)$-modules $\calE = \bigoplus _{b \in \QQ_{\geq 1}} \calE_b$ (resp. $\calE = \bigoplus _{b \in \QQ_{\geq 0}} \calE_{b, \log}$) over $\calR_{K'}^{\eta_0}$, where each of $\calE_b$ (resp. $\calE_{b, \log})$ has uniform slope (resp. log-slope) $b$.
\end{theorem}
\begin{proof}
Since the differential operators act trivially on $\calO$ and commute with $\phi$, It suffices to obtain the decomposition of $\calE$ as a $\nabla$-module over $A_{K'}^1[\eta_0, 1)$;  the uniqueness of the decomposition of $\calE$ follows from the uniqueness of that over $F'_\eta$ for $\eta \in [\eta_0, 1)$.
The logarithmic part of this theorem is proved in \cite[Theorem~2.7.2]{KSK-Swan1}.  We will give the proof of the non-logarithmic decomposition by applying several lemmas in \cite{KSK-Swan1}.

We need to show that if $\calE$ does not have uniform slope $\geq 1$ (see Remark~\ref{R:b-nonlog-geq-1} for the reason of having 1 instead of 0), $\calE$ is decomposable.  If $\partial_0$ is eventually dominant, the decomposition theorem of Christol-Mebkhout \cite[Lemma~2.7.3]{KSK-Swan1} gives the decomposition.  If $\partial_0$ is not eventually dominant, Lemma~\ref{L:eventually-dominant} implies that $\partial_j$ is eventually dominant for some $j \in J$.  By Lemma~\ref{L:rotation}, $\partial_0$ is eventually dominant for $g^*\calE$.  Applying the decomposition theorem \cite[Lemma~2.7.3]{KSK-Swan1} to $g^*\calE$ and pulling back the decomposition along $g^{-1}$, we obtain a nontrivial decomposition of $\calE$ on $\calR_{K'}^{\eta_0}$ for some $\eta_0 \in (0,1)$.
\end{proof}

\begin{proposition}
\label{P:phi-Gamma-decomposition}
In Theorem \ref{T:decomposition}, if the $(\phi, \nabla)$-module $\calE_\rho$ is associated to a $p$-adic representation $\rho$ of finite local monodromy, then the decomposition of $(\phi, \nabla)$-modules induces a direct sum decomposition of the representation $\rho$ so that each of the direct summand of $\calE_\rho$ is the differential module associated to the direct summand of $\rho$.
\end{proposition}
\begin{proof}
By the slope filtration \cite[Theorem~3.4.6]{KSK-Swan1}, the Frobenius action on each direct summand of $\calE$ is of unit-root; the decomposition of the representation follows by \cite[Proposition~3.4.4]{KSK-Swan1}.
\end{proof}

\begin{definition}
Let $\rho: G_k \rar \GL(V_\rho)$ be a $p$-adic representation with finite local monodromy.  Let $\calE$ be the differential module corresponding to $V_\rho / V_\rho^{I_k}$ by Constructions~\ref{Cstr:repn->bounded-module} and \ref{Cstr:bounded->R_K}, where $V_\rho^{I_k}$ is the unramified piece of $V_\rho$ consisting of elements in $V_\rho$ which are fixed by $I_k$.  By Theorem~\ref{T:decomposition} above, there exists a multiset $\{\serie{a_}d\}$ such that for all $\eta$ sufficiently close to 1, $\bbS(\calE, \eta) = \{\eta^{a_1}, \dots, \eta^{a_d}\}$.  Define the \textit{differential Artin conductor} of $\calE$ (resp. $\rho$), denoted by $\Art_\dif (\calE)$ (resp. $\Art_\dif (\rho)$), to be $a_1 + \cdots + a_d$.  The \emph{differential Swan conductor} of $\calE$ (resp. $\rho$), $\Swan_\dif(\calE)$ (resp. $\Swan_\dif(\rho)$), is defined similarly by adding the subscript $\log$ everywhere.
\end{definition}

\begin{remark}  \label{R:cant-dist-unr-vs-tame}
In the above definition, we split off the unramified part because it has both conductors 0.  The reason of doing so is that the convergence radius multiset cannot distinguish the unramified part and the tame part which give different contributions to the Artin conductor.  This does not matter for Swan conductors and we may define the Swan conductor without first taking out the unramified piece.
\end{remark}

\begin{remark}
By \cite[Proposition~2.6.6]{KSK-Swan1}, the definition of the differential Swan conductors does not depend on the choice of the uniformizer $s$ and $p$-basis $\{\serie{b_}m, s\}$.  Moreover, we may feel free to remove the Hypothesis~\ref{H:J-finite-set} and define the differential Swan conductors for arbitrary complete discretely valued field of equal characteristic $p$ \cite[Corollary~3.5.7]{KSK-Swan1}.  The similar statement is also true for the differential Artin conductors and the proof is the same as for Swan conductors.
\end{remark}

\subsection{Basic properties}
\label{S3:basic-prop}

In this subsection, we state some basic properties of differential conductors.  We do not impose any hypothesis on $k$.

\begin{theorem}\label{T:properties-Ked-cond}
Differential conductors satisfy the following properties:

\emph{(0)} When the residue field $\kappa_k$ is perfect, the differential Artin and Swan conductors are the same as the classical ones defined in \emph{\cite{BOOK-local-fields}}.

\emph{(1)} For any representation $\rho$ of finite local monodromy, $\Swan_\dif(\rho) \in \ZZ_{\geq 0}$ and $\Art_\dif(\rho) \in \ZZ_{\geq 0}$.

\emph{(2)} Let $k'/k$ be a tamely ramified extension of ramification degree $e'$.  Let $\rho$ be a representation of $G_k$ of finite local monodromy and let $\rho'$ denote the restriction of $\rho$ to $G_{k'}$. Then $\Swan_\dif(\rho') = e' \cdot \Swan_\dif (\rho)$.  If $e' = 1$, i.e., $k'/k$ is unramified, $\Art_\dif(\rho') = \Art_\dif (\rho)$.

\emph{(3)} Let $\rho$ be a faithful $p$-adic representation of the Galois group of a Galois extension $l/k$.  If $l/k$ is tamely ramified and not unramified, $b_\dif(\rho) = 1$ and $b_{\dif, \log}(\rho) = 0$.  If $l/k$ is unramified, $b_\dif(\rho) = b_{\dif, \log}(\rho) = 0$.

\emph{(4)} Put $G_k^0 = G_k$ and $G_k^a = I_k$ for $a \in (0, 1]$.  For $a > 1$, let $R_a$ be the set of finite image representations $\rho$ with differential ramification breaks less than $a$.  Define $G_k^a = \bigcap _{\rho \in R_a} \big(I_k \cap \ker (\rho)\big)$ and write $G_k^{a+}$ for the closure of $\cup_{b > a} G_k^b$. This defines a differential filtration on $G_k$ such that for all finite image representation $\rho$, $\rho(G_k^a)$ is trivial if and only if $\rho \in R_a$.

Similarly, put $G_{k, \log}^0 = G_k$.  For $a > 0$, let $R_{a, \log}$ be the set of finite image representations $\rho$ with logarithmic differential ramification breaks less than $a$.  Define $G_{k, \log}^a = \bigcap _{\rho \in R_{a, \log}} \big(I_k \cap \ker (\rho)\big)$ and write $G_{k, \log}^{a+}$ for the closure of $\cup_{b > a} G_{k, \log}^b$. This defines a differential logarithmic filtration on $G_k$ such that for all finite image representation $\rho$, $\rho(G_{k, \log}^a)$ is trivial if and only if $\rho \in R_{a, \log}$.

Moreover, 
\begin{eqnarray*}
\textrm{for } a > 0 \textrm{, } G_k^a / G_k^{a+} & = & \left\{
\begin{array}{ll}
0 & a \notin \QQ \\
\textrm{an abelian group killed by $p$} & a \in \QQ
\end{array}
\right. \\
\textrm{for } a > 1 \textrm{, } G_{k, \log}^a / G_{k, \log}^{a+} & = & \left\{
\begin{array}{ll}
0 & a \notin \QQ \\
\textrm{an abelian group killed by $p$} & a \in \QQ
\end{array}
\right.
\end{eqnarray*}
\end{theorem}

\begin{proof}
For (0), see \cite[Theorem~5.23]{KSK-overview}. 

For the rest of the statements, the proof for Swan conductors can be found in \cite[Section~3.5]{KSK-Swan1}; we will only prove the corresponding properties for differential Artin conductors.  As in the proof for the differential Swan conductors, we may first reduce to the case when Hypothesis~\ref{H:J-finite-set} holds.

(1)  The proof goes the same as \cite[Theorem~2.8.2]{KSK-Swan1} because we have the decomposition Theorem \ref{T:decomposition}.  An alternative proof is to apply Lemma~\ref{L:rotation}, and reduce to the case when $\partial_0$ is dominant (see also Remark~\ref{R:rotation=changing-k=k((s))}); then one can forget about $\serie{\partial_}m$ and hence reduce to the perfect residue field case which is statement (0) of the theorem.

(2)  Since an unramified extension $l/k$ only changes the field $K$ but not the uniformizer $s$, we can use the same $s$ as the uniformizer of $l$.  The corresponding differential module $\calE_{\rho'}$ of $\rho'$ is just a simple extension of scalar.  Since the calculation of spectral norms does not depend on the base field (Remark~\ref{R:sp-norm-inv-BC}), we compute the same result on spectral norms and hence have the same Artin conductor.

(3) is an immediate consequence of the Swan case.  But be careful that the differential ramification breaks can not distinguish unramified extensions from tamely ramified extensions (see also Remark~\ref{R:cant-dist-unr-vs-tame}).

(4) The proof for the non-logarithmic differential filtration is much simpler than the logarithmic case because of the different normalization in the Definition~\ref{D:scale-multiset-on-pDE}.  In virtue of the proof of \cite[Theorem~3.5.13]{KSK-Swan1}, it suffices to show that we can do some rotation so that $\partial_0$ becomes dominant; this is exactly the content of Lemma~\ref{L:rotation}.
\end{proof}

\begin{remark}
The converse of (3) is a well-known and well-accepted fact for experts.  However, we are unable to find good references to support a proof.  As it will become an easy consequence of the comparison Theorem~\ref{T:main-theorem} and properties of arithmetic ramification conductors (Proposition~\ref{P:AS-space-properties}(6)), we do not state it here.
\end{remark}

\begin{remark}\label{R:k_0-alg-closed}
Note that the invariance of the differential conductors under unramified base changes enables us to assume that $\kappa_0$ is algebraically closed.  This justifies the assumption we made in Notation~\ref{N:representation}.
\end{remark}

\section{Thickening techniques}
\setcounter{equation}{0}

In this section, we introduce a thickening technique.  Vaguely speaking, it is to construct a reasonable object which can be thought of as a tubular neighborhood of the ``diagonal embedding of $A_K^1[\eta_0, 1)$ into $A_K^1[\eta_0, 1) \times_{K_0} A_K^1[\eta_0, 1)$".  Be careful that the latter rigid space is not really well-defined.

We first start with a geometric interpretation of this construction and then move on to the abstract definition of the thickening space.

We keep the Hypothesis~\ref{H:J-finite-set} throughout this section.

\begin{notation} \label{N:rigid-func-on-polydisc}
For $\alpha \in (0, +\infty)$, let $A_K^m[0,\alpha]$ (resp. $A_K^m[0, \alpha)$) denote the polydisc (resp. open polydisc) with radius $\alpha$ (centered at the origin).  Let $K \langle u_1 / \alpha, \dots, u_m / \alpha \rangle$ denote the ring of analytic functions on the disc $A_K^m[0, \alpha]$.
\end{notation}

Later, we will see many homomorphisms between rings of functions on $K$-rigid spaces which are only $K_0$-linear.  It is unfair to say that they induce morphisms of rigid spaces; however, we prefer to keep some geometric flavor of the whole construction.  On the other hand, these rigid spaces are all quasi-Stein or affinoid; knowing the ring of analytic functions is equivalent to knowing the rigid spaces.

\begin{notation}
For a continuous homomorphism $f^*: A \rar B$ between affinoid or Fr\'echet algebras (not necessarily respecting the ground field $K$), we write formally $f: \Max(B) \rar \Max(A)$, as the \emph{geometric incarnation} of the homomorphism.  Pullbacks along maps and Cartesian diagrams are thought of as (completed) tensor products.  (In fact, in all cases we encounter, we do not need to take the completion for the tensor products.)  In short, whenever such a map is given, strictly speaking, we should view it as a continuous ring homomorphism.
\end{notation}

\subsection{Geometric thickening}
\label{S:geometric-thickening}

In this subsection, we describe the thickening technique when the residue field $\kappa_k$ can be realized as the field of rational functions on a smooth $\kappa_0$-variety.  The purpose of this subsection is solely to provide some geometric intuition of the thickening construction in the next subsection; the content in this subsection will not be used in the rest of this paper.

\begin{hypo} \label{H:geom}
\emph{Only} in this subsection, we assume that the field $\kappa_k$ is a finite separable extension of $\kappa_0(\serie {\bar b_} m)$.
\end{hypo}

\begin{construction}
Let $\overline X$ be a smooth variety over $\kappa_0$ whose field of rational functions is $\kappa_k$; such an $\overline X$ exists because we may realize it an affine scheme \'etale over $\Spec \kappa_0 [\serie{\bar b_}m]$ which induces the extension $\kappa_k / \kappa_0(\serie{\bar b_}m)$.  We may further shrink $\overline X$ so that it is the special fiber of an affine smooth formal scheme $\gothX$ over $\calO_{K_0}$ of topological finite type, i.e. $\gothX \times_{\Spf \calO_{K_0}} \Spec \kappa_0 = \overline X$.  We may further shrink $\gothX$ and $\overline X$ so that we have lifts $B_1, \dots, B_m$ of $\bar b_1, \dots, \bar b_m$ on $\gothX$ and that $dB_1, \dots, dB_m$ form a basis of the sheaf of relative differentials $\Omega^1_{\gothX / \calO_{K_0}}$.  We use $\bbX$ to denote the ``generic fiber" of $\gothX$ as a rigid space over $\Sp(K_0)$, in the sense of Raynaud; it is affinoid.

Consider the following commutative diagram
\[
\xymatrix{
\overline X \ar@{_{(}->}[d] \ar[r] & \gothX \ar@{_{(}->}[d] & \bbX \ar[l] \ar@{_{(}->}[d] \\
P = \overline X \times_{\kappa_0} \AAA^1_{\kappa_0} \ar[d] \ar[r] & \calP = \gothX \times_{\Spf \calO_{K_0}} \widehat\AAA^1_{\calO_{K_0}} \ar[d] & \bbP = \bbX \times_{K_0} A_{K_0}^1[0, 1] \ar[d] \ar[l] \\
\Spec \kappa_0 \ar[r] & \Spf \calO_{K_0} & \Sp(K_0) \ar[l]
}
\]
where the vertical arrows from the first row to the second row are all embeddings of zero sections and the coordinates of $\AAA^1_{\kappa_0}$ and $\widehat\AAA^1_{\calO_{k_0}}$ are denoted by $ s$ and $S$, respectively.

The tube of $\overline X$ in $\bbP$, denoted by $]\overline X[_\calP$, is isomorphic to $\bbX \times A_{K_0}^1[0,1)$.  Let $\calO_\bbX$ be the ring of rigid analytic functions on $\bbX$; then $K$ is exactly the $p$-adic completion of $\Frac \calO_\bbX$.  If we base change the tube $]\overline X[_\calP$ from $\bbX$ over to $K$, we get $A_K^1[0,1)$.  We are interested in the annulus $A_K^1[\eta_0, 1)$ for some $\eta_0 \in (0,1)$, which can be obtained by base changing $\bbX \times A_{K_0}^1[\eta_0, 1)$ from $\bbX$ to $K$.
\end{construction}

Now, we consider the thickening space of this annulus $A_K^1[\eta_0, 1)$.

\begin{construction}
\label{Cstr:geometric}
Consider the following commutative diagram
\[
\xymatrix{
\overline X \ar@{^{(}->}[r] \ar[dr] &  P \ar@{^{(}->}[r]^-{\Delta_{ P}} \ar[d] &  P \times_{\kappa_0}  P \ar[dl] \ar@{^{(}->}[r] & \calP \times_{\calO_{K_0}} \calP \ar[d] & \bbP \times_{K_0} \bbP \ar[l] \ar[d] \\
& \Spec \kappa_0 \ar[rr] & & \Spf \calO_{K_0} & \Sp(K_0) \ar[l]
}
\]
where we use $\pr_i: \calP \times_{\calO_{K_0}} \calP \rar \calP$ to denote the projection to the $i$-th factor for  $i = 1, 2$.  Then $\calP \times_{\calO_{K_0}} \calP$ has a set of local parameter given by $B_1 = \pr_1^*(B_1), \dots, B_m = \pr_1^*(B_m), S = \pr_1^*(S), B'_1 = \pr_2^*(B_1), \dots, B'_m = \pr_2^*(B_m), S' = \pr_2^*(S)$.
By Berthelot's Fibration Theorem \cite[TH\'EOR\`EME~1.3.2]{Berthelot-rig-coh-part-I}, we have an isomorphism 
\[
]\overline X[_{\calP \times_{\calO_{K_0}} \calP} \simeq ]\overline X[_\calP \times_{K_0} A_{K_0}^{m+1}[0, 1),
\]
where the factor $]\overline X[_\calP$ respects the projection $\pr_1$ and the coordinates for the open polydisc on the right hand side are given by $\delta_0 = S - S', \delta_1 = B_1 - B'_1, \dots, \delta_m = B_m - B'_m$.  The geometric thickening space is the subspace of $]\overline X[_{\calP \times_{\calO_{K_0}} \calP}$ where $|\delta_0| = |S - S'| < |S|$, or more precisely,
\[
\bbX \times_{K_0} \big\{(S, \delta_0) \in A^2_{K_0}[0,1) \big| |\delta_0| < |S| \big\} \times_{K_0} A_{K_0}^m[0, 1).
\]
Thus, \emph{the thickening space}, denoted by $TS_k^{\geq \eta_0}$, of $A_K^1[\eta_0, 1)$ is the space obtained by base changing
\[
\bbX \times_{K_0} \big\{(S, \delta_0) \in A^2_{K_0}[0,1) \big| |S| \geq \eta_0, |\delta_0| < |S| \big\} \times_{K_0} A_{K_0}^m[0, 1).
\]
from $\bbX$ to $K$.

The projection $\pr_1: \bbP \times_{K_0} \bbP \rar \bbP$ gives a $K$-morphism of rigid spaces $\pi: TS_k^{\geq \eta_0} \rar A_K^1[\eta_0, 1)$; the projection $\pr_2: \bbP \times_{K_0} \bbP \rar \bbP$ gives a $K_0$-morphism of rigid spaces $\tilde \pi: TS_k^{\geq \eta_0} \rar A_K^1[\eta_0, 1)$.  The morphism $\tilde \pi$ does not respect the $K$-rigid space structure; one should always think of $\tilde \pi$ as the ring homomorphism between the corresponding ring of analytic functions.  In our notation earlier, it is just the geometric incarnation of the map on ring of global sections.
\end{construction}

\subsection{General thickening construction}
\label{S:thicken-annulus}

In this subsection, we introduce the thickening spaces and study basic properties of differential modules over them.

We keep Hypothesis~\ref{H:J-finite-set} in this subsection.  Note that Hypothesis~\ref{H:geom} is no longer in force from now on.

\begin{definition}\label{D:thickening-space}
For $\eta \in (0,1)$, we write $Z_k^\eta = A_K^1[\eta, \eta]$.  For $a \in \QQ_{>1}$ and $\eta_0 \in (0,1)$, we define the \emph{thickening space (of $A^1_K [\eta_0, 1)$ and level $a$)} to be the rigid space over $K$:
\begin{equation} \label{E:TS}
TS_k^{a, \geq \eta_0} = \big\{ (S, \delta_0, \dots, \delta_m) \in A_K^{m+2}[0,1) \big| |S| \geq \eta_0; |\delta_j| \leq |S|^a \textrm{ for } j \in J^+ \big\}.
\end{equation}
For $\eta \in [\eta_0, 1)$, we put
\[
TS_k^{a, \eta} = A_K^1[\eta, \eta] \times_K A_K^{m+1}[0, \eta^a].
\]

Similarly, for $a \in \QQ_{>0}$ and $\eta_0 \in (0,1)$, we define the \textit{log-thickening space (of $A^1_K [\eta_0, 1)$ and level $a$)} to be
\begin{equation}
\label{E:TS-log}
TS_{k, \log}^{a, \geq \eta_0} = \big\{ (S, \delta_0, \dots, \delta_m) \in A_K^{m+2}[0,1) \big| |S| \geq \eta_0; |\delta_0| \leq |S|^{a+1}; |\delta_j| \leq |S|^a \textrm{ for } j \in J \big\}.
\end{equation}
For $\eta \in [\eta_0,1)$, denote
\[
TS_{k, \log}^{a, \eta} = A_K^1[\eta, \eta] \times_K A_K^1[0, \eta^{a+1}] \times_K A_K^m[0, \eta^a].
\]

We use $\calO_{TS_k^{a, \geq \eta_0}}$, $\calO_{TS_k^{a, \eta}}$, $\calO_{TS_{k, \log}^{a, \geq \eta_0}}$, and $\calO_{TS_{k, \log}^{a, \eta}}$ to denote the ring of analytic functions on the corresponding spaces, respectively.

Let $|\cdot|_{Z_k^\eta}$ denote the $\eta$-Gauss norm on $Z_k^\eta$.  For $a \in \QQ_{>1}$, let $|\cdot|_{TS_k^{a, \eta}}$ denote the Gauss norm on $TS_k^{a, \eta}$; for $a>0$, let $|\cdot|_{TS_{k, \log}^{a, \eta}}$ denote the Gauss norm on $TS_{k, \log}^{a, \eta}$.

The union of all $TS_k^{a, \geq \eta_0}$ is the $TS_k^{\geq \eta_0}$ we discussed in Construction~\ref{Cstr:geometric}.
\end{definition}

\begin{caution}  \label{Cau:not-admissible}
One may want to write $TS_k^{a, \geq \eta_0} = \bigcup_{\eta \in [\eta_0, 1)} A_K^1[\eta, 1) \times_K A_K^{m+1}[0, \eta^a]$ for simplicity as in the introduction.  However, this will not define the same rigid space as in \eqref{E:TS}, because the union does  \emph{not} give an admissible cover of $TS_k^{a, \geq \eta_0}$.  Similar expression for log-thickening space is not valid either.  Nevertheless, it might be helpful to think the space and picture the geometry this way.

On the other hand, it is true that an element of $K\llbracket S, \delta_0, \dots, \delta_m\rrbracket$ lies in $\calO_{TS_k^{a, \geq \eta_0}}$ (resp. $\calO_{TS_{k, \log}^{a, \geq \eta_0}}$) if and only if it has bounded norms for all $|\cdot|_{TS_k^{a, \eta}}$ (resp. $|\cdot|_{TS_{k, \log}^{a, \eta}}$) for all $\eta \in [\eta_0, 1)$.
\end{caution}

\begin{remark}
We need $a \in \QQ$ in Definition~\ref{D:thickening-space} to make sure that \eqref{E:TS} and \eqref{E:TS-log} actually define a (Berkovich) rigid analytic space.  For individual $TS_k^{a, \eta}$ and $TS_{k, \log}^{a, \eta}$'s, one may just take $a \in \RR$.
\end{remark}

\begin{notation}
For $a \in \QQ_{>1}$ (resp. $a \in \QQ_{>0}$) and $\eta_0 \in (0,1)$, denote the natural embedding of $Z_k^{\geq \eta_0}$ into the locus where $\delta_j = 0$ for $j \in J^+$ by $\Delta: Z_k^{\geq \eta_0} \inj TS_k^{a, \geq \eta_0}$ (resp. $\Delta: Z_k^{\geq \eta_0} \inj TS_{k, \log}^{a, \geq\eta_0}$).  Also, we have the na\"ive projection $\pi : TS_k^{a, \geq \eta_0} \rar Z_k^{\geq \eta_0}$ (resp. $\pi : TS_{k, \log}^{a, \geq\eta_0} \rar Z_k^{\geq \eta_0}$) by projecting to the first factor.  These morphisms are compatible when changing $a$ and $\eta_0$, or replacing $\geq \eta_0$ by $\eta$ for some $\eta \in [\eta_0, 1)$.

To simplify notation, for $a$ and $\eta_0$ as above, we identify $\calO_{Z_k^{\geq \eta_0}}$ as a subring of $\calO_{TS_k^{a, \geq \eta_0}}$ and of $\calO_{TS_{k, \log}^{a, \geq \eta_0}}$ via $\pi^*$; same for $\eta$ instead of $\geq \eta_0$.  It is worthwhile to point out that $\pi^*$ is an isometry; hence the identification will not change any calculation on norms.
\end{notation}

\begin{proposition}
We have a unique continuous $\calO_{K_0}$-homomorphism $\tilde \pi^*: \OK \llbracket S \rrbracket \rar \OK \llbracket S, \delta_{J^+} \rrbracket$ such that $\tilde \pi^*(S) = S + \delta_0$ and $\tilde \pi^*(B_j) = B_j + \delta_j$ for all $j \in J$.  Moreover, for $g \in \OK$, $\tilde \pi^*(g) - g \in (\delta_1, \dots, \delta_m) (g) \OK \llbracket \delta_1, \dots, \delta_m \rrbracket$.
\end{proposition}
\begin{proof}
It follows from Corollary~\ref{C:cohen-deformation} immediately.
\end{proof}

\begin{theorem} \label{T:thickening-morphism}
For $a \in \QQ_{>1}$ (resp. $a \in \QQ_{>0}$) and $\eta_0 \in (0,1)$, the homomorphism $\tilde \pi^*$ induces a $K_0$-homomorphism $\tilde \pi^*: \calO_{Z_k^{\geq \eta_0}} = \calR_K^{\eta_0} \rar \calO_{TS_k^{a, \geq \eta_0}} $ (resp. $\tilde \pi^*: \calO_{Z_k^{\geq \eta_0}}  \rar \calO_{TS_{k, \log}^{a, \geq \eta_0}}$) such that $\Delta^* \circ \tilde \pi^* = \id$; same if replacing $\geq \eta_0$ by $\eta$ for some $\eta \in [\eta_0, 1)$.

For any $g \in \calO_{Z_k^\eta}$ and for $a>1$ (resp. $a>0$),
\begin{equation}
\label{E:compare-norm-of-pi-and-tilde-pi}
|\tilde \pi^* (g) - g|_{TS_k^{a, \eta}} \leq \eta^{a-1} \cdot |g|_{Z_k^\eta} \quad (\textrm{resp. }
|\tilde \pi^* (g) - g|_{TS_{k, \log}^{a, \eta}} \leq \eta^a \cdot |g|_{Z_k^\eta} \ ).
\end{equation}
In particular, $|\tilde \pi^*(g)|_{TS_k^{a, \eta}} = |\tilde \pi^*(g)|_{TS_{k, \log}^{a, \eta}} = |g|_{Z_k^\eta}$.  Moreover, we have the following bound for $TS_k^{a, \eta}$: if $g \in \calO_{Z_k^\eta} \cap \calO_K \llbracket S \rrbracket$, then
\begin{equation}\label{E:bound-of-tilde-pi-non-log}
|\tilde \pi^* (g) - g|_{TS_k^{a, \eta}} \leq \eta^a.
\end{equation}
\end{theorem}
\begin{proof}
We need only to establish the bound on the norms.
Let $g = \sum_{i \in \ZZ} a_i S^i \in K \llbracket S \rrbracket$ such that $|g|_{Z_k^\eta} < +\infty$, we have
\begin{equation} 
\label{E:thickening-morphism1}
\tilde \pi^*(g) - g = \sum_{i \in \ZZ} \big( \tilde \pi^*(a_i) (S + \delta_0)^i - a_i S^i \big) = \sum_{i \in \ZZ} \big( (\tilde \pi^*(a_i)- a_i) (S + \delta_0)^i + a_i ((S + \delta_0)^i - S^i) \big).
\end{equation}

Since $\tilde \pi^*(a_i) - a_i \in (\serie{\delta_}m) (a_i) \OK \llbracket \serie{\delta_}m \rrbracket$, we have
\begin{equation} 
\label{E:thickening-morphism2}
|\tilde \pi(a_i) - a_i|_{TS_k^{a, \eta}} \leq |a_i| \eta^a, \quad |\tilde \pi(a_i) - a_i|_{TS_{k, \log}^{a, \eta}} \leq |a_i| \eta^a.
\end{equation}
Moreover, we can bound $(S + \delta_0)^i - S^i$ by
\begin{equation} 
\label{E:thickening-morphism3}
\big|(S + \delta_0)^i - S^i \big|_{TS_k^{a, \eta}} \leq \eta^{a+i-1}, \quad \big|(S + \delta_0)^i - S^i \big|_{TS_{k, \log}^{a, \eta}} \leq \eta^{a+i}.
\end{equation}

Plugging the estimates \eqref{E:thickening-morphism2} and \eqref{E:thickening-morphism3} into \eqref{E:thickening-morphism1}, we obtain \eqref{E:compare-norm-of-pi-and-tilde-pi}.  When $g \in \OK \llbracket S \rrbracket$, \eqref{E:thickening-morphism3} always gives $\big|(S + \delta_0)^i - S^i \big|_{TS_k^{a, \eta}} \leq \eta^a$ for $i \geq 0$ (when $i=0$, we have zero); the equation \eqref{E:bound-of-tilde-pi-non-log} follows.

Finally, the equalities $|\tilde \pi^*(g)|_{TS_k^{a, \eta}} = |\tilde \pi^*(g)|_{TS_{k, \log}^{a, \eta}} = |g|_{Z_k^\eta}$ ensures that we have well-defined continuous homomorphisms $\tilde \pi^*: \calO_{Z_k^{\geq \eta_0}}  \rar \calO_{TS_{k}^{a, \geq \eta_0}}$ or $\calO_{TS_{k, \log}^{a, \geq \eta_0}}$.
\end{proof}

\begin{notation}
We use $\tilde \pi: TS_k^{a, \geq \eta_0} \rar Z_k^{\geq \eta_0}$ and $\tilde \pi: TS_{k, \log}^{a, \geq \eta_0} \rar Z_k^{\geq \eta_0}$ to denote the geometric incarnations of the homomorphisms $\tilde \pi^*$ constructed in Theorem~\ref{T:thickening-morphism}; same for $\eta$ in place of $\geq \eta_0$ when $\eta \in [\eta_0, 1)$.  To emphasize again, whenever we refer to $\tilde \pi$, strictly speaking, we are referring to the corresponding homomorphism $\tilde \pi^*$ on rings.
\end{notation}

\begin{remark}
For $a>0$, one can factor the map $\tilde \pi$ for non-log thickening space as $TS_k^{a+1, \geq \eta_0} \rar TS_{k, \log}^{a, \geq \eta_0} \map \rar {\tilde \pi} Z_k^{\geq \eta_0}$, where the second map is the $\tilde \pi$ for the log-thickening space.  Again, this should be thought of as factorization for ring homomorphisms.
\end{remark}

\begin{notation}
For a $\nabla$-module $(\calE, \nabla_\calE)$ over $Z_k^{\geq \eta_0}$ relative to $K_0$, we call $\tilde \pi^*\calE$ the \textit{thickened differential module} of $\calE$, denoted by $\calF$.  We view $\calF$ as a differential module over $TS_k^{a, \geq \eta_0}$ or $TS_{k, \log}^{a, \geq \eta_0}$ relative to $Z_k^{\geq \eta_0}$, with respect to the differential operators $\seriezero{\partial / \partial \delta_}m$.  In precise terms, the connection is given by
\begin{eqnarray*}
\calF = \calE \otimes_{\calR_K^{ \eta_0}, \tilde \pi^*} \calO_{TS_k^{a, \geq \eta_0}} &\stackrel{\nabla_\calE}\lrar & \calE \otimes \Omega^1_{Z_k^{\geq \eta_0} / K_0}\otimes_{\calR_K^{ \eta_0}, \tilde \pi^*} \calO_{TS_k^{a, \geq \eta_0}}\\
&\lrar & \calE \otimes_{\calR_K^{ \eta_0}, \tilde \pi^*} \Omega^1_{TS_k^{a, \geq \eta_0} / K_0}\\
&\lrar & \calE \otimes_{\calR_K^{ \eta_0}, \tilde \pi^*} \Omega^1_{TS_k^{a, \geq \eta_0} / Z_k^{\geq \eta_0}}
\end{eqnarray*}
in the nonlog case and the log case is obtained similarly, with subscript log at appropriate places.  Moreover, this construction is compatible for different $a$'s and $\eta_0$'s.
\end{notation}

The following proposition links the spectral norms on $\calE$ and the spectral norms on its thickening $\calF$.

\begin{proposition}\label{P:sp-norm-on-TS}
Let $\eta \in [\eta_0, 1)$.  The spectral norms of $\partial_{J^+}$ on $\calE_\eta$ over $Z_k^\eta$ and the spectral norms of $\partial / \partial \delta_{J^+}$ on $\calF_{a, \eta} = \calF \otimes \Frac(\calO_{TS_k^{a, \eta}})^\wedge$ and $\calF_{a, \eta, \log} = \calF \otimes \Frac(\calO_{TS_{k, \log}^{a, \eta}})^\wedge$ are related as follows.
\begin{eqnarray*}
|\partial / \partial \delta_j|_{\calF_{a, \eta}, \sp} &=& \max\big\{ |\partial_j|_{\calE_\eta, \sp}, p^{-1/(p-1)}\eta^{-a} \big\}\ , j \in J^+; \\
|\partial / \partial \delta_0|_{\calF_{a, \eta, \log}, \sp} &=& \max\big\{ |\partial_0|_{\calE_\eta, \sp}, p^{-1/(p-1)}\eta^{-a-1} \big\},\\
|\partial / \partial \delta_j|_{\calF_{a, \eta, \log}, \sp} &=& \max\big\{ |\partial_j|_{\calE_\eta, \sp}, p^{-1/(p-1)}\eta^{-a} \big\}\ , j \in J. 
\end{eqnarray*}
\end{proposition}

\begin{proof}
Note that $\tilde \pi^*(dB_j) = dB_j + d\delta_j$ for $j \in J$ and $\tilde \pi^*(dS) = dS + d\delta_0$. The actions of $\partial / \partial \delta_j$, $j \in J$ (resp. $j = 0$),  on $\calF_{a, \eta}$ and $\calF_{a, \eta, \log}$ are the same as the action of $\partial / \partial B_j$ (resp. $\partial / \partial S$) on $\calE_\eta$.  More precisely, we have $\tilde \pi^*(\partial / \partial S(x)) = \partial / \partial \delta_0 (\tilde \pi^*(x))$ and $\tilde \pi^*(\partial / \partial B_j(x)) = \partial / \partial \delta_j (\tilde \pi^*(x))$ for any $j \in J$ and $x \in \calR_{K}^{\eta}$ \emph{or} $\calE_\eta$.

The statement follows from the facts that $\delta_J$ are transcendental over $\calO_{Z_k^\eta}$ and the homomorphism $\tilde \pi^*$ is isometric (via Theorem~\ref{T:thickening-morphism}).
\end{proof}

\subsection{Good generators of the extension}
\label{S:construction-calE}

In this subsection, we first show that when $l/k$ is totally and wildly ramified, we can choose nice generators of $\calO_l$ as an $\calO_k$-algebra, so that the corresponding extension on the Robba rings takes a simple form.
Then, we give a more explicit construction of the differential module associated to a $p$-adic representation.

We keep Hypothesis~\ref{H:J-finite-set} for this subsection.  Moreover, we impose the following.

\begin{hypo}\label{H:tot-ram}
For the rest of this section, we assume that $l/k$ is a finite totally and wildly ramified Galois extension.
\end{hypo}

\begin{remark}
This is a mild hypothesis as both arithmetic and differential conductors behave well under unramified extensions and the tamely ramified case is well-known (see Theorem~\ref{T:properties-Ked-cond}(3) and Proposition~\ref{P:AS-space-properties}(6)).
\end{remark}

\begin{notation}  \label{N:l}
Let $l$ be as above, and let $G_{l/k}$ denote the Galois group of $l/k$.  Denote the ring of integers and the residue field of $l$ by $\calO_l$ and $\kappa_l$, respectively.  Given a uniformizer $t$ of $l$, we fix a non-canonical isomorphism $\kappa_l ((t)) \simeq l$.  For a $p$-basis $\bar c_J$ of $\kappa_l$, we use $c_J$ to denote the image of $\bar c_J$ under this isomorphism; we may use the same index set $J$ because $\kappa_l / \kappa_k$ is a finite extension.

Let $\calO_L$ be the Cohen ring of $\kappa_l$ with respect to $\bar c_J$ and let $C_J$ be the canonical lifts of $\bar c_J$.  Set $L = \Frac \calO_L$.
\end{notation}

\begin{caution} \label{Cau:residue-extn-not-embedable}
The residue field extension $\kappa_l / \kappa_k$ is typically not separable and hence \emph{cannot} be embedded into the extension $l = \kappa_l ((t))$ over $k = \kappa_k ((s))$.
\end{caution}

The reader may skip the following construction and remark when reading the paper first time, as they are not the key of the theory.  The upshot is that we need some ``good" generators and relations of $\calO_l$ as an $\calO_k$-algebra.

\begin{construction}\label{C:cj-generate-l}
We may temporarily drop the finiteness Hypothesis~\ref{H:J-finite-set} on $p$-basis for this construction.  Let $\bbk_0 = \kappa_k$ with $p$-basis $(\bar b_j)_{j \in J}$.  By possibly rearranging the indexing in $\bar b_J$, we will inductively construct a ``good" $p$-basis $(\bar c_j)_{j \in J}$ of $\kappa_l$ and $\bbk_j = \kappa_k(\bar c_1, \dots, \bar c_j)$ with $p$-basis $\big\{ \bar c_1, \dots, \bar c_j, \bar b_{J \bs \{1, \dots, j\}} \big\}$ so that $\bbk_m = \kappa_l$ for $m$ sufficiently large.

Assume that we have constructed $\bbk_{j-1}$.  Let $r_j$ be the unique integer such that $\kappa_l \subseteq \bbk_{j-1}^{p^{-r_j}}$ but $\kappa_l \nsubseteq \bbk_{j-1}^{p^{-r_j+1}}$.  If $r_j = 0$, we must have $\bbk_{j-1} = \kappa_l$; in this case, we set $\bar c_\alpha = \bar b_\alpha$ and $r_\alpha = 0$ for all $\alpha \in J \bs \{1, \dots, j-1\}$ and stop the induction.  Otherwise we assume that $r_j >0$.  Take $\bar c_j$ to be any element in $\kappa_l \bs \bbk_{j-1}^{p^{-r_j+1}}$ and let $\bbk_j = \bbk_{j-1}(\bar c_j)$.  Then $\bar c_j^{p^{r_j}} \in \bbk_{j-1}$ and $[\bbk_j: \bbk_{j-1}] = p^{r_j}$.  There must exist one element in $\bar b_{J \bs\{1, \dots, j-1\}}$ such that the rest together with $\bar c_1, \dots, \bar c_j$ form a $p$-basis of $\bbk_j$.  We assume that this element is $\bar b_j$ by re-indexing $\bar b_{J \bs \{1, \dots, j-1\}}$.  This finishes the induction.

From the induction, one can see that $r_j$ form a decreasing sequence of nonnegative integers.  But we do not need this fact.

Since $\kappa_l / \kappa_k$ is finite, the above construction ensures that  $\bar c_j \in \kappa_k^\times$ for $j \in J \bs J_0$ with $J_0 = \{1, \dots, m\}$ a finite subset.  By the functoriality of $p$-basis (Corollary~\ref{C:k=k((s))-change}), we may change the isomorphism $\kappa_l((t)) \simeq l$ so that $\bar c_{J \bs J_0}$ are sent to elements in $\calO_k^\times$.  Let $c_J$ denote the images of $\bar c_J$ under the above isomorphism.

As a consequence, $\serie{c_}m$ and $t$ generate $\calO_l$ over $\calO_k$. More precisely,
$$
\big\{c_J^{e_J}t^i\,|\,e_j \in \{0, \dots, p^{r_j}-1\} \textrm{ for }j = \serie{}m \textrm{, and } i \in \{0, \dots, e-1\}\big\}
$$
form a basis of $\calO_l$ as a finite free $\calO_k$-module.  They also form a basis of $l$ as a $k$-vector space.
\end{construction}

\begin{remark} \label{R:Sweedler}
It is attractive to hope that we can find a $p$-basis $(\bar b_j)_{j \in J}$ of $\kappa_k$ so that $\kappa_l = \kappa_k (\bar b_j^{p^{-r_j}})$ for some $r_j \in \ZZ_{\geq 0}$.  This however is false in general, as pointed out to the author by Shun Ohkubo.  In fact, Sweedler \cite{Sweedler-insep-extn} studied this phenomenon and called the above case \emph{modular}.  He also gave the following non-modular example \cite[Example~1.1]{Sweedler-insep-extn}.

Let $\kappa_0$ be a perfect field of characteristic $p$ and let $X, Y, Z$ be indeterminants.  Let $\kappa_k = \kappa_0(X^p, Y^p, Z^{p^2})$ and $\kappa_l = \kappa_k(Z, XY + Z)$.  Then $[\kappa_l : \kappa_l \cap \kappa_k^{p^{-1}}] = p^2$ and $[ \kappa_l \cap \kappa_k^{p^{-1}}: \kappa_k ] = p$.  Hence, $\kappa_l / \kappa_k$ cannot be modular.
\end{remark}

Now, we continue to assume Hypothesis~\ref{H:J-finite-set}.

\begin{notation}
For a nonarchimedean ring $R$, we use $R \langle u_0, \dots, u_m \rangle$ to denote the completion of $R [ u_0, \dots, u_m]$ with respect to the natural topology induced from $R$.  When $R = F$ is a complete nonarchimedean field, $F \langle u_0, \dots, u_m \rangle$ is the ring of analytic functions on the unit polydisc $A_F^{m+1}[0,1]$.
\end{notation}

\begin{notation}\label{N:generators-of-calI}
Let $\calO_k \langle\seriezero{u_}m \rangle / \calI \isom \calO_l$ be the homomorphism sending $u_0$ to $t$ and $u_j$ to $c_j$ for $j \in J$.  We choose a set of generators $\seriezero{p_}m$ of $\calI$ as follow:  for $t^e$ or each $c_j^{p^{r_j}}$, one can write it in terms of the basis of $\calO_l$ over $\calO_k$ listed in Construction~\ref{C:cj-generate-l}. This will give us an element $p_0$ or $p_j$ in $\calI$. Obviously, $\seriezero{p_}m$ generate $\calI$.  Moreover,
\begin{eqnarray*}
p_0 & \in & u_0^e - \gothd s + (u_0s, s^2) \cdot \calO_k [\seriezero{u_}m]; \\
p_j & \in & u_j^{p^{r_j}} - \gothb_j + (u_0, s) \cdot \calO_k [\seriezero{u_}m], \textrm{ for }  j = \serie{}m,
\end{eqnarray*}
where $\gothb_j$ is a polynomial in $u_1, \dots, u_{j-1}$ with coefficients in $\calO_k$ and with degree on $u_{j'}$ strictly smaller than $p^{r_{j'}}$ for $j' = 1, \dots, j-1$, and $\gothd \in \calO_k[u_J]$ such that $\gothd(c_1, \dots, c_m) \in \calO_L^\times$.  Denote $\bar \gothb_j$ to be the reduction of $\gothb_j$ in $\kappa_k [u_1, \dots, u_{j-1}]$.
\end{notation}

\begin{remark}
It was pointed out to the author by Shun Ohkubo that introducing $\gothd$ is necessary.  In general, one may not be able to  find uniformizers $s$ and $t$ of $k$ and $l$, respectively, such that $t^e \equiv s \mod t^{e+1}\calO_l$.  Ohkubo provided the following counterexample.  
Although we do not know if there is a counterexample for which $L/K$ is Galois, we still try to avoid this restriction.
\end{remark}

\begin{example}
Let $k$ be a complete discretely valued field with non-perfect residue field $\kappa_k$.  Let $b \in \calO_k$ be such that $\bar b \in \kappa_k \bs \kappa_k^p$.  Choose $\alpha, \beta \in \overline K$ as follows: let $\alpha$ be a root of the polynomial $X^p +s X + b \in k[X]$ and $\beta$ a root of the polynomial $Y^p + s Y + s \alpha \in k(\alpha)[Y]$.  Denote $l = k(\alpha, \beta)$.  Then $l/k$ is a separable extension of degree $p^2$ with na\"ive ramification degree $p$.  The ring of integers of $k(\alpha)$ and $k(\alpha, \beta)$ are $\calO_k[\alpha]$ and $\calO_k[\alpha, \beta]$, respectively.  We claim that we cannot choose uniformizers $t$ and $s$ so that $t^p / s \equiv 1 \mod \gothm_l$.

It is clear that $\beta$ is a uniformizer of $l$.  For any uniformizer $t$ of $l$,
\[
\xymatrix{
\dfrac{t^p} {s} = \dfrac{\beta^p} {s}  \big( \dfrac{t} \beta \big)^p\in (-\alpha - \beta) (\calO_l^\times)^p \ar[rr]^-{\pmod {\gothm_l}} &&   (-\alpha)\kappa_l^p \subset \kappa_l. 
}
\]
In particular, $t^p /s$ is not congruent to 1 modulo $ \gothm_l$.
\end{example}

\begin{remark}
It is generally false that $\seriezero{p_}m$ generate $\Ker(\calO_k [\seriezero{u_}m] \rar \calO_l)$; this will not matter since we take $a>0$ and $a>1$ in Definition~\ref{D:thickening-space}.
\end{remark}

\begin{construction}\label{C:map-C_k-to-C_l}
For each $j \in J$, fix an element in $\calO_L \llbracket T \rrbracket$ lifting $b_j \in \calO_k \subset \kappa_l \llbracket t \rrbracket$ for $j \in J$ and fix an element in $T^e + T^{e+1} \calO_L \llbracket T \rrbracket$ lifting $s \in \calO_k \subset \kappa_l \llbracket t \rrbracket$.  By Proposition~\ref{P:cohen-functorial}, there exists a continuous homomorphism $f^*: C_k \inj C_l$ sending $B_J$ and $S$ to the elements chosen above; it naturally restricts to $f^*: \OK \llbracket S \rrbracket \inj \calO_L \llbracket T \rrbracket$.
\end{construction}

The proof of the following lemma is not enlightening.  The reader may skip it when reading the paper for the first time.  The upshot is that we can turn the ``good" generators and relations of $\calO_l$ as an $\calO_k$-algebra into ``good" generators and relations of $\calR_L^{\eta_0^{1/e}}$ as an $\calR_K^{\eta_0}$-algebra.

\begin{lemma}\label{L:base-change-map}
Keep the notation as above.

\emph{(1)} The homomorphism $f^*$ is finite, and  $\serie{C_}m$ and $T$ generate $\calO_L \llbracket T \rrbracket$ over $\OK \llbracket S \rrbracket$.  Hence, $f^*$ induces a surjective map $\OK \llbracket S \rrbracket \langle \seriezero{U_}m \rangle \surj \calO_L \llbracket T \rrbracket$ sending $U_0$ to $T$ and $U_j$ to $C_j$ for $j \in J$.  Moreover, one can choose generators $\seriezero{P_}m$ of the kernel so that, modulo $p$, they are exactly $p_{J^+}$ in Notation~\ref{N:generators-of-calI}.  In particular, 
\begin{eqnarray*}
P_0 & \in & U_0^e - \gothD S + (p, U_0S, S^2) \cdot \OK \llbracket S \rrbracket \langle \seriezero{U_}m \rangle, \\
P_j & \in & U_j^{p^{r_j}} - \gothB_j + (p, U_0, S) \cdot \OK \llbracket S \rrbracket \langle \seriezero{U_}m \rangle,
\end{eqnarray*}
where $\gothB_j$ is a polynomial in $U_1, \dots, U_{j-1}$ with coefficients in $\calO_K$ and with degree on $U_{j'}$ strictly smaller than $p^{r_{j'}}$ for $j' = 1, \dots, j-1$, and $\gothD \in \calO_K[U_J]$ lifts $\gothd$.  Moreover, $\big\{ U_{J^+}^{e_{J^+}} \big|\, 0 \leq e_0 < e;\, 0 \leq e_j < p^{r_j}, j \in J \big\}$ form a basis of $\calO_K \llbracket S \rrbracket \langle U_0, \dots, U_m \rangle / (P_{J^+})$ over $\calO_K \llbracket S \rrbracket$.

\emph{(2)} The map $f^*$ extends to a map $f_\eta^*: K \langle \eta / S, S / \eta \rangle \rar L \langle \eta^{1/e} / T, T / \eta^{1/e} \rangle$ for $\eta \in [0, 1)$.  Thus $f^*$ extends by continuity to a homomorphism $f^*: \calR_K^{\eta_0} \rar \calR_L^{\eta_0^{1/e}}$, or in geometric notation, $f: A_K^1[\eta_0, 1) \rar A_L^1[\eta_0^{1/e}, 1)$ for $\eta_0 \in (0, 1)$.

\emph{(3)} Let $\Gamma_K^\dag$ and $\Gamma_L^\dag$ be the integral Robba rings over $K$ and $L$, respectively, similarly constructed as in Construction~\ref{Cstr:repn->bounded-module} but without tensoring with $F$.  Let $\calR_L$ be the Robba ring over $L$ as in Notation~\ref{N:Robba-ring-radius-eta}.  Then $\Gamma_L^\dag$ is a finite \'etale extension of $\Gamma_K^\dag$ with Galois group $G_{l/k}$.  Moreover, $\calR_L \simeq \Gamma_L^\dag \otimes_{\Gamma_K^\dag} \calR_K$.

\emph{(4)} For some $\eta_0 \in (0,1)$, $A_L^1[\eta_0^{1/e}, 1)$ is Galois \'etale over $\eta \in [\eta_0, 1)$ via $f^*$ with Galois group $G_{l/k}$.  Hence, $\calR_L^{\eta_0^{1/e}}$ becomes a regular $G_{l/k}$-representation over $\calR_K^{\eta_0}$ via $f^*$.
\end{lemma}

\begin{proof}
(1) is equivalent to its mod $p$ version, which is exactly Construction~\ref{C:cj-generate-l}.

(2) It suffices to prove that $f^*$ is continuous respecting the norms $|\cdot|_{Z_k^\eta}$ on $C_k$ and $|\cdot|_{Z_l^{\eta^{1/e}}}$ on $C_l$, for all $\eta \in [\eta_0, 1)$.  Since $f^*(\OK) \in \calO_L \llbracket T \rrbracket$ and $f^*(S) \in T^e + T^{e+1}\calO_L \llbracket T \rrbracket$, we have $|g|_{Z_k^\eta} = |f^*(g)|_{Z_l^{\eta^{1/e}}}$ for any $g \in C_k$.  Hence the map $f^*$ extends continuously to $f^*_\eta: K \langle \eta / S, S / \eta \rangle \rar L \langle \eta^{1/e} / T, T / \eta^{1/e} \rangle$.

(3) The first statement follows from Lemma~\ref{L:bounded-Robba-henselian}.  The second statement is true because  $\Gamma_L^\dag \otimes_{\Gamma^\dag} \calR_K$ is complete and dense in $\calR_L$.

(4) It follows from (2) and (3) since $\calR_K$ (resp. $\calR_L$) is a limit of $\calR_K^{\eta_0}$ (resp. $\calR_L^{\eta_0^{1/e}}$).
\end{proof}

\begin{remark}\label{R:naive-K-structure}
The homomorphism $f^*$ does not respect the na\"ive $K$-algebra structure on $\calR_L^{\eta_0^{1/e}}$; this is precisely because of Caution~\ref{Cau:residue-extn-not-embedable}.  But it respects the $K$-algebra structure on $\calR_L^{\eta_0^{1/e}}$ induced by $\OK \inj \calO_K \llbracket S \rrbracket \stackrel {f^*} \rar \calO_L \llbracket T \rrbracket$.  So, it might be better not to view $Z_l^{\geq \eta_0^{1/e}} \rar Z_k^{\geq \eta_0}$ as a morphism between rigid spaces, but the geometric incarnation of $f^*$.
\end{remark}

\begin{construction} \label{Cstr:calE}
Keep the notation as in Construction~\ref{Cstr:bounded->R_K}.  Let $\rho: G_{l/k} \rar GL(V_\rho)$ be a $p$-adic representation, where $V_\rho$ is a finite dimensional vector space over $F$.  We have
\[
\calE_\rho = D^\dag(\rho) \otimes_{\Gamma_K^\dag} \calR_K =
 \big( V_\rho \otimes_\calO \tilde \Gamma^\dag \big)^{G_k} \otimes_{\Gamma_K^\dag} \calR_K =
\big( V_\rho \otimes_{\ZZ_q} \Gamma_L^\dag \big)^{G_{l/k}} \otimes_{\Gamma_K^\dag} \calR_K
= \big( V_\rho \otimes_{\ZZ_q} \calR_L \big)^{G_{l/k}}.
\]
Hence, for some $\eta_0 \in (0, 1)$, the differential module $\calE_\rho$ descents to
\[
\calE_\rho = \Big( V_\rho \otimes_{\ZZ_q} f_* \calO_{Z_l^{\geq \eta_0^{1/e}}}\Big)^{G_{l/k}};
\]
this is a differential module over $\calR_K^{\eta_0} \otimes_{\QQ_q} F = \calR_{K'}^{\eta_0}$ relative to $K_0$.  This construction respect tensor products, i.e., given another $p$-adic representation $\rho'$ of $G_{l/k}$ over $F$, we have
\[\calE_{\rho \otimes \rho'} = \calE_\rho \otimes_{\calR_{K'}^{\eta_0}} \calE_{\rho'}
\]
\end{construction}

\begin{hypo} \label{H:eta_0}
From now on, we always assume that $\eta_0 \in (0,1)$ is close enough to $1^-$ so that all statements in Lemma~\ref{L:base-change-map} hold and $\calE_\rho$ descents to $\calR_{K'}^{\eta_0}$.
\end{hypo}

\subsection{Spectral norms and connected components of thickening spaces}
\label{S:calF}

In this subsection, we relate the spectral norms of differential operators on $\calE$ to the connected components of certain rigid spaces.

We keep Hypotheses~\ref{H:J-finite-set}, \ref{H:tot-ram}, and \ref{H:eta_0} in this subsection.

\begin{definition}\label{D:tilde-TS}
Let $a \in \QQ_{>1}$.  We define
\begin{align*}
\calO_{TS_{l/k}^{a, \geq \eta_0}} &= \calR_L^{\eta_0^{1/e}} \otimes_{f^*, \calR_K^{\eta_0}, \pi^*} \calO_{TS_k^{a, \geq \eta_0}} \\
\calO_{TS_{k\backslash l}^{a, \geq \eta_0}} &= \calO_{TS_k^{a, \geq \eta_0}} \otimes_{\tilde \pi^*, \calR_K^{\eta_0}, f^*} \calR_L^{\eta_0^{1/e}}\\
\calO_{TS_{l/k \backslash l}^{a, \geq \eta_0}} &= \calR_L^{\eta_0^{1/e}} \otimes_{f^*, \calR_K^{\eta_0}, \pi^*} \calO_{TS_k^{a, \geq \eta_0}} \otimes_{\tilde \pi^*, \calR_K^{\eta_0}, f^*} \calR_L^{\eta_0^{1/e}}
\end{align*}
Here we do not have to complete the tensor products because $f^*$ is finite.
(We intentionally put the tensor products on different sides so that it is easy to distinguish the two base changes by $f^*$ through $\pi^*$ and $\tilde \pi^*$ respectively.)
Let $TS_{l/k}^{a, \geq \eta_0}$, $TS_{k\backslash l}^{a, \geq \eta_0}$, and $TS_{l/k \backslash l}^{a, \geq \eta_0}$ denote the geometric incarnations of these rings respectively.  We have formally the following Cartesian diagram.
\begin{equation}\label{E:diagram-ts-space}
\xymatrix{
Z_l^{\geq \eta_0^{1/e}} \ar[d]^f & TS_{l/k}^{a, \geq \eta_0} \ar[d]^{f \times 1} \ar[l]_-{1 \times \pi} & \ar[l]_{\tilde f} \ar[d] TS_{l/k\backslash l}^{a, \geq \eta_0}\\
Z_k^{\geq \eta_0} & TS_k^{a, \geq \eta_0} \ar[l]_-\pi \ar[d]^{\tilde \pi} & \ar[l]_-{1 \times f} TS_{k \backslash l}^{a, \geq \eta_0} \ar[d] \\
& Z_k^{\geq \eta_0} & \ar[l]_-f Z_l^{\geq \eta_0^{1/e}}
}
\end{equation}

We make similar constructions for the logarithmic version of all spaces if $a \in \QQ_{>0}$.
\end{definition}

\begin{remark}
The morphisms $\pi$ and $1 \times f$ are genuine morphisms between rigid spaces over $Z_k^{\geq \eta_0}$ and $\tilde f$ and $1 \times \pi$ are genuine morphisms between rigid spaces over $Z_l^{\geq \eta_0^{1/e}}$; this is because the rigid space structures on thickening spaces are given by the projection $\pi$ and $1 \times \pi$, respectively.  In contrast, all the vertical arrows in \eqref{E:diagram-ts-space}
should all be thought of as just geometric incarnations for the corresponding ring homomorphisms.
\end{remark}

\begin{remark}
The na\"ive base change $f \times 1$ helps to realize geometric connected components as connected components (see Theorem~\ref{T:ram-break-equiv-disconn}).  The base change $\tilde f$ (and also $1 \times f$) encodes the ramification information, which is what we are interested in.
\end{remark}

\begin{remark}
One may want to relate $TS_{l / k\backslash l}^{a, \geq \eta_0}$ to the thickening space of $Z_l^{\geq \eta_0^{1/e}}$.  However, it is not clear how to compare the levels or radii of the two spaces.  We do not need this result in our paper.
\end{remark}

\begin{corollary}
The space $TS_{l/k\backslash l}^{a, \geq \eta_0}$ admits an action of $G_{l/k}$ by morphisms between $K$-rigid spaces, obtained by pulling back the action on $Z_l^{\geq \eta_0^{1/e}}$ over $Z_k^{\geq \eta_0}$ via $\tilde \pi \circ (f \times 1)$.  Under this action, $\tilde f_* \calO_{TS_{l/k\backslash l}^{a, \geq \eta_0}}$ is a regular representation of $G_{l/k}$ over $\calO_{TS_{l/k}^{a, \geq \eta_0}}$.  For a $p$-adic representation $\rho$ of $G_{l/k}$ over $F$, define
\[
\widetilde \calF_\rho = \big(V_\rho \otimes_{\QQ_q} \tilde f_* \calO_{TS_{l/k\backslash l}^{a, \geq \eta_0}} \big)^{G_{l/k}};
\]
this is a differential module over $TS_{l/k}^{a, \geq \eta_0} \times_{\QQ_q} F$ relative to $Z_l^{\geq \eta_0^{1/e}} \times_{\QQ_q} F$.  Moreover, we have $\widetilde \calF_\rho \simeq (f \times 1)^* \tilde \pi^* \calE_\rho$.

The same statement also holds for log-space.
\end{corollary}
\begin{proof}
The differential module structure on $\tilde f_* \calO_{TS_{l/k\backslash l}^{a, \geq \eta_0}}$ is given by the composition of natural homomorphisms
\[
\tilde f_* \calO_{TS_{l/k\backslash l}^{a, \geq \eta_0}} \lrar \tilde f_* \Big(\Omega^1_{TS_{l/k\backslash l}^{a, \geq \eta_0} \big/ Z_l^{\geq \eta_0^{1/e}}} \Big) \simeq \tilde f_* \Big(\tilde f^* \Big(\Omega^1_{TS_{l/k}^{a, \geq \eta_0} \big/ Z_l^{\geq \eta_0^{1/e}}} \Big)\Big) \simeq \tilde f_* \calO_{TS_{l/k\backslash l}^{a, \geq \eta_0}} \otimes \Omega^1_{TS_{l/k}^{a, \geq \eta_0} \big/ Z_l^{\geq \eta_0^{1/e}}}.
\]
(In fact this construction works for any finite \'etale morphisms.)
The statement of the corollary is an easy consequence of flat base change for the two Cartesian squares on the right in Diagram~\eqref{E:diagram-ts-space}.
\end{proof}

\begin{notation}
We may view $1 \times \pi:TS_{l/k}^{a, \eta_0} \rar Z_l^{\eta_0^{1/e}}$ as bundles, whose fibers are polydiscs (of different radii) with parameters $\delta_0, \dots, \delta_m$; again this morphism is a genuine morphism between rigid spaces.  By \emph{zero section} $\mathbf Z$, we mean the natural closed subspace of this bundle defined by $\delta_0 = 0, \dots, \delta_m = 0$.
\end{notation}

\begin{notation}
\label{N:Taylor-series}
Let $M$ be a differential module over a differential ring $R$ with derivatives $\partial_1, \dots, \partial_n$.  For $x \in M$ and $r_1, \dots, r_n \in R$, we define the \emph{Taylor series} 
\[
\TT(x; \partial_1, \dots, \partial_n; r_1, \dots, r_n) = \sum_{\alpha_1, \dots, \alpha_n \in \ZZ_{\geq 0}} \frac{r_1^{\alpha_1}\cdots r_n^{\alpha_n}\partial_1^{\alpha_1} \cdots \partial_n^{\alpha_n}}{n!}(x)
\]
if it converges.  If $x \in R$, we have $\TT(ax; \partial_1, \dots, \partial_n; r_1, \dots, r_n) = \TT(a; \partial_1, \dots, \partial_n; r_1, \dots, r_n) \cdot \TT(x; \partial_1, \dots, \partial_n; r_1, \dots, r_n)$ if all terms converge.
\end{notation}

\begin{notation}
Let $M$ be a differential module over a differential ring $R$ with derivatives $\partial_1, \dots, \partial_n$.  We denote $H^0_\nabla(R, M) = \{x \in M \;|\; \partial_i(x) = 0, i = 1, \dots, n\}$ to be the set of horizontal sections of $M$ over $R$.  In particular, if $r_1, \dots, r_n \in R$ be elements such that $\partial_i(r_j) = 1$ if $i = j$ and $0$ otherwise, then an elementary calculation shows that the Taylor series $\TT(x; \partial_i, \dots, \partial_n; r_1, \dots, r_n)$ is an element in $H^0_\nabla(R, M)$ for any $x \in M$ such that the Taylor series converges.

We usually use the geometric counterparts in places of $R$ and $M$ in the notation, e.g., we write $H^0_\nabla(\Max(R), M)$ if $R$ is an affinoid algebra.
\end{notation}

The following lemma is known as the Dwork's transfer theorem, which will be frequently used in proving the theorem below.  It works in greater generality, but we content ourselves with this special case.

\begin{lemma}[Dwork]
\label{L:Dwork-transfer}
Let $a>1$.
Let $\widetilde \calF$ be a differential module over $TS_{l/k}^{a, \geq \eta_0}$ relative to $Z_l^{\geq \eta_0^{1/e}}$.  Assume $|\partial / \partial\delta_i|_{\widetilde \calF_{a, \eta}} \leq p^{-1/(p-1)} \eta^{-a}$ for all $j \in J$ and $\eta \in [\eta_0, 1)$.  Then for any rational number $c>a$, the natural homomorphism of finite $\calR_L^{\eta_0^{1/e}}$-modules
\begin{equation}\label{E:Dwork}
\Theta: H^0_\nabla (TS_{l/k}^{c, \geq \eta_0}, \widetilde \calF ) \isom \Gamma( \mathbf Z, \widetilde \calF | _{\mathbf Z} )
\end{equation}
is an isomorphism.  In particular, $\widetilde \calF$ is a trivial $\nabla$-module relative to $Z_l^{\geq \eta_0^{1/e}}$. 
The same statement is also true if we base change everything to $F$ over $\QQ_q$.  Moreover, when $\widetilde \calF = f_*\calO_{TS_{l/k\bs l}^{a, \geq \eta_0}}$, $\Theta$ induces a \emph{ring} homomorphism for any rational number $c>a$
\[
\Gamma\Big( \mathbf Z, f_* \calO_{TS_{l/k \backslash l}^{c, \geq \eta_0}} \Big| _{\mathbf Z} \Big)  \stackrel {\stackrel{\Theta}\sim}\longleftarrow H^0_\nabla \Big(TS_{l/k}^{c, \geq \eta_0}, f_* \calO_{TS_{l/k \backslash l}^{c, \geq \eta_0}} \Big) \inj \Gamma \Big(TS_{l/k \backslash l}^{c, \geq \eta_0}, \calO_{TS_{l/k \backslash l}^{c, \geq \eta_0}} \Big).
\]

The same statements hold for the log version with $a>0$ and inserting subscript log appropriately.
\end{lemma}
\begin{proof}
The proof for the log version goes verbatim the same, except inserting subscript log appropriately and for the $\delta_0$ coordinate, increasing the exponents on $\eta$ by 1.  The proof for the tensor $F$ version goes  word by word the same except we need to tensor $F$ everywhere.

Now, we prove the lemma for the nonlog case over $\QQ_q$.  In fact we may define an inverse of to the map $\Theta$ using Taylor series:
\[
\Theta^{-1} (x) = \TT(\tilde x; \partial / \partial \delta_0, \dots, \partial / \partial \delta_m; \delta_0,\dots,  \delta_m)
\]
for $x \in \Gamma(\mathbf Z, \widetilde \calF |_{\mathbf Z})$, where $\tilde x$ is a lift of $x$ in $\Gamma (TS_{l/k}^{c, \geq \eta_0}, \widetilde \calF)$.  The Taylor series converges over $TS_{l/k}^{c, \geq \eta_0}$ by the condition $|\partial / \partial\delta_i|_{\widetilde \calF_{a, \eta}} \leq p^{-1/(p-1)} \eta^{-a} < p^{-1/(p-1)} \eta^{-c}$ for all $j \in J$ and $\eta \in [\eta_0, 1)$.  Moreover, the Taylor series converges to a horizontal section in $H^0_\nabla(TS_{l/k}^{c, \geq \eta_0}, \widetilde \calF )$.

When $\widetilde \calF = f_*\calO_{TS_{l/k\bs l}^{a, \geq \eta_0}}$, $\Theta$ is a homomorphism, which can be also seen from the fact that Taylor series gives a ring homomorphism of rings (see Notation~\ref{N:Taylor-series}).
\end{proof}

The following theorem is one of the key steps of the proof of Hasse-Arf theorem.  This is the main ingredient (a) described in the introduction.  It allows us to compare the differential ramification breaks with the geometric connected components of the thickening spaces; we will  later identify the thickening spaces with the lifts of the Abbes-Saito spaces (Theorem~\ref{T:AS=TS}).

\begin{theorem}\label{T:ram-break-equiv-disconn}
Let $\rho: G_{l/k} \rar GL(V_\rho)$ be a \emph{faithful} $p$-adic representation over $F$ with $l/k$ satisfying the Hypotheses~\ref{H:J-finite-set} and \ref{H:tot-ram}. Then, for $b>1$, the following conditions are equivalent:

\emph{(1)} $\rho$ has differential ramification break $\leq b$.

\emph{(2)} For any rational number $c>b$, when $\eta_0 \rar 1^-$, $\widetilde \calF= \widetilde \calF_\rho$ is a trivial $\nabla$-module over $TS_{l/k}^{c, \geq \eta_0} \times_{\QQ_q} F$ relative to $Z_l^{\geq \eta_0^{1/e}} \times_{\QQ_q} F$.

\emph{(3)} For any rational number $c> b$, when $\eta_0 \rar 1^-$, $TS_{l/k \backslash l}^{c, \geq \eta_0}$ has exactly $[l:k]$ connected components.

\emph{(4)} For any rational number $c> b$, when $\eta_0 \rar 1^-$, $Z_{l'}^{\geq \eta_0^{1/ e'}} \times_{Z_l^{\geq \eta_0^{1/e}}, \pi} TS_{l/k \backslash l}^{c, \geq \eta_0}$ has exactly $[l:k]$ connected components for some finite extension $l'/l$, where $e'$ is the na\"ive ramification degree of $l'/k$.

Also, the similar conditions for logarithmic spaces are equivalent provided that $b>0$.
\end{theorem}

\begin{proof}
We will only prove the statement for the non-logarithmic spaces and the statement for the logarithmic spaces can be proved completely in the same way by adding the subscript $\log$ and changing the scales on $\partial_0$ and $\partial / \partial \delta_0$ from $\eta^b$ to $\eta^{b+1}$ and $\eta^c$ to $\eta^{c+1}$.

We make a general remark that the Proposition~\ref{P:sp-norm-on-TS} is unchanged if we replace $\calF$ by $\widetilde \calF$ as the spectral norms are invariant under scalar extensions.

We first prove the equivalence between (1) and (2).  We first assume (1) $\rho$ has differential ramification break $\leq b$.  By Definition~\ref{D:differential-break}, for $\eta_0$ sufficiently close to $1^-$, the generic radius of $\calE_\rho$ satisfies $\bbT(\calE_\rho, \eta) \geq \eta^b$ for $\eta \in [\eta_0, 1)$, or equivalently $|\partial_j|_{\calE_{\rho, \eta}, \sp} \leq p^{-1/(p-1)}\eta^{-b}$ for any $j \in J^+$ and $\eta \in [\eta_0, 1)$. Then  Proposition~\ref{P:sp-norm-on-TS} and Remark~\ref{R:sp-norm-inv-BC} imply that for all $\eta \in [\eta_0, 1)$, $|\partial / \partial \delta_j|_{\widetilde \calF_{b, \eta}, \sp} \leq p^{-1/(p-1)} \eta^{-b}$, and hence $\widetilde \calF_\rho$ is a trivial differential module over $TS_{l/k}^{c, \geq \eta_0} \times_{\QQ_q} F$ relative to $Z_l^{\geq \eta_0^{1/e}}\times_{\QQ_q} F$ for any rational number $c >b$ by Dwork's transfer Lemma~\ref{L:Dwork-transfer}. This proves (2).

Now, we assume (2), i.e., $\widetilde \calF_\rho$ is trivial over $TS_{l/k}^{c, \geq \eta_0} \times_{\QQ_q} F$ relative to $Z_l^{\geq \eta_0^{1/e}}\times_{\QQ_q} F$ for any rational number $c>b$ and some $\eta_0 \in (0, 1)$.  It follows that $|\partial / \partial \delta_j|_{\widetilde \calF_{c, \eta}, \sp} = |\partial / \partial \delta_j|_{\Frac(\calO_{TS_{l/k}^{c, \eta}})^\wedge, \sp} = p^{-1/(p-1)} \eta^{-c}$. By Proposition~\ref{P:sp-norm-on-TS}, $|\partial_j|_{\sp, \calE_\eta} \leq p^{-1/(p-1)}\eta^{-c}$, for any $j \in J^+$, $\eta \in [\eta_0, 1)$, and $c \in \QQ_{>b}$.  By Definition~\ref{D:differential-break}, this implies that the differential ramification break $\leq b$, as rational numbers are dense in the real numbers.

Obviously, (3)$\Rar$(2).  To see the converse, we first claim that for any rational number $c>b$, $f_*\calO_{TS_{l/k\bs l}^{c, \geq \eta_0}}$ is a trivial differential module over $TS_{l/k}^{c, \geq \eta_0}$ relative to $Z_l^{\geq \eta_0^{1/e}}$.  Indeed, for a rational number $c' \in (b, c)$, we know that $\widetilde \calF_\rho$ is a trivial differential module over $TS_{l/k}^{c', \geq \eta_0} \times_{\QQ_q} F$ relative to $Z_l^{\geq \eta_0^{1/e}}\times_{\QQ_q} F$, then for any $n \in \NN$, $\widetilde \calF^{\otimes n}_\rho$ is also a trivial differential module (relative to $Z_l^{\geq \eta_0^{1/e}}\times_{\QQ_q} F$), which corresponds to $V_\rho^{\otimes n}$ by functoriality (Construction~\ref{Cstr:calE}).  By Lemma~\ref{L:tannakian-cat} below from the theory of representations of finite groups (or standard Tannakian arguments), the differential module
\begin{equation}
\label{E:homo-diff-module-not-G}
\Big( F[G_{l/k}] \otimes_{\QQ_q} f_* \calO_{TS_{l/k \backslash l}^{c', \geq \eta_0}} \Big)^{G_{l/k}}  \stackrel \simeq \lrar  F \otimes_{\QQ_q} f_* \calO_{TS_{l/k \backslash l}^{c', \geq \eta_0}}
\end{equation}
corresponding to the regular representation
is a direct summand of a direct sum of some $\widetilde \calF^{\otimes n}_\rho$'s and hence is a trivial differential module  (relative to $Z_l^{\geq \eta_0^{1/e}}\times_{\QQ_q} F$).  To make it perfectly rigorous, here the isomorphism \eqref{E:homo-diff-module-not-G} of differential modules is given by $
\sum_{g \in G_{l/k}} \bbf g \otimes g\bbv  \longmapsto  \bbf \cdot \bbv$, 
where $\bbf \in F$ and $\bbv \in f_* \calO_{TS_{l/k \backslash l}^{c, \geq \eta_0}}$; this map does \emph{not} respect the $F[G_{l/k}]$-module structures.  We would have finished the proof of the claim if $F = \Qp$.  If $F \neq \Qp$, we observe that know that, for all $j \in J^+$, the spectral norms of $\partial / \partial \delta_j$ at radius $\eta$ on the right hand side of \eqref{E:homo-diff-module-not-G} are $p^{-1/(p-1)} \eta^{-c'}$, which equal the spectral norms of $\partial / \partial \delta_j$ on $f_* \calO_{TS_{l/k \backslash l}^{c', \geq \eta_0}}$ at radius $\eta$.  By Dwork's transfer Lemma~\ref{L:Dwork-transfer}, we finish the proof of the claim.

We now apply the second part of Lemma~\ref{L:Dwork-transfer} and obtain, for any rational numbers $c'>c$, a \emph{ring} homomorphism
\begin{equation}\label{E:isom-horizontal-section}
\Gamma\Big( \mathbf Z, f_* \calO_{TS_{l/k \backslash l}^{c', \geq \eta_0}} \Big| _{\mathbf Z} \Big)  \stackrel {\stackrel{\Theta}\sim}\longleftarrow H^0_\nabla \Big(TS_{l/k}^{c', \geq \eta_0}, f_* \calO_{TS_{l/k \backslash l}^{c', \geq \eta_0}} \Big) \inj \Gamma \Big(TS_{l/k \backslash l}^{c', \geq \eta_0}, \calO_{TS_{l/k \backslash l}^{c', \geq \eta_0}} \Big).
\end{equation}
The key is that the left hand side of \eqref{E:isom-horizontal-section} is isomorphic to the ring functions on $Z_l^{\geq \eta_0^{1/e}} \times _{Z_k^{\geq \eta_0}} Z_l^{\geq \eta_0^{1/e}}$ because the restrictions of $\tilde \pi$ and $\pi$ to $\mathbf Z$ are both the same as $f$.  Moreover, since $Z_l^{\geq \eta_0^{1/e}}$ is finite \'etale Galois over $Z_k^{\geq \eta_0}$ (Lemma~\ref{L:base-change-map}), $Z_l^{\geq \eta_0^{1/e}} \times _{Z_k^{\geq \eta_0}} Z_l^{\geq \eta_0^{1/e}} = \coprod_{g \in G_{l/k}} Z_l^{\geq \eta_0^{1/e}}$.  In particular, we have fundamental idempotent elements in $\Gamma\Big( \mathbf Z, f_* \calO_{TS_{l/k \backslash l}^{c', \geq \eta_0}} \Big| _{\mathbf Z} \Big)$ corresponding to each connected components.  Via the composition of the homomorphisms in \eqref{E:isom-horizontal-section}, we can ``lift" the idempotent elements on $Z_l^{\geq \eta_0^{1/e}} \times _{Z_k^{\geq \eta_0}} Z_l^{\geq \eta_0^{1/e}}$ to idempotent elements in $\calO_{TS_{l/k \backslash l}^{c', \geq \eta_0}}$.  This shows that $TS_{l/k \backslash l}^{c', \geq \eta_0}$ has at least $[l:k]$ connected components.  But this space is finite and flat of degree $[l:k]$ over an irreducible rigid space $TS_{l/k}^{c', \geq \eta_0}$; it can have at most $[l:k]$ connected components.  Therefore, (3) holds.

The equivalence between (2) and (4) can be proved similarly, using a version of Lemma~\ref{L:Dwork-transfer} over $Z_{l'}^{\eta_0^{1/e'}}$.  The upshot here is that we need base change to at least $Z_l^{\geq \eta_0^{1/e}}$ in (3) so that we can split the fiber over $\mathbf Z$; this is why we did not state the theorem for $TS_{k\bs l}^{c, \geq \eta_0}$ and $\widetilde \calF$ themselves.
\end{proof}

\begin{remark}
The faithfulness condition on $\rho$ in the theorem is harmless as we will very easily reduce to this case later in the proof of Theorem~\ref{T:main-theorem}.
\end{remark}

\begin{lemma}\label{L:tannakian-cat}
Let $G$ be a finite group and $F$ be a field of characteristic $0$.  Let $\rho: G \rar \GL(V_\rho)$ be a faithful representation over $F$.  Then the regular representation $F[G]$ is a direct summand of a direct sum of some self-tensor products of $V_\rho$.
\end{lemma}
This is an easy exercise of finite group representations but we do not know a good reference.  The author would like to thank Xuhua He for providing the following proof.
\begin{proof}
Let $\chi$ be the character of $V_\rho$ and let $d$ be the dimension of $V_\rho$.  Since the representation is injective, $\chi(1) = d$ and $\chi(g) \neq d$ for all $g \in G$ nontrivial.  (This is because all the eigenvalues of $\rho(g)$ are roots of unity and cannot all be 1.)

Therefore, for each $g \neq 1$ there exists a polynomial $P_g$ in $\chi$ with integer coefficients such that $P_g(\chi(g)) = 0$ but $P_g(d) \neq 0$.  Let $P = \prod_{1 \neq g \in G} P_g$ and then $P(d) \neq 0$ but $P(\chi(g)) = 0$ for all $g \neq 1$.  Moreover, by multiplying a constant, we may assume that $\#G$ divides $P(d)$ and $P(d) >0$.  If $P(X) = a_n X^n + \cdots + a_0 \in \ZZ [ X]$, then $(V^{\otimes n})^{\oplus a_n} \oplus \cdots \oplus V^{\oplus a_1} \oplus \boldsymbol 1_F^{\oplus a_0} = F[G]^{P(d) / \#G}$ in the Grothendieck group of the representations of $G$, where $\boldsymbol 1_F$ denote the trivial representation.  As a consequence, if we take the direct sum of terms on the left hand side with positive $a_i$, the regular representation will be a natural direct summand of it.
\end{proof}

\section{Arithmetic ramification filtrations}

\subsection{Review of Abbes and Saito's definition}\label{S:Review-AS}

We briefly review the definition of arithmetic ramification filtrations on the Galois group of a complete discretely valued field $k$.  For more details, one can consult \cite{AS-cond1, AS-cond2}.  The filtrations can be defined for $k$ of mixed characteristic; however, for the purpose of this paper, we focus on the case $k$ is of equal characteristic $p>0$.

In this subsection, we do \emph{not} assume any hypothesis we have been using in previous sections.

\begin{notation}  \label{N:theta=|s|}
Keep the notation as in previous sections.  Fix uniformizers $s$ and $t$ for $k$ and $l$, respectively. Let $v_l(\cdot)$ be the valuation on $l$ normalized so that $v_l(t) = 1$.  Denote $\theta = |s|$.
\end{notation}

\begin{notation} \label{N:free-jJm}
In this subsection, we temporarily free $j$ and $J$ from the notational restraint (Notation~\ref{N:J+}). But in later application, we will specialize to the case in which $j$ and $J$ actually index $p$-bases.
\end{notation}

\begin{definition}\label{D:AS-space}
Take $Z = (z_j)_{j \in J} \subset \calO_l$ to be a finite set of elements generating $\calO_l$ as an $\calO_k$-algebra, i.e., $\calO_k [ (u_j)_{j\in J} ] / \calI \isom \calO_l$ mapping $u_j$ to $z_j$ for $j \in J = \{\serie{}m\}$ and for some appropriate ideal $\calI$. Let $(f_i) _{i=\serie{}n}$ be a finite set of generators of $\calI$.  For $a \in \QQ_{>0}$, define the \emph{(non-logarithmic) Abbes-Saito space} to be
\begin{equation} \label{E:AS-space}
as_{l/k, Z}^a = \big\{ (u_1, \dots, u_m) \in A_k^m[0,1] \; \big| \;
|f_i(u_J)| \leq \theta^a, 1 \leq i \leq n \big\}.
\end{equation}

We denote the \emph{geometric} connected components (see \cite[9.1.4/8]{BGR} for definition) of $as_{l/k,Z}^a$ by $\pi_0^\geom (as_{l/k, Z}^a)$. The \emph{arithmetic ramification break} $b_{ar}(l/k)$ is defined to be the minimal number $b$ such that for any $a>b,$ $\# \pi_0^\geom (as_{l/k,Z}^a) = [l:k]$.
\end{definition}

\begin{definition}\label{D:AS-space-log}
Keep the notation as above.  We single out a subset $P \subset Z$ and assume that $P$ and hence $Z$ contain the uniformizer $t$. For each $j \in J$, let $e_j = v_l(z_j)$. Take a lift $g_j \in \calO_k [(u_j)_{j\in J}]$ of $z_j^e / s^{e_j}$ for each $z_j \in P$, and take a lift $h_{i,j} \in \calO_k [(u_j)_{j\in J}]$ of $z_j^{e_i} / z_i^{e_j}$ for each pair $(z_i, z_j) \in P \times P$. For $a \in \QQ_{>0}$, define the \emph{logarithmic Abbes-Saito space} to be
\[
as_{l/k, \log,Z,P}^a = \left\{ (u_J) \in A_k^m[0,1] \Bigg| 
\begin{array}{cc}
|f_i(u_J)| \leq \theta^a,  & 1 \leq i \leq n \\
|u_j^e - s^{e_j} g_j| \leq \theta^{a+e_j} & \textrm{ for all } z_j \in P \\
|u_j^{e_i} - u_i^{e_j} h_{i,j}| \leq \theta^{a+e_i e_j/e}& \textrm{ for all } (z_i, z_j) \in P \times P
\end{array}
\right\}.
\]

Similarly, the \emph{logarithmic arithmetic ramification break} $b_{ar, \log}(l/k)$ is defined to be the minimal number $b$ such that for any $a>b$, $\# \pi_0^\geom (as_{l/k,\log, Z,P}^a) = [l:k]$.
\end{definition}

\begin{remark}\label{R:idea-AS}
To ease the readers who are not familiar with Abbes and Saito's definition, we give an intuitive way to understand the definition following \cite{AS-cond1}.

First, if $a \rar \infty$, the conditions on $f_1, \dots, f_n$ in \eqref{E:AS-space} basically restrict the possible $u_J$ to be very close to $z_J$ or other solutions to the equations $f_1 = 0, \dots, f_n = 0$, which are exactly Galois conjugates of $z_J$.  Thus, one may believe that $as_{l/k, Z}^a$ has exactly $[l:k]$ geometric connected components, each of which looks like a small ``polydisc" centered at one of the solutions.  In contrast, if $a \rar 0^+$, the conditions on $f_1, \dots, f_n$ are almost vacuum and $as_{l/k, Z}^a$ is almost the whole unit polydisc.  In particular, the space is likely to be geometrically connected.  From the two extreme cases, we know that, when we increase $a$, the Abbes-Saito space shrinks from a whole unit polydisc to smaller polydiscs and, at some $a$, a bigger polydisc breaks apart into several smaller polydiscs.  The arithmetic ramification break captures the last break point.
\end{remark}

We reproduce several statements from \cite{AS-cond1} and \cite{AS-cond2}.

\begin{proposition} \label{P:AS-space-properties}
The Abbes-Saito spaces have the following properties.

\emph{(1)} For $a>0$, the spaces $as_{l/k,Z}^a$ and $as_{l/k, \log, Z, P}^a$ do not depend on the choice of generators $(f_i)_{i = \serie{}n}$ of $\calI$ and lifts $g_j$ and $h_{i,j}$ for $i, j \in P$ \emph{\cite[Section~3]{AS-cond1}}.

\emph{(1')} If in the definition of both Abbes-Saito spaces, we choose polynomials $(f_i)_{i = \serie{}n}$ as a set of generators of $\Ker(\calO_k \langle(u_j)_{j \in J} \rangle \rar \calO_l)$ instead of $\Ker(\calO_k[(u_j)_{j \in J}] \rar \calO_l)$, the spaces will not change.

\emph{(2)} If we substitute in another pair of generating sets $Z$ and $P$ satisfying the same properties, then we have a canonical bijection on the sets of the geometric connected components $\pi_0^\geom (as_{l/k, Z}^a)$ and $\pi_0^\geom (as_{l/k,\log,Z,P}^a)$ for different generating sets, where $a>0$. In particular, both highest arithmetic ramification breaks are well-defined \emph{\cite[Section~3]{AS-cond1}}.

\emph{(3)} The highest arithmetic ramification break (resp. highest logarithmic arithmetic ramification break) gives rise to a filtration on the Galois group $G_k$ consisting of normal subgroups $\Fil^a G_k$ (resp., $\Fil_{\log}^a G_k$) for $a>0$ such that $b_{ar}(l/k) = \inf\{a|\Fil^a G_k \subseteq G_l\}$ (resp. $b_{ar, \log}(l/k) = \inf\{ a|\Fil^a_\log G_k \subseteq G_l\}$) \emph{\cite[Theorems~3.3 and 3.11]{AS-cond1}}. Moreover, for $l/k$ a finite Galois extension, both arithmetic ramification breaks are rational numbers \emph{\cite[Theorems~3.8 and 3.16]{AS-cond1}}.

\emph{(4)} Let $k'/k$ be an algebraic extension of complete discretely valued fields or the completion of such an extension.  If $k'/k$ is unramified, then $\Fil^a G_{k'} = \Fil^a G_k$ for $a>0$ \emph{\cite[Proposition~3.7]{AS-cond1}}. If $k'/k$ is tamely ramified with ramification index $e'$, then $\Fil_\log^{e'a}G_{k'} = \Fil_\log^a G_k$ for $a>0$ \emph{\cite[Proposition~3.15]{AS-cond1}}.  More generally, for a (not necessarily algebraic) extension $k'/k$ of complete discretely valued fields with the \emph{same} valued group and linearly independent from $l/k$ such that $\calO_{lk'} = \calO_{k'} \otimes_{\calO_k} \calO_l$, we have $b_{ar}(lk' / k') = b_{ar}(l/k)$ and $b_{ar, \log}(lk' / k') = b_{ar, \log}(l/k)$ \emph{\cite[Lemme~2.1.5]{AM-sous-groupes}}.

\emph{(5)} For $a>0$, define $\Fil^{a+} G_k = \overline{\cup_{b>a}\Fil^b G_k}$ and $\Fil_\log^{a+} G_k = \overline{\cup_{b>a}\Fil_\log^b G_k}$. Then, the subquotients $\Fil^a G_k / \Fil^{a+} G_k$ are abelian $p$-groups if $a \in \QQ_{> 1}$ and are $0$ if $a \notin \QQ$ \emph{(\cite[Theorem~3.8]{AS-cond1} and \cite[Theorem~1]{AS-cond2})}; the subquotients $\Fil_\log^a G_k / \Fil_\log^{a+} G_k$ are elementary abelian $p$-groups if $a \in \QQ_{> 0}$ and are $0$ if $a \notin \QQ$ \emph{(\cite[Theorem~3.16]{AS-cond1} and \cite[Theorem~1.3.3]{Saito-wild-ram})}.

\emph{(6)} The inertia subgroup is $\Fil^{a} G_k$ if $a \in (0, 1]$ and the wild inertia subgroup is $\Fil^{1+} G_k = \Fil^{0+}_\log G_k$ \emph{\cite[Theorem~3.7, 3.15]{AS-cond1}}.

\emph{(7)} When the residue field $\kappa_k$ is perfect, the arithmetic ramification filtrations agree with the classical upper numbered filtrations in the following way: $\Fil^a G_k = \Fil_\log^{a-1} G_k = \Gal_k^{a-1}$ for $a \geq 1$ \emph{\cite[Section~6.1]{AS-cond1}}, where $\Gal_k^a$ is the classical upper numbered filtration on $G_k$.
\end{proposition}
\begin{proof}
For the convenience of readers, we point out some ingredients of the proof.  For details, one can consult original papers.

(1) is straightforward by matching up points.

(1') is not in any literature.  However, it can be proved verbatim as (1).

(2) One can show that if we add a new (dummy) generator in $Z$ or $P$, the new Abbes-Saito space admits a fibration over the original Abbes-Saito space whose fibers are closed discs of radius $\theta^a$.

(3) The first statement is just abstract non-sense.  The second one is essentially because Abbes-Saito spaces are defined over $k$ and the geometric connect components can be detected over the algebraic closure $k^\alg$, which has valued group $|k^\times|^\QQ$.  However, realizing this principle needs formal models of rigid spaces.  As we will reprove this result in the main theorem, we refer to the original paper for the formal model proof.  

(4) When $\calO_{lk'} \simeq \calO_l \otimes_{\calO_k} \calO_{k'}$, one can match up the non-logarithmic Abbes-Saito space for $lk'/k'$ and the extension of scalar of that for $lk' / k'$ in a natural way.  Actually, the logarithmic ramification break is not considered in \cite[Lemme~2.1.5]{AM-sous-groupes}, but the proof carries over similarly.  In the tamely ramified and the logarithmic case, one can also identify two logarithmic Abbes-Saito spaces \cite[Proposition~9.8]{AS-cond1}; it is slightly more complicated.

(5) The proof used the formal models of the Abbes-Saito spaces and their stable reductions, which is in an orthogonal direction of the present paper.  One may consult \cite{AS-cond2} and \cite{Saito-wild-ram} for a complete treatment.

(6) is an easy fact.

(7) follows from an explicit calculation in monogenic case.
\end{proof}

\begin{remark}
To avoid confusion, we point out that in the proof of our main Theorem~\ref{T:main-theorem}, we do not need (5) and the second statement of (3) on the rationality of the breaks in the proposition above.  Therefore, we can obtain these properties from the properties of differential conductors in Theorem~\ref{T:properties-Ked-cond} via the comparison Theorem~\ref{T:main-theorem}.
\end{remark}

\begin{definition}
Let $\rho: G_k \rar GL(V_\rho)$ be a representation of finite local monodromy.  Define the \emph{arithmetic Artin and Swan conductors} as
\begin{eqnarray}
\Art_{ar}(\rho) & \Def & \sum_{a \in \QQ_{\geq 0}} a \cdot \dim \big( V_\rho^{\Fil^{a+} G_k} \big/ V_\rho^{\Fil^a G_k} \big), \\
\Swan_{ar}(\rho) &\Def& \sum_{a \in \QQ_{\geq 0}} a \cdot \dim \big( V_\rho^{\Fil_\log^{a+} G_k} \big/ V_\rho^{\Fil_\log^a G_k} \big).
\end{eqnarray}
They are actually finite sums.
\end{definition}

\begin{conjecture}\label{Conj:HA-thm}
(Hasse-Arf Theorem)
Let $k$ be a complete discretely valued field of equal characteristic $p$.  For any representation $\rho$ of $G_k$ of finite local monodromy, the arithmetic conductors are nonnegative integers, namely, $\Art_{ar}(\rho) \in \ZZ_{\geq 0}$ and $\Swan_{ar}(\rho) \in \ZZ_{\geq 0}$.
\end{conjecture}

\begin{proposition}
Conjecture~\ref{Conj:HA-thm} is true if the residue field $\kappa_k$ is perfect.
\end{proposition}
\begin{proof}
By Proposition~\ref{P:AS-space-properties}(8), we are reduced to the classical Hasse-Arf theorem \cite[\S VI.2~Theorem~1' and \S IV.2~Corollary~3]{BOOK-local-fields}.  Note that in this case, $\Swan_{ar}(\rho) = \Art_{ar}(\rho) - \dim  V_\rho / V_\rho^{I_k}$.
\end{proof}

We will prove the Conjecture~\ref{Conj:HA-thm} in Corollary~\ref{C:HA-thm-AS-fil}.

\subsection{Standard Abbes-Saito spaces and their lifts} \label{S:standard-AS}

In practice, we only study Abbes-Saito spaces which are given by some particular generators.  We explicitly write down spaces and their lifts in the sense of Section 1.

In this subsection, we retrieve Hypotheses~\ref{H:J-finite-set} and \ref{H:tot-ram}, assuming that $k$ has finite $p$-basis and the extension $l/k$ is totally and wildly ramified.  Also, we retrieve Notation~\ref{N:J+} on indexing $p$-basis.

\begin{construction}\label{C:AS-space}
We take $Z = \{\serie{c_}m, t\}$ to be the set of generators of $\calO_l/\calO_k$ given by Construction~\ref{C:cj-generate-l}.  (Maybe some of them are already in the field $k$, but we still keep those.)  We take $P = \{t\}$.  By Proposition~\ref{P:AS-space-properties}(1'), we can take the relations to be $\seriezero{p_}m$ from Notation~\ref{N:generators-of-calI}.  For $a \in \QQ_{>0}$, we define the \emph{standard Abbes-Saito spaces} as
\begin{eqnarray*}
as_{l/k}^a &=& \big\{(\seriezero{u_}m) \in A_k^{m+1}[0,1]\; \big|\;
|p_0(u_{J^+})| \leq \theta^a, \dots, |p_m(u_{J^+})| \leq \theta^a \big\} \\
as_{l/k, \log}^a &=& \big\{(\seriezero{u_}m) \in A_k^{m+1}[0,1]\; \big|\;
|p_0(u_{J^+})| \leq \theta^{a+1}, |p_1(u_{J^+})| \leq \theta^a, \dots, |p_m(u_{J^+})| \leq \theta^a
\big\}.
\end{eqnarray*}

Let $P_{J^+}$ be the lifts of $p_{J^+}$ as in Lemma~\ref{L:base-change-map}.
For $a \in \QQ_{>0}$ and $\eta_0 \in (0,1)$, we define the \emph{lifting Abbes-Saito spaces} to be
\begin{eqnarray*}
AS_{l/k}^{a, \geq \eta_0} &=& \left\{(U_{J^+}, S) \in A_K^{m+2}[0,1] \Bigg|
\begin{array}{c}
\eta_0 \leq |S| <1 \\
|P_0(U_{J^+}, S)| \leq |S|^a, \dots, |P_m(U_{J^+}, S)| \leq |S|^a
\end{array}\right\} \\
AS_{l/k, \log}^{a, \geq \eta_0} &=& \left\{(U_{J^+}, S) \in A_K^{m+2}[0,1] \Bigg|
\begin{array}{c}
\eta_0 \leq |S| <1, \;
|P_0(U_{J^+}, S)| \leq |S|^{a+1}, \\ |P_1(U_{J^+}, S)| \leq |S|^a, \dots, |P_m(U_{J^+}, S)| \leq |S|^a
\end{array}\right\};
\end{eqnarray*}
they are viewed as rigid spaces over $Z_k^{\geq \eta_0}$.
\end{construction}

\begin{lemma}  \label{L:lifting-functor}
Let $k'/k$ be a finite Galois extension of na\"ive ramification degree $e'$.  If we identify $C_k$ as a subring of $C_{k'}$ as in Construction~\ref{C:map-C_k-to-C_l}, we may view $P_{J^+}$ as polynomials in $U_{J^+}$ with coefficients in $\calO_{K'} \llbracket S'\rrbracket$, where $K'$ is the fraction field of the Cohen ring of $\kappa_{k'}$ and $S'$ is a lift of the uniformizer $s'$ in $k'$.  Then, for $\eta_0 \in (0,1)$ and $a \in \QQ_{>0}$, we have
\begin{eqnarray*}
Z_{k'}^{\geq \eta_0^{1/e'}} \times_{Z_k^{\geq \eta_0}} AS_{l/k}^{a, \geq \eta_0} &\cong& \left\{(U_{J^+}, S') \in A_{K'}^{m+2}[0,1] \Bigg|
\begin{array}{c}\eta_0^{1/e'} \leq |S'| <1 \\
|P_0| \leq |S'|^{e'a}, \dots, |P_m| \leq |S'|^{e'a}
\end{array}\right\} \\
Z_{k'}^{\geq \eta_0^{1/e'}} \times_{Z_k^{\geq \eta_0}} AS_{l/k, \log}^{a, \geq \eta_0} &\cong& \left\{(U_{J^+}, S') \in A_{K'}^{m+2}[0,1] \Bigg|
\begin{array}{c}
\eta_0^{1/e'} \leq |S'| <1, \;
|P_0| \leq |S'|^{e'(a+1)}, \\
|P_1| \leq |S'|^{e'a}, \dots, |P_m| \leq |S'|^{e'a}
\end{array}\right\};
\end{eqnarray*}
\end{lemma}
\begin{proof}
The only thing not obvious is that we replace $|P_j| \leq |S|^{a(+1)}$ by $|P_j| \leq |S'|^{e'a(+e')}$; this is because $|S| = |S'|^{e'}$ as proved in Lemma~\ref{L:base-change-map}(2).
\end{proof}

\begin{remark}
Note that $Z_k^{\geq \eta_0} \rar Z_{k'}^{\eta_0^{1/e'}}$ is not a morphism between rigid spaces for the same reason explained in Remark~\ref{R:naive-K-structure}.  So, strictly speaking, the $Z_{k'}^{\geq \eta_0^{1/e'}} \times_{Z_k^{\geq \eta_0}} AS_{l/k}^{a, \geq \eta_0}$ and the log counterpart should be thought of as the geometric incarnations of the tensor products of the corresponding ring of analytic functions.  The new spaces are, however, well-defined rigid analytic spaces over $Z_{k'}^{\geq \eta_0^{1/e'}}$.
\end{remark}

\begin{theorem}  \label{T:connected-components-as}
For $a \in \QQ_{>0}$, there is a one-to-one correspondence between the geometric connected components of $as_{l/k,(\log)}^a$ and the following limit of connected components:
\[
\varprojlim_{k'/k} \lim_{\eta_0 \rar 1^-} \pi_0^{\geom} \Big( Z_{k'}^{\geq \eta_0^{1/e'}} \times_{Z_k^{\geq \eta_0}} AS_{l/k(, \log)}^{a, \geq \eta_0} \Big),
\]
where $e'$ is the na\"ive ramification degree of $k' / k$ and the second limit only takes $\eta_0 \in p^\QQ \cap (0,1)$.
\end{theorem}
\begin{proof}
By Lemma~\ref{L:lifting-functor} and Example~\ref{Example:lifting-space}, when $e'a \in \ZZ$, $Z_{k'}^{\geq \eta_0^{1/e'}} \times_{Z_k^{\geq \eta_0}} AS_{l/k(, \log)}^{a, \geq \eta_0}$ is a lifting space of $as_{l/k(, \log)}^{a}$.  The theorem then follows from Corollary~\ref{C:connected-components}.
\end{proof}

\begin{remark} \label{R:Berkovich-vs-classical}
Here, we need $\eta_0 \in p^\QQ \cap (0,1)$ because Corollary~\ref{C:connected-components} requires it.
\end{remark}

\begin{remark}
Introducing this ramified extension $k'/k$ to make $e'a \in \ZZ$ may not be essential, but it eases the proof.
\end{remark}

\subsection{Comparison of rigid spaces} \label{S4:comparison-rigid-spaces}

In this subsection, we will prove that the lifting  Abbes-Saito spaces are isomorphic to some thickening spaces we constructed in Subsection~\ref{S:calF}.  In this subsection, we continue to assume Hypotheses~\ref{H:J-finite-set} and \ref{H:tot-ram}.

Before proving the comparison theorem, we need to closely analyze Construction~\ref{C:cj-generate-l} and give a new view of $\tilde \pi^*$ using differentials.  However, the proofs of the following two lemmas are not so enlightening in this generality; the reader may skip them when reading the paper for the first time, but see Remark~\ref{R:easy-proof}.

\begin{lemma}  \label{L:Taylor-homo}
Modulo $p$, the homomorphism $\tilde \pi^*$ gives a continuous homomorphism $\overline {\tilde \pi}^*: \kappa_k \rar \kappa_k \llbracket \delta_J \rrbracket$.
For $\bar g \in \kappa_k$, we can write $d\bar g = \bar g_1 d \bar b_1 + \cdots + \bar g_m d \bar b_m$ in $\Omega^1_{\kappa_k/ \Fp}$.  Then $\overline {\tilde \pi}^*(\bar g) \equiv \bar g + \bar g_1 \delta_1 + \cdots + \bar g_m \delta_m$ modulo $(\delta_J)^2 \cdot \kappa_K \llbracket \delta_J \rrbracket$.
\end{lemma}
\begin{proof}
Use the $p$-basis to express $\bar g$ (uniquely) as $\bar g = \sum_{e_J = 0}^{p-1} \bar a_{e_J}^p \bar b_J^{e_J}$
for some $\bar a_{e_J} \in \kappa_k$.  Thus, $d\bar g = \sum_{e_J = 0}^{p-1} \bar a_{e_J}^p d( \bar b_J^{e_J})$.  On the other hand, $\overline {\tilde \pi}^*(\bar g) \equiv \sum_{e_J = 0}^{p-1} \bar a_{e_J}^p (\bar b_J +\delta_J)^{e_J}$ modulo $(\delta_J)^p \cdot \kappa_K \llbracket \delta_J \rrbracket$.  The statement follows by simply comparing the two formulas.
\end{proof}

\begin{remark}
The analogous result for $\tilde \pi^*$ is true.  Actually, $\tilde \pi^*$ is the Taylor expansion homomorphism (see \cite[Definition~2.2.2]{KSK-Swan1}).
\end{remark}

\begin{lemma} \label{L:determinant}
Keep the notation as in Section~\ref{S:construction-calE}.  We have
\[
\det \Big(\frac{\partial (\tilde \pi^*(P_i) -P_i)} {\partial \delta_j} \Big)_{i,j \in J^+} \Big| _{\delta_{J^+} = 0} \in \big( \calO_K\llbracket S \rrbracket \langle U_{J^+}\rangle / (P_{J^+}) \big)^\times = (\calO_L \llbracket T \rrbracket)^\times.
\]
In particular, the corresponding matrix is invertible.
\end{lemma}
\begin{proof}
It is enough to prove that the matrix is of full rank modulo $(p, T)$.  First, modulo $(p, T)$, the first row will be all zero except the first element which is $\gothd(\bar c_1, \dots, \bar c_m) \in \kappa_l^\times$.  Hence, we need only to look at 
\begin{equation} \label{E:matrix}
\Big(\frac{\partial (\tilde \pi^*(P_i) -P_i)} {\partial \delta_j} \Big)_{i, j \in J} \mod (p, T, \delta_{J^+})
 = \Big(\frac{\partial (\overline {\tilde \pi}^*(\bar \gothb_i) - \bar \gothb_i)} {\partial \delta_j} \Big)_{i,j \in J} \mod (t, \delta_{J^+})
\end{equation}
Let $\bar \alpha_{ij} \in \kappa_l$ denote the entries in the matrix on the right hand side of \eqref{E:matrix}, where we identify $\calO_k \langle u_{J^+} \rangle / (p_{J^+}, u_0) \isom \kappa_l$.  Under this identification,  $\bar \gothb_i$ will become $\bar c_i^{p^{r_i}}$ for all $i \in J$.  It suffices to show that the $i$-th row is $\kappa_l$-linearly independent from the first $i-1$ rows for all $i$.  If we write 
\[
\bar \gothb_i = \sum_{e_1=0}^{p^{r_0}-1} \cdots \sum_{e_{i-1} = 0}^{p^{r_{i-1}}-1} \bar 
\lambda_{e_1, \dots, e_{i-1}} u_1^{e_1} \cdots u_{i-1}^{e_{i-1}},
\]
where $\bar \lambda_{e_1, \dots, e_{i-1}} \in \kappa_k$ for which $d \bar \lambda_{e_1, \dots, e_{i-1}} = \bar \mu_{e_1, \dots, e_{i-1}, 1}d \bar b_1 + \cdots + \bar \mu_{e_1, \dots, e_{i-1}, m}d \bar b_m$, then by Lemma~\ref{L:Taylor-homo},
\begin{align*}
\bar \alpha_{i1} d\bar b_1 + \cdots + \bar \alpha_{im} d\bar b_m &= \sum_{e_1=0}^{p^{r_0}-1} \cdots \sum_{e_{i-1} = 0}^{p^{r_{i-1}}-1} 
\bar c_1^{e_1} \cdots \bar c_{i-1}^{e_{i-1}} \big( \bar \mu_{e_1, \dots, e_{i-1}, 1}d \bar b_1 + \cdots + \bar \mu_{e_1, \dots, e_{i-1}, m}d \bar b_m \big) \\
& \equiv d(\bar c_i^{p^{r_i}}) \textrm{ modulo } \big(d\bar c_1, \dots, d\bar c_{i-1} \big)
\end{align*}
in $\Omega^1_{\bbk_{i-1}/\Fp}$; it is in fact nontrivial because $d\bar c_1, \dots, d\bar c_m$ form a basis of $\Omega^1_{\kappa_L / \Fp}$ and hence there should not be any auxiliary relation among $d\bar c_1, \dots, d\bar c_i$ in $\Omega^1_{\bbk_i/\Fp}$.  But we know that the sums $\bar \alpha_{i'1} d\bar b_1 + \cdots + \bar \alpha_{i'm}db_m$ for $i'<i$ all lie in the submodule of $\Omega^1_{\bbk_{i-1}/\Fp}$ generated by $d\bar c_1, \dots, d\bar c_{i-1}$.  Hence the $i$-th row of the matrix in \eqref{E:matrix} is ($\bbk_{i-1}$-)linearly independent from the first $i-1$ rows.  The lemma follows.
\end{proof}

\begin{remark}\label{R:easy-proof}
When $\kappa_l / \kappa_k$ is modular in the sense of \cite{Sweedler-insep-extn}, we can choose the $p$-basis of $\kappa_k$ so that $\bar c_j^{p^{r_j}} = \bar b_j$; in that case, the above lemma is much easier to prove because the matrix, modulo $(p, T)$, is lower triangular with 1 on the diagonal.  However, this may not be the case in general; see also Remark~\ref{R:Sweedler}.
\end{remark}

\begin{theorem}\label{T:AS=TS}
There exists $\eta'_0 \in (0, 1)$ such that for any $a \in \QQ_{>1}$ and any $\eta_0 \in (\max\{p^{-1/a}, \eta'_0\},1)$, there exists an isomorphism of rigid spaces over $Z_k^{\geq \eta_0}$:
\begin{equation}
\label{E:TS=AS-nonlog-eta0}
TS_{k\backslash l}^{a, \geq \eta_0} \simeq AS_{l/k}^{a, \geq \eta_0}.
\end{equation}

Similarly, There exists $\eta'_0 \in (0, 1)$ such that for any $a \in \QQ_{>0}$ and any $\eta_0 \in (\max\{p^{-1/a}, \eta'_0\},1)$, there exists an isomorphism of rigid spaces over $Z_k^{\geq \eta_0}$:
\begin{equation}
\label{E:TS=AS-log-eta0}
TS_{k\backslash l, \log}^{a, \geq \eta_0} \simeq AS_{l/k, \log}^{a, \geq \eta_0}.
\end{equation}
\end{theorem}

\begin{proof}
We will give the proof for the log-spaces and make  changes for the non-log spaces when necessary.  The proofs in two cases will be almost the same except that when constructing the morphism $\chi_2$, we have slightly different approximations.  We will match up the ring of functions on the two rigid spaces in \eqref{E:TS=AS-log-eta0} (resp. \eqref{E:TS=AS-nonlog-eta0}).  Fix an $\eta_0 \in (p^{-1/a}, 1)$ satisfying Hypothesis~\ref{H:eta_0}. ($\eta'_0$ is given by the conditions in Hypothesis~\ref{H:eta_0}.)

Recall that $\calO_{TS_{k, \log}^{a, \geq \eta_0}} = \calR_K^{\eta_0} \langle S^{-a-1}\delta_0, S^{-a}\delta_J \rangle$ (resp. $\calO_{TS_k^{a, \geq \eta_0}} = \calR_K^{\eta_0} \langle S^{-a}\delta_{J^+} \rangle$).  For each $j \in J^+$, $\tilde \pi^*(P_j)$ is the polynomial $P_j$ with coefficients replaced by their pull-backs to $\calO_{TS_{k, \log}^{a, \geq \eta_0}}$ (resp. $\calO_{TS_k^{a, \geq \eta_0}}$) via $\tilde \pi^*$. So the ring of functions on $TS_{k\backslash l, \log}^{a, \geq \eta_0}$ (resp. $TS_{k\backslash l}^{a, \geq \eta_0}$) is
\begin{equation}\label{E:def-TS}
\begin{array}{c}
\calR_{1, \log} = \calR_K^{\eta_0} \langle S^{-a-1}\delta_0, S^{-a}\delta_J \rangle \langle U_{J^+} \rangle \big/ \big(\tilde \pi^*(P_{J^+})\big)\\
(\textrm{resp. } \calR_1 = \calR_K^{\eta_0} \langle S^{-a}\delta_{J^+} \rangle \langle U_{J^+} \rangle \big/ \big(\tilde \pi^*(P_{J^+})\big)\ ).
\end{array}
\end{equation}

By Lemma~\ref{L:base-change-map}(1),
\begin{eqnarray*}
\tilde \pi^* (P_j) &\in& U_j^{p^{r_j}} - \tilde \pi^*(\gothB_j) + (p, U_0, S, \delta_0) \cdot \OK \llbracket \delta_{J^+}, S\rrbracket [ U_{J^+} ],  \\
\tilde \pi^* (P_0) &\in& U_0^e - \tilde \pi^*(\gothD) S - \delta_0 + (p, U_0S, S^2, U_0\delta_0, S\delta_0, \delta_0^2)\cdot \OK\llbracket \delta_{J^+}, S \rrbracket [U_{J^+}],
\end{eqnarray*}

Thus, we can view $\calR_{1, \log}$ (resp. $\calR_1$) as a finite free module over $\calO_{TS_{k, \log}^{a, \geq \eta_0}}$ (resp. $\calO_{TS_k^{a, \geq \eta_0}}$) with basis $\big\{ U_{J^+}^{e_{J^+}} \big|\, 0 \leq e_0 < e;\, 0 \leq e_j < p^{r_j}, j \in J \big\}$.  For each $\eta \in [\eta_0, 1)$, we provide $\calR_{1, \log}$ (resp. $\calR_1$) the following norm: for $g = \sum \lambda_{e_{J^+}}U_{J^+}^{e_{J^+}}$ with $\lambda_{e_{J^+}} \in \calO_{TS_{k, \log}^{a, \geq \eta_0}}$ (resp. $\lambda_{e_{J^+}} \in \calO_{TS_k^{a, \geq \eta_0}}$), summed over $e_0 = 0, \dots, e-1$ and $e_j = 0, \dots, p^{r_j}-1$ for $j \in J$, we define
$$
|g|_{\calR_{1, \log}, \eta} = \max_{e_{J^+}} \{|\lambda_{e_{J^+}}|_{TS_{k, \log}^{a, \eta}} \cdot \eta^{e_0/e}\} \quad \textrm{ (resp. } |g|_{\calR_1, \eta} = \max_{e_{J^+}} \{|\lambda_{e_{J^+}}|_{TS_k^{a, \eta}} \cdot \eta^{e_0/e}\} \ ).
$$
It is clear that $\calR_{1, \log}$ (resp. $\calR_1$) is the Fr\'echet completion for the norms $|\cdot|_{\calR_{1, \log}, \eta}$ (resp. $|\cdot|_{\calR_1, \eta}$) for all $\eta \in [\eta_0, 1)$.

On the other hand, by the definition of ${AS}_{l/k, \log}^{a, \geq \eta_0}$ (resp. ${AS}_{l/k}^{a, \geq \eta_0}$), its ring of functions is
\begin{eqnarray*}
&\calR_{2, \log} = \calR_K^{\eta_0} \langle S^{-a-1}  V_0, S^{-a}V_J \rangle \langle U_{J^+}\rangle \big/ (P_{J^+} -V_{J^+}) \\
&\textrm{(resp. }\calR_2 = \calR_K^{\eta_0} \langle S^{-a}V_{J^+} \rangle \langle U_{J^+}\rangle \big/ (P_{J^+} -V_{J^+})\ ),
\end{eqnarray*}
which is clearly a finite free module over
$$
W_\log = \calR_K^{\eta_0} \langle V_0 / \eta^{a+1}, V_J / \eta^a \rangle \quad \textrm{ (resp. } W = \calR_K^{\eta_0} \langle V_{J^+} / \eta^a \rangle \ )
$$
with basis $\big\{U_{J^+}^{e_{J^+}} \big|\, 0 \leq e_0 < e;\ 0 \leq e_j < p^{r_j}, j \in J \big\}$.  Similarly, for $\eta \in [\eta_0, 1)$, we provide $\calR_{2, \log}$ (resp. $\calR_2$) with the following norm: for $g = \sum \lambda_{e_{J^+}}U_{J^+}^{e_{J^+}}$ with $\lambda_{e_{J^+}} \in W_\log$ (resp. $\lambda_{e_{J^+}} \in W$) summed over $e_0 = 0, \dots, e-1$ and $e_j = 0, \dots, p^{r_j}-1$ for $j \in J$, we define 
$$
|g|_{\calR_{2, \log}, \eta} = \max_{e_{J^+}} \{|\lambda_{e_{J^+}}|_{W_\log} \cdot \eta^{e_0/e}\} \quad \textrm{ (resp. } |g|_{\calR_2, \eta} = \max_{e_{J^+}} \{|\lambda_{e_{J^+}}|_W \cdot \eta^{e_0/e}\}\ ).
$$
It is clear that $\calR_{2, \log}$ (resp. $\calR_2$) is the Fr\'echet completion for the norms $|\cdot|_{\calR_{2, \log}, \eta}$ (resp. $|\cdot|_{\calR_2, \eta}$) for all $\eta \in [\eta_0, 1)$.

We will identify $U_{J^+}$ in different rings but the $V_{J^+}$ will not be same as $\delta_{J^+}$.  Be caution that the two norms will not be the same under the identification, but they will give the same topology.

Now, we define a continuous $K$-homomorphism $\chi_1: \calR_{2, \log} \rar \calR_{1, \log}$ (resp. $\chi_1: \calR_2 \rar \calR_1$) so that $\chi_1(S) = S$, $\chi_1(U_j) = U_j$, $\chi_1(V_j) = P_j(U_{J^+})$ for all $j \in J^+$.  We need only to check that for any $\eta \in [\eta_0, 1)$,
\begin{equation}\label{E:r1<r2}
|\chi_1(V_j)|_{\calR_{1, \log}, \eta} \leq \left\{
\begin{array}{ll}
\eta^{a+1} & j = 0 \\
\eta^a & j \in J 
\end{array}
\right. \textrm{ (resp. } |\chi_1(V_j)|_{\calR_1, \eta} \leq \eta^a,\  \forall j \in J^+).
\end{equation}
Here we need separate arguments for logarithmic case and non-logarithmic case.  In the logarithmic case, Inequality \eqref{E:compare-norm-of-pi-and-tilde-pi} tells us $|P_j - \tilde \pi^*(P_j)|_{\calR_{1, \log}, \eta} \leq \eta^a |P_j|_{\calR_{2, \log}, \eta}$ for $j \in J^+$, which exactly gives the bound in \eqref{E:r1<r2} because $|P_0|_{\calR_{2, \log}, \eta} \leq \eta$ and $|P_j|_{\calR_{2, \log}, \eta} \leq 1$ for $j \in J$ by  Lemma~\ref{L:base-change-map}(1).  In the non-logarithmic case, combining Lemma~\ref{L:base-change-map}(1) and inequality \eqref{E:bound-of-tilde-pi-non-log}, one has $|P_j - \tilde \pi^*(P_j)|_{\calR_1, \eta} \leq \eta^a$ for $j \in J^+$; Inequality \eqref{E:r1<r2} follows.

Conversely, we will define a continuous $K$-homomorphism $\chi_2: \calR_{1, \log} \rar \calR_{2, \log}$ (resp. $\chi_2: \calR_1 \rar \calR_2$) as the inverse to $\chi_1$.  Obviously, we need $\chi_2(S) = S$, $\chi_2(U_j) = U_j$ for all $j \in J^+$.  The only thing not clear is $\chi_2(\delta_j)$ for all $j \in J^+$.

By Lemma~\ref{L:determinant}, let
\[
A := \Big(\frac{\partial (\tilde \pi^*(P_i) -P_i)} {\partial \delta_j} \Big)_{i,j \in J^+}\Big| _{\delta_{J^+} = 0} \in \GL_{m+1}(\calO_L\llbracket T \rrbracket) \cong \GL_{m+1}\big( \calO_K \llbracket S \rrbracket \langle U_{J^+}\rangle / (P_{J^+}) \big).
\]
Let $A^{-1}$ denote the $(m+1) \times (m+1)$ matrix whose entries are in the free $\calO_K\llbracket S \rrbracket$-module generated by the basis in Lemma~\ref{L:base-change-map}(1) and whose image in $M_{m+1}(\calO_K \llbracket S \rrbracket \langle U_{J^+}\rangle / (P_{J^+}))$ is the inverse of $A$.  Thus,
\begin{equation}\label{E:A-times-A-1}
A^{-1} \cdot A - I \in \Mat_{m+1} \big( (\delta_{J^+}) \cdot \calO_K \llbracket S \rrbracket \langle U_{J^+} \rangle \big),
\end{equation}
where $I$ is the $(m+1) \times (m+1)$ identity matrix.  Now, we write
\begin{equation}\label{E:delta-in-terms-of-V}
\vectzero{\delta_}m = (I-A^{-1}A) \vectzero{\delta_}m - A^{-1} \left( \Bigg(\begin{array}{c}
\tilde \pi^* (P_0) - P_0 \\ \vdots \\ \tilde \pi^*(P_m) - P_m \end{array}\Bigg) - A \vectzero{\delta_}m \right) - A^{-1} \vectzero{P_}m;
\end{equation}
the last term is just $- A^{-1} \cdot \chi_1(V_{J^+})$.  We need to bound the first two terms.

By \eqref{E:A-times-A-1}, $I - A^{-1}A$ has norm $\leq \eta^a$.  Hence, in the non-logarithmic case, the first term in \eqref{E:delta-in-terms-of-V} has norm $\leq \eta^{2a}$; in the logarithmic case the first term in \eqref{E:delta-in-terms-of-V} has norm $\leq \eta^{2a}$, except for the first row, which has norm $\leq \eta^{2a+1}$.  By the definition of $A$ and Theorem~\ref{T:thickening-morphism}, the second term in \eqref{E:delta-in-terms-of-V} has norm $\leq \eta^{2a}$ in the non-logarithmic case; it has norm $\leq \eta^{2a}$ in the logarithmic case, except for the first row, which has norm $\leq \eta^{2a+1}$.

Since we want $\chi_2$ to be the inverse of $\chi_1$, we define recursively by
\begin{equation}
\label{E:iteration}
\chi_2 \vectzero{\delta_}m = -A^{-1} \vectzero{V_}m + \chi_2 \vectzero{\Lambda_}m,
\end{equation}
where $\Lambda_{J^+}$ denotes the sum of the first two terms in \eqref{E:delta-in-terms-of-V}.  Since $\Lambda_{J^+}$ have strictly smaller norms than $\delta_{J^+}$ and $\Lambda_{J^+}$ are in the ideal $(\delta_{J^+})$, one can plug the image of $\chi_2(\delta_{J^+})$ back into $\chi_2(\Lambda_{J^+})$ and iterate this substitution.  This iteration will converge to the values of $\chi(\delta_{J^+})$, as an element in $\calR_{2, \log}$ (resp. $\calR_2$).  Moreover, from the construction, one can see that
\begin{eqnarray*}
&|\chi_2(\delta_j)|_{\calR_1} \leq \eta^a, \textrm{ for all } j \in J^+, \eta \in [\eta_0,1),& \\
&|\chi_2(\delta_0)|_{\calR_{1, \log}} \leq \eta^{a+1}\textrm{ and } |\chi_2(\delta_j)|_{\calR_{1, \log}} \leq \eta^a \textrm{ for all } j \in J, \eta \in [\eta_0, 1). &
\end{eqnarray*}
Hence, if we define $\chi_2: \calR_K^{\eta_0}\langle S^{-a-1}\delta_0, S^{-a}\delta_J\rangle\langle U_{J^+} \rangle \rar \calR_{2, \log}$ (resp. $\chi_2: \calR_K^{\eta_0}\langle S^{-a}\delta_{J^+}\rangle\langle U_{J^+} \rangle \rar \calR_{2}$) such that $\chi_2(u_{J^+}) = u_{J^+}$ and $\chi_2(\delta_{J^+})$ to be the limit we obtained above; this gives a continuous homomorphism.  We will check that this homomorphism factors through $\calR_{1, \log}$ (resp. $\calR_1$).  Indeed, by the recursive formula \eqref{E:iteration} which is \eqref{E:delta-in-terms-of-V} after applying $\chi_2$, we see exactly that
\[
-A^{-1} \chi_2\Bigg(\begin{array}{c}
\tilde \pi^*(P_0) - P_0 \\ \vdots \\ \tilde \pi^*(P_m) - P_m \end{array}\Bigg) - A^{-1}\vectzero{V_}{m} = \Bigg(\begin{array}{c}
0 \\ \vdots \\ 0  \end{array}\Bigg)
\]
We know that $A^{-1}$ has invertible image in $GL_{m+1}(\calO_K\llbracket S\rrbracket\langle U_{J^+} / (P_{J^+}))$, and hence is invertible over $\calR_{1, \log}$ (resp. $\calR_1$).  We must have $0= \chi_2(\tilde \pi^*(P_j) - P_j) +V_j = \chi_2(\tilde \pi^*(P_j)) +V_j -P_j = \chi_2(\tilde \pi^*(P_j))$ for all $j \in J^+$.  This proves that $\chi_2$ factors through $\calR_{1, \log}$ (resp. $\calR_1$).

Finally, we claim that $\chi_2$ and $\chi_1$ are inverse to each other.  One may check this from the definition directly.  Alternatively, we observe that, by our definition, they are inverse to each other on a dense subset $K[S, u_{J^+}]$, the polynomial ring inside the Fr\'echet algebras; and therefore, they have to be inverse to each other and give an isomorphism between the ring of functions on Abbes-Saito space and the ring of functions on thickening space.
\end{proof}

\begin{remark}
The isomorphisms constructed in Theorem~\ref{T:AS=TS} are canonical in the sense that they match up $U_{J^+}$ on the both sides.  However, slight perturbations of the isomorphisms will continue to be isomorphic.  This point will be important when studying the mixed characteristic case.
\end{remark}

\subsection{Comparison of conductors}
\label{S4:Comparison-thm}

In this subsection, we will prove the comparison between the arithmetic conductors and the differential conductors.  As a reminder, we do not impose Hypotheses~\ref{H:J-finite-set} and \ref{H:tot-ram} in this subsection.

\begin{theorem} \label{T:main-theorem}
Let $k$ be a complete discretely valued field of equal characteristic $p>0$ and let $G_k$ be its absolute Galois group.  For a $p$-adic representation $\rho: G_k \rar GL(V_\rho)$ of finite local monodromy, the arithmetic Artin conductor $\Art_{ar}(\rho)$ of $\rho$ coincides with the differential Artin conductor $\Art_\dif(\rho)$; the arithmetic Swan conductor $\Swan_{ar}(\rho)$ coincides with the differential Swan conductor $\Swan_\dif(\rho)$.
\end{theorem}

\begin{proof}
It suffices to prove for irreducible representations, as all the conductors are additive.  Since all the conductors remain the same if we pass to the completion of the unramified closure of $k$ (Proposition~\ref{P:AS-space-properties}(4), Theorem~\ref{T:properties-Ked-cond}(2)), we may assume that the residue field $\kappa_k$ is separably closed; hence $\rho$ factors through the Galois group of a finite totally ramified extension $l/k$ as $\rho: G_k \surj \Gal(l/k) \inj \GL(V_\rho)$ with the second map injective.  Moreover, we may assume that $l/k$ is wildly ramified because the theorem is known when $l/k$ is tamely ramified (Proposition~\ref{P:AS-space-properties}(6) and Theorem~\ref{T:properties-Ked-cond}(3)).  To sum up, we may assume Hypothesis~\ref{H:tot-ram}.  In particular, $b_{ar}(l/k)>1$ and $b_{ar, \log}(l/k) >0$.

Next, we want to reduce to the case when the $p$-basis of $k$ is finite.  By Construction~\ref{C:cj-generate-l}, one can choose $p$-basis of $l$ so that all but finitely many of them are actually in $k$.  Let $(c_i)_{i\in I}$ be a subset of those elements in the $p$-basis which lie in $k$.  Denote $\tilde k = k(c_i^{1/p^n}, i \in I, n \in \NN)^\wedge$ and $\tilde l = l \tilde k$.  We claim that $\calO_{\tilde l} = \calO_l \otimes _{\calO_k} \calO_{\tilde k}$.  Indeed, after base change to $\tilde k$, the valued groups do not change: $|\tilde k^\times| = |k^\times|$.  Thus, $[|\tilde l^\times|:|\tilde k^\times|] \geq [|l^\times|: |k^\times|]$.  On the other hand, the residue field extension of $\tilde l / \tilde k$ has degree at least the same as $\kappa_l / \kappa_k$ because $\bar c_{J \backslash I}$ are not in the residue field of $\tilde k$.  But we know that the degree of the extension does not increase.  Therefore, we have equality on both na\"ive ramification degrees and degrees of residue field extension.  It is then clear that $\calO_{\tilde l} = \calO_l \otimes _{\calO_k} \calO_{\tilde k}$, as the right hand side contains the uniformizer of the left hand side and both sides are isomorphic modulo that uniformizer.  Therefore, by Proposition~\ref{P:AS-space-properties}(4), $b_{ar} (\tilde l / \tilde k) = b_{ar} (l/k)$.

On the differential conductors side, \cite[Lemma~3.5.4]{KSK-Swan1} (the non-log case follows by similar argument) shows that we can consider only finitely many elements in the $p$-basis and the differential conductors are unchanged after making inseparable field extension with respect to other elements in the $p$-basis.

To sum up, we can make an inseparable extension so that all conductors do not change, and we are reduced to the case where Hypothesis~\ref{H:J-finite-set} holds.

Now, we will prove the comparison theorem for the Swan conductors and the proof for the Artin conductors follows verbatim, except replacing Swan by Art and $a>0$ by $a>1$ and dropping all the log's in the subscripts.

Since $\rho$ is irreducible, $\Swan_{ar}(\rho) = b_{ar, \log}(l/k) \cdot \dim V_\rho$.  Recall that in Subsection~\ref{S2-def-diff-conductors}, we can associate to $\rho$ a differential module $\calE_\rho$ over $\calR_K^{\eta_0} \otimes_{\QQ_q} F$ for some $\eta_0 \in (0,1)$.  As the representation $\rho$ is irreducible, $\calE_\rho$ has a unique ramification break $b_{\dif, \log}(\calE_\rho)$.  So the differential Swan conductor of $\rho$ is $\Swan_\dif(\rho) = b_{\dif, \log}(\calE_\rho) \cdot \dim V_\rho$.  Therefore, to conclude, it suffices to show that $b_{ar, \log}(l/k) = b_{\dif, \log}(\calE_\rho)$.

This follows from the following equivalence relations.
\begin{align*}
&\ a > b_{\dif, \log}(\calE_\rho) \\
\Longleftrightarrow & \begin{array}{l}
\textrm{for any (or some) extension } l'/l \textrm{ with na\"ive ramification degree } e',\\
\pi_0^\geom \Big( Z_{l'}^{\geq \eta_0^{1/ee'}} \times_{Z_l^{\geq \eta_0^{1/e}}} TS_{l/k \backslash l, \log}^{a, \geq \eta_0}\Big) = [l:k] \textrm{ when } \eta_0 \rar 1^-
\end{array}
\ \ \textrm{(Theorem~\ref{T:ram-break-equiv-disconn} (1)$\Leftrightarrow$(4))} \\
\Longleftrightarrow &\ \pi_0^\geom \big( as_{l/k, \log}^a \big) = [l:k] \qquad \textrm{(Theorem~\ref{T:connected-components-as})} \\
\Longleftrightarrow &\ a > b_{ar, \log}(l/k),
\end{align*}
where $a$ is a rational number.
\end{proof}

\begin{remark}
In an early version of this paper, Theorem~\ref{T:main-theorem} is stated for representations with finite image.  Andrea Pulita pointed out that this could be extended to the finite local monodromy case by a standard argument as in the proof.
\end{remark}

\begin{corollary}\label{C:HA-thm-AS-fil}
\emph{(1) (Hasse-Arf Theorem)} Let $k$ be a complete discretely valued field of equal characteristic $p>0$, let $G_k$ be its absolute Galois group, and let $\rho: G_k \rar GL(V_\rho)$ be a $p$-adic representation of finite local monodromy.  Then the arithmetic Artin conductor $\Art_{ar}(\rho)$ and the arithmetic Swan conductor $\Swan_{ar}(\rho)$ are integers.

\emph{(2)} Let $k$ be a complete discretely valued field of equal characteristic $p > 0$. Then the subquotients $\Fil^a G_k / \Fil^{a+} G_k$ (resp. $\Fil_\log^a G_k / \Fil_\log^{a+} G_k$) of the arithmetic ramification filtrations are elementary $p$-abelian groups if $a \in \QQ_{> 1}$ (resp. $a \in \QQ_{> 0}$) and are trivial if $a \notin \QQ$.
\end{corollary}

\begin{proof}
It follows from Theorems~\ref{T:properties-Ked-cond} and \ref{T:main-theorem}.
\end{proof}

\section{Applications}
\setcounter{equation}{0}
In this section, we give two applications of the comparison Theorem~\ref{T:main-theorem}.  The first one is to deduce an integrality result concerning the ramification filtration of finite flat group schemes, introduced in \cite{AM-sous-groupes}.  The other one is to compare the arithmetic and differential Artin conductors to the Artin conductor defined by Borger \cite{Borger-conductor}.

\begin{remark}  \label{R:only-equal-char}
It is worthwhile to point out first that all applications in this section can be carried over to the mixed characteristic case if there is a good theory of differential conductors.  For the application to finite flat group schemes, one needs the Hasse-Arf theorem of arithmetic Artin conductors; for the comparison with Borger's Artin conductor, one needs a mixed characteristic version of Proposition~\ref{P:diff-inv-p-th-root}.  In the absence of these statements, we only focus on the equal characteristic $p$ case throughout this section.
\end{remark}

\subsection{Hasse-Arf theorem for finite flat group schemes}
\label{S:ffgs}

We first recall some definitions and basic properties from \cite{AM-sous-groupes} and \cite{Hattori-tame-char}.  Then, we use a theorem by Raynaud \cite[Th\'eor\`em~3.1.1]{BBM-dieudonne-crisII} to reduce the integrality result to the case of finite Galois extension of complete discretely valued fields.

Keep the notation as in previous sections.  We do not assume any hypothesis on $k$ (and there will be no $l$ in this subsection).

\begin{convention}
All finite flat groups schemes are commutative.
\end{convention}

The construction of the canonical filtration on a generically \'etale finite flat group scheme is similar to that of the arithmetic ramification filtration.

\begin{definition}
Let $A$ be a finite flat $\calO_k$-algebra.  Write $A = \calO_k[x_1, \dots, x_n]/\calI$ with $\calI$ an ideal generated by $f_1, \dots, f_r$.  For $a \in \QQ_{>0}$, define the rigid space
\[
X^a = \big\{ (x_1, \dots, x_n) \in A_K^n[0,1]\; \big| \;
|f_\alpha(x_1, \dots, x_n)| \leq \theta^a, \alpha = 1, \dots, r \big\},
\]
where $\theta = |s|$ as in Notation~\ref{N:theta=|s|}.  The \emph{highest break} $b(A / \calO_k)$ is the smallest number such that for all $a> b(A / \calO_k)$, $\pi_0^\geom(X^a) = \rank_{\calO_k}A$.  This is the same as Definition~\ref{D:AS-space} if $A = \calO_l$, except here we use the ring of integers instead of the fields in the notation.
\end{definition}

\begin{notation}
For a finite flat group scheme $G = \Spec A$, it is \emph{generically \'etale} if $G \times_{\calO_k}k$ is \'etale over $k$; it is \emph{generically trivial} if $G \times_{\calO_k}k$ is a disjoint union of copies of $\Spec k$.
\end{notation}

\begin{definition}
For a geometrically \'etale finite flat group scheme $G = \Spec A$, we have a natural map of points $G(k^\alg) \inj X^a(k^\alg)$; further composing with the map for geometric connected components, we obtain a map
\[
\sigma^a: G(k^\alg) \inj X^a(k^\alg) \rar \pi_0^\geom(X^a).
\]
Define $G^a$ to be the closure of $\ker \sigma^a$. We use $b(G/\calO_k)$ to denote the highest break $b(A/\calO_k)$; then for $a > b(G / \calO_k)$, $G^a = \Spec \calO_k$.
\end{definition}

\begin{proposition}
\emph{\cite[Lemme~2.3.2]{AM-sous-groupes}}
Let $0 \rar G' \rar G \rar G'' \rar 0$ be an exact sequence of finite flat group schemes.  Then for $a>0$, $0 \rar G'^a \rar G^a \rar G''^a \rar 0$ is exact.
\end{proposition}

\begin{caution}
For a subgroup scheme $H \subset G$ and $a \geq 0$, we do not know how to link $H^a$ with $H \times_G G^a$.
\end{caution}

The following question is first raised in \cite{Hattori-tame-char}; and its proof is essentially due to Hattori.  The author would like to thank him for clarifying this and the permission to include the proof here.

\begin{theorem}
Let $\calO_k$ be a complete discrete valuation ring of equal characteristic $p$. For any generically trivial finite flat groups scheme $G$ over $\calO_k$, $b(G / \calO_k)$ is a nonnegative integer.
\end{theorem}
\begin{proof}
We may assume that $G$ is connected by taking the connected component of the identity.  By a theorem of Raynaud \cite[Th\'eor\`em~3.1.1]{BBM-dieudonne-crisII}, we may realize $G$ as the kernel of an isogeny $f: \gothB \rar \gothA$ of two abelian schemes over $\Spec \calO_k$.  Let $\alpha$ and $\beta$ be generic points of the special fibers of $\gothA$ and $\gothB$, respectively.  Then by \cite[Lemme~2.1.6]{AM-sous-groupes}, $b(\calO_{\gothB, \beta}^\wedge / \calO_{\gothA, \alpha}^\wedge) = b(G / \calO_k)$.

Since the generic fiber of $G$ is a disjoint union of copies of $\Spec k$, we know that $\calO_{\gothB, \beta}^\wedge / \calO_{\gothA, \alpha}^\wedge$ is a generically \'etale finite Galois extension of complete discrete valuation rings, with Galois group $G(k)$; in particular, all irreducible representations of this Galois group over an algebraically closed field are one dimensional.  By Hasse-Arf Theorem~\ref{T:main-theorem}, $b(\calO_{\gothB, \beta}^\wedge / \calO_{\gothA, \alpha}^\wedge) = b(G / \calO_k)$ is an integer.
\end{proof}

\subsection{Generic $p^\infty$-th roots}
\label{S5:dummy-var}

In this subsection, we introduce the notation of generic $p^\infty$-th roots.  This idea was first introduced in \cite{Borger-conductor} as a key ingredient of Borger's Artin conductor.

Keep the notation as in previous sections.  We assume Hypothesis~\ref{H:J-finite-set} that $k$ has a finite $p$-basis $b_J$.

\begin{notation}
Let $x_1, \dots, x_m$ be transcendental over $k$.  Define $k'$ to be the completion of $k(x_1, \dots, x_m)$ with respect to the $(1, \dots, 1)$-Gauss norm.  Set $l' = k' l$.  Obviously, $l'$ is the completion of $l(x_1, \dots, x_m)$ with respect to the $(1, \dots, 1)$-Gauss norm.  We call $x_1, \dots, x_m$ \emph{dummy variables}.
\end{notation}

\begin{definition}\label{D:gen-pinfty-root}
We use \emph{adding generic $p^\infty$-th roots} to refer to the following procedure.  Consider
\[
k \inj \tilde k = k' \big((b_j + x_j s)^{1/p^n}; j\in J, n\in \NN\big)^\wedge, 
\]
instead of $k$, namely, put all $p$-power roots of $b_j + x_j s$ for all $j \in J$ into $k'$ and then take the completion.  We provide $\tilde k$ with the $p$-basis $x_J$, i.e., replacing $b_j$ by $x_j$ for all $j \in J$.  For a finite field extension $l/k$, we replace it by the extension of the composite $\tilde l =l\tilde k / \tilde k$.  Note that $\Gal(\tilde l / \tilde k) = \Gal(l/k)$ as $\tilde k$ is linearly independent from $l$.  
\end{definition}

The proof of the following proposition is essentially the same as \cite[Lemma~3.5.4]{KSK-Swan1}.  It is also implicitly contained in Borger's construction of Artin conductors (Subsection~\ref{S5:comp-borger-cond}).

\begin{proposition}\label{P:finite-gen-pth-roots}
Let $l/k$ be a finite Galois extension of complete discretely valued fields of equal characteristic $p$ and with finite $p$-basis.  Then, after finitely many operations of adding generic $p^\infty$-th roots, the field extension has separable residue field extension.
\end{proposition}
\begin{proof}
First, the tamely ramified part is always preserved under these operations.  So, we can assume that $l/k$ is totally wildly ramified and hence the Galois group $G_{l/k}$ is a $p$-group.  We can filter the extension $l/k$ as $k = k_0 \subset \cdots \subset k_n = l$, where $k_i / k_{i-1}$ is a (wildly ramified) $\ZZ / p\ZZ$-Galois extension and $k_i / k$ is Galois for each $i = \serie{}n$.  Each of these subextensions 
\begin{enumerate}
\item [(a)] either has inseparable residue field extension (and hence has na\"ive ramification degree 1),
\item [(b)] or has separable residue field extension (and hence has na\"ive ramification degree $p$).
\end{enumerate}

Let $i_0$ be the maximal number such that $k_i / k_{i-1}$ has separable residual extension for $i = \serie{}i_0$.  Obviously adding generic $p^\infty$-th roots does not decrease $i_0$ because after adding generic $p^\infty$-th roots, the na\"ive ramification degree of $\tilde k_{i_0} / \tilde k$ still equals to the degree $p^{i_0}$.  It then suffices to show that after finitely many operations of adding generic $p^\infty$-th roots, $k_{i_0+1} / k_{i_0}$ has separable residue field extension.  Suppose the contrary.

Let $g \in G_{k_{i_0+1} / k_{i_0}} \simeq \ZZ / p\ZZ$ be a generator.  We claim that $\gamma = \min_{w \in \calO_{k_{i_0+1}}} \big(v_{k_{i_0+1}}(g(w) - w) \big)$ decreases by at least 1 after adding $p^\infty$-th roots.  This would conclude the proposition, because the number $\gamma$ is always a nonnegative integer, which would lead to a contradiction.

Let $z$ be a generator of $\calO_{k_{i_0+1}}$ as an $\calO_{k_{i_0}}$-algebra.  It satisfies an equation
\begin{equation}\label{E:k_i_0-over-k_i_0+1}
z^p + a_1 z^{p-1} + \cdots  + a_p = 0
\end{equation}
where $\serie{a_}{p-1} \in \gothm_{k_{i_0}}$ and $a_p \in \calO_{k_{i_0}}^\times$ with $\bar a_p \in \kappa_{k_{i_0}}^\times \backslash (\kappa_{k_{i_0}}^\times)^p = \kappa_k^\times \backslash (\kappa_k^\times)^p$.  It is easy to see that $\gamma = v_{k_{i_0}}(g(z) - z)$.

Adding generic $p^\infty$-th roots to $k$ gives us the field $\tilde k$.  Now, the field extension $\tilde k k_{i_0+1} / \tilde k k_{i_0}$ is also generated by $z$ as above.  But we can write $a_p = \alpha^p +\beta$ for $\alpha \in \calO_{\tilde k k_{i_0}}$ and $\beta \in \gothm_{\tilde k k_{i_0}}$.  Hence if we substitute $z' = z+\alpha$ into \eqref{E:k_i_0-over-k_i_0+1}, we get
$z'^p + a'_1 z'^{p-1} + \cdots + a'_p = 0$, with $\serie{a'_}p \in \gothm_{\tilde k k_{i_0}}$.  Hence, $v_{\tilde k k_{i_0+1}}(z') > 0$.  By assumption that the extension $\tilde k k_{i_0+1} / \tilde k k_{i_0}$ has na\"ive ramification degree $1$, a uniformizer $\pi_{k_{i_0}}$ of $k_{i_0}$ is also a uniformizer for $\tilde k k_{i_0+1}$ and hence $z' / \pi_{k_{i_0}}$ lies in $\calO_{\tilde k k_{i_0+1}}$.  Thus,
$$
\gamma' = \min_{w \in \calO_{\tilde k k_{i_0+1}}} \big(v_{\tilde k k_{i_0+1}}(g(w) - w) \big) \leq v_{\tilde k k_{i_0+1}} \big(g(z' / \pi_{k_{i_0}}) - z' / \pi_{k_{i_0}}\big) = v_{k_{i_0+1}}\big(g(z) - z\big) - 1 = \gamma - 1.
$$

This proves the claim and hence the proposition.
\end{proof}

\subsection{Borger's Artin conductors}
\label{S5:comp-borger-cond}

We start with reviewing Borger's definition of Artin conductors following \cite{Borger-conductor}.  Then, we prove the comparison theorem linking this to arithmetic and differential conductors.

Keep the notation as above.  Let $k$ be a complete discretely valued field of equal characteristic $p$, with no further hypothesis added.  In fact, 
Borger's construction works of mixed characteristic case, but we only focus on the equal characteristic case (see Remark~\ref{R:only-equal-char}).

\begin{definition}
An $\FF_p$-algebra $R$ is called \emph{perfect} if $F: x \mapsto x^p$ is an isomorphism.  For a $\FF_p$-algebra $R$, we use $R^\pf = \cup_{n \in \NN} R^{1/p^n}$ to denote its \emph{perfection}.  Let $\CRP_{\calO_k}$ be the subcategory of the category of $\calO_k$-algebras consisting of flat $\calO_k$-algebras $A$, complete with respect to the $\gothm_k$-adic topology and for which $A / \gothm_k A$ is perfect.
\end{definition}

\begin{proposition}\label{P:Borger-cond-constr}
\emph{\cite[Theorem~1.4]{Borger-conductor}} This category $\CRP_{\calO_k}$ has an initial object $\calO_k^u$, the \emph{universal residual perfection} of $\calO_k$.  We have an equivalence of categories
\begin{equation} \label{E:equiv-cate-Borger}
\CRP_{\calO_k} \isom \PerfAlg_{\overline {\calO_k^u}}, \quad A \mapsto A / \gothm_k A,
\end{equation}
where $\PerfAlg_{\overline {\calO_k^u}}$ is the category of perfect $\calO_k^u / \gothm_k \calO_k^u$-algebras.
\end{proposition}

\begin{definition}
Let $\calO_k^g$ be the inverse image of $\Frac(\calO_k^u / \gothm_k \calO_k^u)$ under \eqref{E:equiv-cate-Borger}, called the \emph{generic residual perfection} of $\calO_k$.  Denote $k^g = \Frac(\calO_k^g)$.  By Proposition~\ref{P:Borger-cond-constr}, $\calO_k^g$ is a complete discrete valuation ring with perfect residue field.

We have a homomorphism of Galois groups $G_{k^g} \rar G_k$.  Given a representation $\rho$ of $G_k$ with finite image, we define the \emph{Borger's Artin conductor} $\Art_B(\rho)$ to be $\Art(\rho_{G_{k^g}})$, where the latter term is as in the classical definition \cite{BOOK-local-fields}.
\end{definition}

\begin{remark}
In \cite{Borger-conductor}, Borger only defined Artin conductors for representations of finite image.  We expect his definition can be extended to representations of finite local monodromy.  However, this additional freedom is not essential, so we stick to finite image case to ease the argument.
\end{remark}

Obviously, Borger's Artin conductors have a Hasse-Arf property naturally inherited from that of $k^g$, a complete discretely valued field with perfect residue field.

\begin{proposition}\label{P:Borger-cond-prop}
\emph{\cite[Theorem~A]{Borger-conductor}} Borger's Artin conductor $\Art_B(\rho)$ is a nonnegative integer and it coincides with the classical definition when the residue field $\kappa_k$ is perfect.

\emph{\cite[Proposition~2.3]{Borger-conductor}} Furthermore, $\Art_B(\rho)$ is unchanged after a finite unramified extension of $k$.
\end{proposition}

Moreover, Borger proved that his definition coincides with a variant of arithmetic Artin conductor $\Art_K$ for characters using the definition of Kato \cite{Kato-cond}.  (As we will not use Kato's definition, we just mention the following proposition as a fact.)

\begin{proposition}
\emph{\cite[Theorem~B]{Borger-conductor}} Let $\chi$ be a class in $H^1(G_k, \QQ / \ZZ)$ and $\chi'$ its image in $H^1(G_{k^g}, \QQ / \ZZ)$. Then $\Art_K(\chi) = \Art_K(\chi')$.  In particular, for a rank one representation $\rho$ of $G_k$ with finite image, $\Art_K(\rho) = \Art_B(\rho)$.
\end{proposition}

Borger gave the following explicit descriptions of $k^u$ and $k^g$.

\begin{proposition}\label{P:struc-ku-kg}
We have $k^u = \big(\kappa_k [v_{i,j}\,|\,j \in J, i \in \NN]\big)^\pf((\pi_{k^u}))$.  The homomorphism $k \inj k^u$ is determined by $s \mapsto \pi_{k^u}$ and $b_j \mapsto b_j + \sum_{i >0} v_{i, j}\pi_{k^u}^i$.  Also, $k^g = \Frac\big( \kappa_k [v_{i,j};j \in J, i \in \NN]^\pf\big) ((\pi_{k^u}))$ and the homomorphism $k \rar k^g$ is given by composing $k \inj k^u$ with the natural morphism $k^u \inj k^g$.
\end{proposition}

\subsection{Comparison with Borger's conductors}

The key to prove the comparison between Borger's Artin conductors and the arithmetic Artin conductors is to study how the arithmetic Artin conductors behave under the operations of adding generic $p^\infty$-th roots.

In this subsection, we do not impose any hypothesis on $k$.

\begin{proposition}\label{P:diff-inv-p-th-root}
Assume Hypothesis~\ref{H:J-finite-set}. For representations of finite images, the differential Artin conductor for a representation of finite image is unchanged after adding generic $p^\infty$-th roots.
\end{proposition}
\begin{proof}
Since the operation of adding $p^\infty$-th roots does not change the Galois group of the finite Galois extension, we may assume that the representation is irreducible and totally and wildly ramified.  Hence it suffices to consider the differential ramification break of a totally and wildly ramified finite Galois extension $l/k$.

Recall that we have a differential module $\calE$ over $Z_k^{\geq \eta_0} = A_K^1[\eta_0, 1)$ for some $\eta_0 \in (0,1)$ with differential operators $\partial_{B_J}$ and $\partial_S$, associated to the regular representation of $\Gal(l/k)$ over $\Qp$.  The base change $k \inj k' = k(x_J)^\wedge$ is translated into the base change of $\calE$ into $\calE'$, from $Z_k^{\geq \eta_0}$ to $Z_{k'}^{\geq \eta_0} = A_{K(X_J)^\wedge}^1[\eta_0, 1)$, where $K' = K(X_J)^\wedge$ is the completion of $K(X_J)$ with respect to the $(1, \dots, 1)$-Gauss norm;  $\calE'$ has differential operators $\partial_{B_J}$, $\partial_{X_J}$, and $\partial_S$.

Consider the rotation $f:Z_{k'}^{\geq \eta_0} \rar Z_{k'}^{\geq \eta_0}$ by $f^*(B_J) = B_J +X_JS$, $f^*(X_J) = X_J$, and $f^*(S) = S$;  write $\partial'_{B_J}$, $\partial'_{X_J}$, and $\partial'_S$ for the action of differential operators on $f^*\calE'$.  Then,
\[
\partial'_{B_J} = \partial_{B_J}, \qquad 
\partial'_{X_J} = S \cdot \partial_{B_J} + \partial_{X_J}, \qquad
\partial'_S = \sum_{j \in J} X_j \cdot \partial_{B_j} + \partial_S.
\]
Since $X_J$ are transcendental over $K$, we have
\begin{equation} \label{E:sp-norm-rotation}
\max\{|\partial_{B_J}|_{\calE_\eta, \sp}, |\partial_S|_{\calE_\eta, \sp}, |\partial_{X_J}|_{\calE_\eta, \sp}\} = |\partial'_S|_{\calE'_\eta, \sp} = \max\{|\partial'_{X_J}|_{\calE'_\eta, \sp}, |\partial'_S|_{\calE'_\eta, \sp}\},\quad \forall \eta \in [\eta_0, 1).
\end{equation}

Note that adding generic $p^\infty$-th roots to $k$, exactly corresponds to replacing $\calE$ by $f^*\calE'$ and \emph{forgetting} the differential operators $B_J$.  By \eqref{E:sp-norm-rotation} above, the differential non-logarithmic ramification break of $\tilde l / \tilde k$ is the same as that of $l/k$.
\end{proof}

\begin{theorem}\label{T:ar=diff=Borger}
For a complete discretely valued field $k$ of equal characteristic $p$ and a representation $\rho$ of its Galois group $G_k$ with finite image, the arithmetic Artin conductors $\Art_{ar}(\rho)$ as well as the differential Artin conductors $\Art_\dif(\rho)$ are the same as Borger's Artin conductors $\Art_B(\rho)$.
\end{theorem}
\begin{proof}
First we may assume that $\rho$ is irreducible and it factors exactly through the Galois group $G_{l/k}$ of a totally ramified Galois extension $l/k$ because all conductors are additive and remain the same under a (finite) unramified extension (Theorem~\ref{T:properties-Ked-cond}(2) and Propositions~\ref{P:AS-space-properties}(4) and \ref{P:Borger-cond-prop}).  As $k^g$ has a perfect residue field, $\Art_B(\rho) = \Art_B(\rho|_{G_{k^g}}) = \Art_\dif(\rho|_{G_{k^g}})$ are the same as in the classical definition.  It suffices to show $\Art_\dif(\rho) = \Art_\dif(\rho|_{G_{k^g}})$.

Similarly to the proof of Theorem~\ref{T:main-theorem}, one may add the $p^\infty$-th roots of all but finitely many elements of the $p$-basis into $k$ without changing the differential Artin conductors.  In other words, there exists $k \inj k_1 = k(b_j^{p^{-n}}| j \in J \backslash J_0, n \in \NN)^\wedge$ for some finite set $J_0 \subset J$, such that $\Art_\dif(\rho) = \Art_\dif(\rho|_{\Gal_{k_1}})$.  Since the residue field of $k^g$ is perfect, there exists $k_1 \inj k^g$ extending $k \inj k^g$.  Hence, we may assume Hypothesis~\ref{H:J-finite-set}, i.e., $k$ has a finite $p$-basis.

By Proposition~\ref{P:finite-gen-pth-roots}, we can do finitely many operations of adding generic $p^\infty$-th roots and make the resulting field extension $k_2 l / k_2$ not fiercely ramified and $\Art_\dif(\rho|_{G_{k_1}}) = \Art_\dif(\rho|_{G_{k_2}})$.  In order to link $k_2$ with $k^g$, we need to show that we have a homomorphism $k_2 \inj k^g$ extending $k_1 \inj k^g$, for which we return to the proof of Proposition~\ref{P:finite-gen-pth-roots} and construct the homomorphism step by step.

The $r$-th ($1\leq r \leq r_0$) step of adding generic $p^\infty$-th roots is to construct
$$
k^{(r)}_1 = \Big(k_1^{(r-1)}(x_{r, J})\big( (x_{r-1,j} + x_{r,j}\pi_k)^{1/p^n} ;\, j \in J, n\in \NN \big)\Big)^\wedge.
$$
where $x_{0,j} = b_j$ for $j \in J$ and $k_1^{(0)} = k_1$.  One checks that mapping
$$
x_{r, j} \mapsto \sum_{r'\geq r} v_{r', j} \pi_{k^g}^{r'-r},\quad \forall j \in J, r = \serie{}r_0;
$$
gives the desired homomorphism $k_2 \inj k^g$.

Now, $k_2l / k_2$ has na\"ive ramification degree $[k_2l: k_2]$, so $\calO_{k^gl} = \calO_{k^g} \otimes_{\calO_{k_2}} \calO_{k_2l}$.  Hence, $\Art_\dif(\rho|_{G_{k_2}}) = \Art_{ar}(\rho|_{G_{k_2}}) = \Art_{ar}(\rho|_{G_{k^g}}) = \Art_\dif(\rho|_{G_{k^g}})$ via Theorem~\ref{T:main-theorem} and Proposition~\ref{P:AS-space-properties}(4).  This finishes the proof.
\end{proof}

\bibliographystyle{plain}

\begin{thebibliography}{10}

\bibitem{AM-sous-groupes}
Ahmed Abbes and Abdellah Mokrane.
\newblock Sous-groupes canoniques et cycles \'evanescents {$p$}-adiques pour
  les vari\'et\'es ab\'eliennes.
\newblock {\em Publ. Math. Inst. Hautes \'Etudes Sci.}, (99):117--162, 2004.

\bibitem{AS-micro-local}
Ahmed Abbes and Takeshi Saito.
\newblock {A}nalyse micro-locale $l$-adique en caracteristique $p>0$: Le cas
  d'un trait, \textit{Publication of the Research Institute for Mathematical Sciences} 45-1 (2009) 25--74.

\bibitem{AS-cond1}
Ahmed Abbes and Takeshi Saito.
\newblock Ramification of local fields with imperfect residue fields.
\newblock {\em Amer. J. Math.}, 124(5):879--920, 2002.

\bibitem{AS-cond2}
Ahmed Abbes and Takeshi Saito.
\newblock Ramification of local fields with imperfect residue fields, {II}.
\newblock {\em Doc. Math.}, (Extra Vol.):5--72 (electronic), 2003.
\newblock Kazuya Kato's fiftieth birthday.

\bibitem{Berkovich-book}
Vladimir~G. Berkovich.
\newblock {\em Spectral theory and analytic geometry over non-{Archimedean}
  fields}, volume~33 of {\em Mathematical Surveys and Monographs}.
\newblock American Mathematical Society, Providence, RI, 1990.

\bibitem{Berthelot-rig-coh-part-I}
Pierre Berthelot.
\newblock {\em Cohomologie rigide et cohomologie rigide \`a support propre.
  Premi\`ere partie}.
\newblock IRMAR, 1996.

\bibitem{BBM-dieudonne-crisII}
Pierre Berthelot, Lawrence Breen, and William Messing.
\newblock {\em Th\'eorie de {D}ieudonn\'e cristalline. {II}}, volume 930 of
  {\em Lecture Notes in Mathematics}.
\newblock Springer-Verlag, Berlin, 1982.

\bibitem{Borger-conductor}
James~M. Borger.
\newblock Conductors and the moduli of residual perfection.
\newblock {\em Math. Ann.}, 329(1):1--30, 2004.

\bibitem{BGR}
S.~Bosch, U.~G{\"u}ntzer, and R.~Remmert.
\newblock {\em Non-{A}rchimedean analysis: a systematic approach to rigid
  analytic geometry}, volume 261 of {\em Grundlehren der Mathematischen
  Wissenschaften}.
\newblock Springer-Verlag, Berlin, 1984.

\bibitem{ChPu-cond}
Bruno Chiarellotto and Andrea Pulita.
\newblock {Arithmetic and Differential Swan Conductors of rank one
  representations with finite local monodromy}.
\newblock arXiv: \texttt{math.NT/0711.0701v1}.

\bibitem{Esbd-comm-alg}
David Eisenbud.
\newblock {\em Commutative algebra: with a view toward algebraic geometry},
  volume 150 of {\em Graduate Texts in Mathematics}.
\newblock Springer-Verlag, New York, 1995.

\bibitem{EGA-IV-1}
A.~Grothendieck.
\newblock \'{E}l\'ements de g\'eom\'etrie alg\'ebrique. {IV}. \'{E}tude locale
  des sch\'emas et des morphismes de sch\'emas. {I}.
\newblock {\em Inst. Hautes \'Etudes Sci. Publ. Math.}, (20):259, 1964.

\bibitem{Hattori-tame-char}
Shin Hattori.
\newblock Tame characters and ramification of finite flat group schemes.
\newblock {\em J. Number Theory}, 128(5):1091--1108, 2008.

\bibitem{Kato-cond}
Kazuya Kato.
\newblock Swan conductors for characters of degree one in the imperfect residue
  field case.
\newblock In {\em Algebraic $K$-theory and algebraic number theory (Honolulu,
  HI, 1987)}, volume~83 of {\em Contemp. Math.}, pages 101--131. Amer. Math.
  Soc., Providence, RI, 1989.

\bibitem{KSK-notes}
Kiran~S. Kedlaya.
\newblock $p$-adic differential equations.
\newblock \texttt{http://www-math.mit.edu/\textasciitilde kedlaya/papers/}.

\bibitem{KSK-overview}
Kiran~S. Kedlaya.
\newblock Local monodromy of {$p$}-adic differential equations: an overview.
\newblock {\em Int. J. Number Theory}, 1(1):109--154, 2005.

\bibitem{KSK-Swan1}
Kiran~S. Kedlaya.
\newblock Swan conductors for {$p$}-adic differential modules. {I}. {A} local
  construction.
\newblock {\em Algebra Number Theory}, 1(3):269--300, 2007.

\bibitem{Mats-Swan-cond}
Shigeki Matsuda.
\newblock On the {S}wan conductor in positive characteristic.
\newblock {\em Amer. J. Math.}, 119(4):705--739, 1997.

\bibitem{Mats-Conj-AS-fil}
Shigeki Matsuda.
\newblock Conjecture on {A}bbes-{S}aito filtration and {C}hristol-{M}ebkhout
  filtration.
\newblock In {\em Geometric aspects of Dwork theory. Vol. I, II}, pages
  845--856. Walter de Gruyter GmbH \& Co. KG, Berlin, 2004.

\bibitem{Saito-wild-ram}
Takeshi Saito.
\newblock {Wild ramification and the characteristic cycle of an $l$-adic
  sheaf}, \textit{Journal de l'Institut de Mathematiques de Jussieu}, (2009) 8(4), 769--829.

\bibitem{BOOK-local-fields}
Jean-Pierre Serre.
\newblock {\em Local fields}, volume~67 of {\em Graduate Texts in Mathematics}.
\newblock Springer-Verlag, New York, 1979.

\bibitem{Sweedler-insep-extn}
Moss~Eisenberg Sweedler.
\newblock Structure of inseparable extensions.
\newblock {\em Ann. of Math. (2)}, 87:401--410, 1968.

\bibitem{Tsuzuki-finite-monodromy-on-var}
Nobuo Tsuzuki.
\newblock Morphisms of {$F$}-isocrystals and the finite monodromy theorem for
  unit-root {$F$}-isocrystals.
\newblock {\em Duke Math. J.}, 111(3):385--418, 2002.

\bibitem{Whitney-Cohen-rings}
Wayne~A. Whitney.
\newblock {\em Functorial Cohen Rings}.
\newblock PhD thesis, University of California, Berkeley, 2002.

\end{thebibliography}

\end{document}